%% file: article.tex
\documentclass[onefignum,onetabnum]{siamonline250211}

\input{shared}

\ifpdf
\hypersetup{
  pdftitle={Learning firmly nonexpansive operators},
  pdfauthor={K. Bredies, J. Chirinos-Rodríguez, E. Naldi}
}
\fi

\begin{document}
\footnotetext[1]{Department of Mathematics and Scientific Computing, University of Graz, Graz, Austria (\email{kristian.bredies@uni-graz.at})}
\footnotetext[2]{Corresponding author. MaLGa, DIBRIS, Università degli Studi di Genova, Genoa, Italy (\email{jonathanchirinosrodriguez@gmail.com)}}
\footnotetext[3]{MaLGa, DIMA, Dipartimento di Eccellenza 2023-2027, Università degli Studi di Genova, Genoa, Italy (\email{emanuele.naldi@edu.unige.it})}.

\maketitle

\begin{abstract}
This paper proposes a data-driven approach for constructing firmly nonexpansive operators. We demonstrate its applicability in Plug-and-Play (PnP) methods, where classical algorithms such as Forward-Backward splitting, Chambolle--Pock primal-dual iteration, Douglas--Rachford iteration or alternating directions method of multipliers (ADMM), are modified by replacing one proximal map by a learned firmly nonexpansive operator. We provide sound mathematical background to the problem of learning such an operator via expected and empirical risk minimization. We prove that, as the number of training points increases, the empirical risk minimization problem converges (in the sense of Gamma-convergence) to the expected risk minimization problem. Further, we derive a solution strategy that ensures firmly nonexpansive and piecewise affine operators within the convex envelope of the training set. We show that this operator converges to the best empirical solution as the number of points in the envelope increases in an appropriate way. Finally, the experimental section details practical implementations of the method and presents an application in image denoising, where we consider a novel, interpretable PnP Chambolle--Pock primal-dual iteration.
\end{abstract}

\begin{keywords} Plug-and-Play methods, machine learning, firmly nonexpansive operators, Lipschitz functions, image denoising, proximal algorithms.
\end{keywords}

\begin{AMS}
65J20, 
65K10, 
46N10, 
52A05. 
\end{AMS}

\section{Introduction}
Solving an inverse problem consists in recovering a solution $y^*$ from measurements $x$, by knowing only the mapping that assigns to each possible solution the respective measured data. It is well known that reconstructing good solutions can become a particularly challenging task when the measurements contain noise and the problem is ill-posed. A classical example is the denoising problem, where the forward operator is equal to the identity, and a common  approach to tackle it is to find solutions of the following variational problem
\begin{equation}\label{eq:var}
\min_{y} \, f(y, x) +R(y),
\end{equation}
where $f$ is a \emph{data-fidelity} term such that $f(y, x)$ measures the discrepancy between $y$ and the observation $x$, and constrains the solution to remain close to the available measurements. The function $R$, called \emph{regularizer}, incorporates prior knowledge about the solution $y^*$ into the problem formulation. In this case, for instance, $R$ should be designed in such a way that the value of $R(y)$ is low when $y$ has low noise, and high otherwise. To this day, selecting an appropriate $R$ remains a challenging problem. 

The above strategy can be viewed as a \emph{model-based} technique, relying on a mathematical model with well-established properties. Variational methods have for a long time achieved state-of-the-art results \cite{burgerbenning,ChambollePock2016,CombettesPesquet2021,scher} in imaging problems. The design of refined regularization terms for solving image denoising problems (e.g. Total Variation \cite{chambtv,rof} or its further higher-order generalizations \cite{Bredies_2020,tgv,chanmarqmu}) contributed to achieve remarkable practical performances, while additionally providing robust theoretical guarantees. Notwithstanding this, \emph{data-driven} methodologies have gained significant attention in recent years, since they demonstrate improved performance in various practical scenarios while overcoming relevant challenges of classical methods (see \cite{Arridge2019} and references therein). The starting point of data-driven approaches is the assumption that a finite set of pairs of measurements and exact solutions $(\bx_1, \by_1),\dots, (\bx_n, \by_n)$, $n\in\mathbb{N}$, is available. This {\em training set} is then used to define, or refine, a regularization strategy to be applied to any future observation $\bx_{\text{new}}$, for which an exact solution is not known.

A recent data-driven approach for solving denoising problems is based on the following observation: if the noise present in the observations $x$ is Gaussian, then the MAP estimate is given by a solution of \cref{eq:var} using a quadratic data-fidelity term $f(y, x)=(1/2)\|y-x\|^2$ (see e.g. \cite{venkat2013}). Additionally, if the regularizer $R$ is proper, convex and lower semi-continuous, then the solution of \cref{eq:var} is unique, and is commonly termed as the proximal map of $R$; i.e., $y=\mathrm{prox}_R(x)$ for given data $x$. The latter consideration enlightens the link between proximal maps and the denoising problem, since, for a given input $x$, $\mathrm{prox}_R(x)$ outputs a denoised version of $x$. Consequently, it seems reasonable to construct instead a general operator $T^*$ functioning as a denoiser (e.g., $T^*\approx\mathrm{prox}_R$ for some unknown, convex, $R$) such that, given a measurement $x$, $T^*$ outputs a denoised version of $x$, $T^*(x)$, that is close, in some sense, to the ground truth $y^*$. The task of learning denoisers can then be further applied to solve more general inverse problems such as image deblurring or reconstruction problems in Magnetic Resonance Imaging (MRI) and Computed Tomography (CT), among others. The so-called Plug-and-Play (PnP) methods enable exactly this, since they aim to solve a general variational problem by substituting the proximal map in various splitting algorithms by a general denoiser. We mention here splitting methods such as the Forward-Backward splitting (FBS) algorithm \cite{ABS11,Bredies_FB,Rockafellar_FB,CW05,Gabay_FB,ggar,Tseng_FB}, the Alternating Directions Method of Multipliers (ADMM) \cite{Bot_ADMM,admm,Eckstein_ADMM,Bredies_2017}, the Douglas--Rachford splitting algorithm \cite{DR_original,EcksteinBertsekas_DR,DR_monotone,SVAITER2011}, the Chambolle--Pock primal-dual iteration (CP) \cite{Chambolle2011}, or other primal-dual approaches \cite{BrediesSun,combpesq,condat,OconnorVandenberghe}. This approach tackles two major challenges in traditional methods: it removes the need to select an appropriate regularizer $R$, which is often challenging, and it eliminates the need to compute the proximal operator of a convex function, which typically does not have a closed-form expression.

PnP approaches have demonstrated remarkable empirical performance across a diverse range of imaging applications \cite{buzz,dong,sree,teod,venkat2013} and, since then, several follow-up works have been published where different choices of denoisers and structures have been taken in consideration. Notably, the BM3D denoiser \cite{dabov,heide,sree}, employed in PnP methods \cite{pnpreview}, while ``handcrafted'' and, in a sense, also classical, is not a variational approach (i.e., it does not correspond to the proximal operator of a regularization functional). Likewise, denoising methods based on deep learning techniques have been considered more recently \cite{he,mein,zhang}. However, in this context, convergence guarantees for PnP methods are usually hard to ensure. Some works have already studied the theoretical properties of these methods. Among them, we mention \cite{chan}, where convergence with a denoiser which is assumed to satisfy certain asymptotic criteria, is proven, or \cite{sree,teod17}, where the authors prove convergence for PnP Forward-Backward splitting (PnP-FBS) and for PnP-ADMM under the assumption that the denoiser is nonexpansive. The authors of \cite{ryupnp2019} provide a convergence analysis for both PnP-FBS and PnP-ADMM under certain contractivity assumptions on the operator $D_\sigma-\Id$, where $D_\sigma$ is the denoiser. In order to recover more classical convergence guarantees, and motivated by the fact that the proximal operator of a convex function is a firmly nonexpansive operator, some authors have designed the denoiser as an artificial neural network that aims to be an averaged (or firmly nonexpansive) operator \cite{steidl2021,pescterr2021,celledoni23,Terris2020}. Authors in \cite{steidl2021} are, to the best of our knowledge, the first ones constructing  $\alpha$-averaged, firmly nonexpansive artificial neural networks. 

As we will discuss later, learning an averaged (or firmly nonexpansive) operator ultimately reduces to the challenge of learning a nonexpansive operator, a problem that has gained considerable attention in the past years. In fact, ensuring nonexpansivity, or more generally, enforcing a Lipschitz constraint on an artificial neural network, poses significant practical difficulties \cite{pescterr2021}. Additionally, it is well known that imposing strict constraints on the Lipschitz constant of a network's layer can negatively impact its expressive power \cite{Anil2019,Lanthaler2024,Tanielian2021,zhang2022}, even though recent works have shown that carefully selecting activation functions can improve expressivity also in this context \cite{ducotterd2024,Neumayer2023}. To avoid imposing strict constraints, the authors of \cite{pescterr2021} propose a stochastic penalization of the gradient norm of the network which, while not strictly guaranteeing the desired properties, has demonstrated promising practical performance. The above considerations raise the need for a complete theoretical analysis of the problem of learning (firmly) nonexpansive operators, giving the first motivation for our work.

Another significant motivation for our work arises from the inherent infinite-dimensional nature of the set of nonexpansive operators. As a result, a possible approach to tackle this issue in practice is to discretize this set. For example, methodologies relying on neural networks approximate this space using a vast array of parameters. Consequently, several pertinent questions arise: how close is the discretized space to its continuous version? Is it possible to provide any density results? Does the solution of the discretized problem converge to that of its continuous counterpart? For instance, in \cite{pescterr2021}, authors were able to prove that the set of their introduced operators is dense in a subclass of maximal monotone operators. To the best of our knowledge, the later work is the only one providing a density result in this flavor.

\subsection*{Contributions} In this paper, we focus on providing an analysis that incorporates all of the theoretical aspects mentioned above. The contributions of this work are described below:
\begin{itemize}
 \item In \cref{prelimns}, we provide a rigorous mathematical framework to address the problem of learning an operator through a constrained minimization problem. We first recall that the focus can be shifted from the space of firmly nonexpansive operators to the space of nonexpansive operators. We analyze the topological properties of this set, which can be characterized as a subset of the dual of a suitable Banach space, and thus naturally inherits a notion of weak$^*$ convergence  \cite{Guerrero2018, Weaver}. We also show that this pre-dual space is separable whenever the ambient space is separable.

 \item To construct nonexpansive operators, we adopt in \cref{sec:statmodel} a supervised learning approach \cite{cusma}. We fix $(\bar X, \bar Y)$ to be a pair of random variables and we aim at finding $N^*$, the minimizer of the expected risk, within the set $\mathcal{N}$ of nonexpansive operators,
 $$
 N^*\in\argmin_{N\in\mathcal{N}}F(N), \quad F(N):=\mathbb{E}[\|N(\bar X)-\bar Y\|^2].
 $$
By simply assuming that the pair $(\bar X, \bar Y)$ satisfies $\mathbb{E}[\|\bar X\|^2+\|\bar Y\|^2]<\infty$, we show that $N^*$ exists. This will be proved by combining both \cref{thm:existence} and \cref{prop:experr}. However, there is no access to the exact distribution of the pair $(\bar X, \bar Y)$, but to a finite set $(\bar X_1,\bar Y_1), \ldots, (\bX_n, \bY_n)$, $n\in\mathbb{N}$, of independent and identically distributed copies of $(\bar X, \bar Y)$. Then, the nonexpansive operator that approximates $N^*$ is constructed as a minimizer of the associated empirical risk minimization (ERM) problem over the space of nonexpansive operators:
\begin{equation}\label{eq:nonexERM}
N^*_n\in\argmin_{N\in \N}F_n(N), \quad F_n(N):=\frac{1}{n}\sum_{i=1}^n\|N(\bar X_i)-\bar Y_i\|^2.
\end{equation}
In the context of image denoising, in which the random variable $\bar X$ denotes noisy images and $\bar Y$ clean images, $N^*$ is the best nonexpansive operator performing the denoising task for such a given training set, and $N^*_n$ its discrete approximation. Naturally, it is possible to utilize $N^*_n$ for denoising future measurements that are similar to those seen during training. However, our objective is to utilize such operator for constructing the denoiser that will be further plugged into a Plug-and-Play method. We point out that, up to this point, we have followed a classical supervised learning approach \cite{cusma}, and the novelty consists in using this methodology within the particular context of (firmly) nonexpansive operators.

\item In the aforementioned setting, we prove that, with probability $1$, $F_n$ in \cref{eq:nonexERM} $\Gamma$-converges to the expected risk $F$ as the number of points in the training set goes to infinity. This result, presented in \cref{thm:gammaconv}, holds significant importance as it implies, in particular, that, with probability $1$, if a sequence of minimizers of the ERM, for every $n\in\mathbb{N}$, is considered, then (up to subsequences) it converges to a minimizer $N^*$ of the expected risk, see \cref{cor:convofminimizers}.

\item As previously mentioned, the class of nonexpansive operators is infinite-dimensional, and a further discretization is required. To address this, we propose a discretization for such a space via piecewise affine functions uniquely determined by simplicial partitions, or triangulations, of the underlying space. This setting yields two main consequences:
\begin{enumerate}[label=(\roman*),leftmargin=*]
\item By making use of standard numerical analysis techniques, we propose in \cref{sec:piecewise} a constructive approach for designing an operator which will be ensured to be piecewise affine and nonexpansive.
\item We establish in \cref{sec:density} a density result, showing that the class of considered approximating operators is actually large enough: we show that successively finer triangulations lead to a closer approximation to $\widehat N$, minimizer of \cref{eq:nonexERM}, see \cref{thm:density}.
\end{enumerate}
\item In \cref{exps} we present our numerical simulations, in which we precisely detail and illustrate how to practically implement the piecewise affine methodology described in the previous point. In particular, and as such approach is known to be computationally intractable in high dimensions, we limit our experimental setup to a a low-dimensional setting. We point out that the main objective of this section is not to present a state-of-the-art denoiser, but to illustrate our constructive approach, showing that it can work well in practice. We provide two main contributions: 
\begin{enumerate}[label=(\roman*),leftmargin=*]
\item In \cref{sec:pnpcp}, we introduce and analyze a novel PnP Chambolle--Pock primal–dual iteration that makes use of the so-called Moreau's identity, thereby providing a better interpretability of the proposed method. Additionally, while also enjoying strong theoretical guarantees, we show that our PnP Chambolle--Pock algorithm can perform well in practice.
\item Finally, we evaluate our proposed method in imaging applications, particularly addressing the problem of image denoising. In particular, we show that the bad scalability of simplicial partitions can be tackled by noticing that, in many applications, the proximity operator that we aim to learn is separable, and hence it can be treated in a low-dimensional setting. We analyze in depth our learned denoiser and compare it with variational approaches based on classic regularizers.
\end{enumerate}
\end{itemize}
In summary, we believe that this work addresses several questions previously raised in the literature by establishing a robust theoretical framework to the problem of learning (firmly) nonexpansive operators. This will be further complemented with a constructive, practical approach that aligns with and is supported by the theoretical findings.
\section{Preliminaries}\label{prelimns}
\subsection{Notation and definitions} \label{sec:defs}
Throughout the paper, $(\cX, \langle\cdot,\cdot\rangle)$ denotes a real separable Hilbert space. We denote by $\|\cdot\|$ the norm induced by the inner product $\langle\cdot,\cdot\rangle$ in $\cX$. In the context of convex analysis, $\Gamma_0(\cX)$ denotes the space of proper, convex, and lower semi-continuous functions $f: X\to(-\infty, +\infty]$. Moreover, we recall that given a set $C\subset \cX$, its \emph{convex envelope} $\conv (C)$ is the smallest convex subset of $\cX$ containing $C$. We define the \emph{indicator} function of $C$ as follows
$$
\iota_C(x)=
\begin{cases}
0 & \text{if } x\in C,\\
+\infty & \text{else},
\end{cases}
$$
and the \emph{characteristic} function of any set $C\subset \cX$ as
$$
\Chi_C(x)=
\begin{cases}
1 & \text{if } x\in C,\\
0 & \text{else}.
\end{cases}
$$
We define now the main object of study of this paper. An operator $T:\cX\to\cX$ is said to be \emph{firmly nonexpansive} if
\[\|T(x)-T(x')\|^2+\|(\Id-T)(x)-(\Id-T)(x')\|^2\leq\|x-x'\|^2, \text{ for every } x,x'\in\cX,\]
where $\Id$ denotes the identity map on $\cX$. We say that $T$ is
\emph{nonexpansive} if
\[\|T(x)-T(x')\|\leq\|x-x'\|, \text{ for every } x,x'\in\cX.\]
A (set-valued) operator $A:\cX\to 2^\cX$ is \emph{monotone} if
\[\langle u-u',x-x'\rangle \geq 0, \text{ for every } (x,u),(x',u') \in \gra A,\]
where $\gra A$ denotes the graph of the mapping $A$. The operator $A$ is said to be \emph{maximal monotone} if it is monotone and, for a monotone operator $A'$ such that $\gra A \subset \gra A'$, it follows that $A=A'$. The \emph{resolvent operator} of $A$ is defined as 
$$
J_A:=(\Id+A)^{-1}.
$$
An important class of maximal monotone operators is the class of subdifferentials of proper, convex and lower semi-continuous functions. It is well-known \cite{BCombettes} that the proximal operator of $R$ is characterized as the resolvent of its subdifferential. For this reason, we report below some important properties of the class of resolvents of maximal monotone operators.

\subsection{A useful characterization}\label{sec:charac}
Since we want to study the structure of the set of resolvent operators,
we first introduce the following set
\begin{equation}\label{eq:fnespace}
\M:=\left\{T:\cX\to\cX \, : \, T=J_{A}, \text{ where } A \text{ is maximal monotone}\right\}.
\end{equation}
An element $T\in\M$ is not necessarily a proximal operator, but every proximal operator of a function in $\Gamma_0(\cX)$ is contained in $\M$. For completeness, we recall the following result taken from \cite[Corollary 23.9]{BCombettes}. 
\begin{lemma}\label{lem:fneequivmm}
The mapping $T:\cX\to\cX$ is the resolvent of a maximal monotone operator if and only if $T$ is firmly nonexpansive.
\end{lemma}
Moreover, as already pointed out in \cite{pescterr2021} for this setting, given $T:\cX\to\cX$ a firmly nonexpansive operator, there exists a nonexpansive operator $N:\cX\to\cX$ such that $T=(1/2)(\Id+N)$. Actually, these conditions are equivalent \cite[Proposition 4.4]{BCombettes}. In fact, an operator $T$ can be expressed by $T=(1/2)(\Id+N)$, for some nonexpansive operator $N:\cX\to\cX$, if and only if $T$ is firmly nonexpansive (we refer to \cite{BCombettes} for further details). We can therefore reduce our study to the space of nonexpansive operators, i.e., the space
\begin{equation}\label{eq:nonexpspace}
\mathcal{N}:=\left\{N\colon \cX\to\cX \, : \, N \text{ is nonexpansive}\right\}.
\end{equation}
For further purposes, we also denote by $\mathcal{N}_0$ the space of nonexpansive operators that vanish at zero. We recall that nonexpansive operators coincide with $1$-Lipschitz operators from $\cX$ to $\cX$. This fact motivates us to study the space of Lipschitz operators and its topological structure.

\subsection{\texorpdfstring{The spaces $\Lip_0(\cX, \cX)$ and $\Lip(\cX, \cX)$}{The spaces Lip0(X) and Lip(X)}} Denote by $\Lip(\cX, \cX)$ the space of $\cX$-valued Lipschitz operators $T\colon \cX\to\cX$ and by $\Lip_0(\cX, \cX)$ the space of $\cX$-valued Lipschitz operators that vanish at $0$. Given $T\in\Lip(\cX, \cX)$, define the quantity
$$
\cL(T):=\sup_{x\neq y} \frac{\|T(x)-T(y)\|}{\|x-y\|},
$$
which corresponds to the smallest Lipschitz constant of $T$. Then, $\Lip_0(\cX, \cX)$ is a Banach space with norm $\|\cdot\|_{\Lip_0}:=\cL$, see e.g.~\cite{Guerrero2018} (or \cite{Weaver} for the scalar case). Recall that, to be consistent with the rest of the paper, we denote by $\mathcal{N}_0$ the closed unit ball of this space. Moreover, the space of Lipschitz operators $\Lip(\cX, \cX)$, endowed with the norm
$$
\|T\|_{\Lip}:=\|T(0)\|+\cL(T-T(0)),
$$
is also a Banach space. Such property can also be derived from the following result.
\begin{proposition}\label{iota}
$\Lip(\cX, \cX)$ is isometrically isomorphic to $\mathcal{Y}=\cX\times\Lip_0(\cX, \cX)$; i.e.,
$$
\Lip(\cX, \cX)\cong \cX\times \Lip_0(\cX, \cX),
$$
where the right-hand side is a Banach space with the norm
$$
\|(x,T_0)\|_{\mathcal{Y}}:= \|x\|+\cL(T_0),
$$
for every $x\in\cX$, $T_0\in \Lip_0(\cX, \cX)$.
\end{proposition}
\begin{proof} Define the following linear mapping
$$
\varphi\colon\Lip(\cX, \cX)\rightarrow \cX\times\Lip_0(\cX, \cX); \quad T\mapsto (T(0), T-T(0)).
$$
First, observe that $\varphi$ is an isometry:
$$
\|\varphi(T)\|_{\mathcal{Y}}=\|T(0)\|+\cL(T-T(0))=\|T\|_{\Lip}.
$$
Moreover, $\varphi$ is bijective. If we take $T\in\Lip(\cX, \cX)$ such that
$$
(T(0),T-T(0))=(0,0),
$$
we first get that $T(0)=0$, and by using this fact in the second component, we obtain that $T=0$. Thus, $\varphi$ is injective. Finally, let $(x,T_0)\in \cX\times\Lip_0(\cX, \cX)$. Choosing $T=x+T_0\in\Lip(\cX, \cX)$ and noting that $T(0)=x+T_0(0)=x$, we see that $\varphi(T)=(x, T_0)$. Hence, $\varphi$ is surjective and the result has been proven.
\end{proof}

Due to this result, we can identify every $T\in\Lip(\cX,\cX)$ as $T\equiv (T(0),T-T(0))$. As $\Lip(\cX, \cX)$ is a normed space, it naturally inherits the strong topology associated to the norm $\|\cdot\|_{\Lip}$. However, we next aim to introduce and characterize a weaker topology for $\Lip(\cX, \cX)$, $\tau_{\Lip}$, that will be used throughout the paper. We will see that the space $\Lip(\cX, \cX)$, or more specifically $\Lip_0(\cX, \cX)$, is a dual space, and so, it has an associated weak$^*$ topology. When considering Lipschitz spaces of scalar functions $\Lip_0(\cX,\R)$, the existence of a pre-dual is well-known, and in that case, the pre-dual space (in the sense of a strongly unique pre-dual space) of $\Lip_0(\cX,\R)$, often denoted by $\mathcal{F}(\cX, \R)$, is known as Lipschitz-free \cite{Cobzas2019, Godefroy03}, Arens--Eels \cite{Arens56,Weaver}, or Kantorovich--Rubinstein space \cite{Hanin92, Kantorovich82}. Guerrero, López--Pérez and Rueda Zoca mentioned in \cite{Guerrero2018} the construction of a space $\mathcal{F}(\cX, \cX)$ that they consider as a pre-dual of $\Lip_0(\cX, \cX)$; i.e., such that
\begin{equation}\label{eq:predual_Lip0}
\Lip_0(\cX, \cX)\cong \Lip_0(\cX, \cX^*) \cong \mathcal{F}(\cX, \cX)^*.
\end{equation}
To the best of our knowledge, the work \cite{Guerrero2018} is the first to introduce a pre-dual for the general setting of $\Lip_0(\cX, \cX)$. However, a proof of \cref{eq:predual_Lip0} is not given there. For completeness, we detail here the construction and the needed proofs. We will follow \cite[Section 8.2.2]{Cobzas2019} and adapt the proof of \cite[Theorem 8.2.7]{Cobzas2019} to the vector-valued case. We will not address the question of whether this pre-dual is a (strongly) unique pre-dual.

For any $x,z\in\cX$, we define as in \cite{Guerrero2018} the linear operators
\[
\delta_{x,z} \colon \Lip_0(\cX, \cX) \to \R \ ; \ T \mapsto \langle T(x),z\rangle,
\]
and we note that they are also continuous since, for any $T\in \Lip_0(\cX,\cX)$, it holds $\delta_{x,z}(T) = \langle T(x),z\rangle \leq \mathcal{L}(T) \|x\|\|z\|$. Next, define
\begin{equation}\label{eq:setM}
M := \text{span}\{\delta_{x,z} \colon x, z \in\cX\}\subset \Lip_0(\cX,\cX)^*.
\end{equation}
In \cite{Guerrero2018}, the authors consider the completion of $M$ with respect to the topology of $\Lip_0(\cX, \cX)^*$ as a natural candidate to be a pre-dual of $\Lip_0(\cX, \cX)$; i.e., they define
$$
\mathcal{F}(\cX, \cX):= \overline{\text{span}}\{\delta_{x,z} \colon x, z \in\cX\},
$$
and state that the space above is a pre-dual of $\Lip_0(\cX, \cX)$. We prove this in the following result.

\begin{theorem}\label{thm:predual}
   The space $\Lip_0(\cX,\cX)$ is isometrically isomorphic to $\mathcal{F}(\cX,\cX)^*$.
\end{theorem}
\begin{proof}
To prove the result, we consider the map $\Lambda \colon \Lip_0(\cX,\cX) \to \mathcal{F}(\cX,\cX)^*$ that is determined by the condition
    \begin{equation}\label{eq:equation_in_proof}
    (\Lambda T) \left(\sum_{i=1}^{n} a_i \delta_{x_i,z_i}\right) = \sum_{i=1}^n a_i \langle T(x_i), z_i \rangle,
    \end{equation}
    for $\sum_{i=1}^{n} a_i \delta_{x_i,z_i} \in M$, where $M$ given by \cref{eq:setM}. We will show that  \eqref{eq:equation_in_proof} uniquely determines $\Lambda$ on $\Lip_0(\cX, \cX)$ and that $\Lambda$ is an isometric isomorphism. Clearly, for every $T \in \Lip_0(\cX,\cX)$ it holds that $\Lambda T$ is linear on $M$. Moreover, for $T\neq 0$ we have 
\begin{equation*}
    \begin{aligned}
    \frac{1}{\mathcal{L}(T)}\bigg|(\Lambda T)\bigg(\sum_{i=1}^n a_i \delta_{x_i,z_i}\bigg)\bigg|\leq  \sup_{T \in \mathcal{N}_0} \bigg| \sum_{i=1}^n a_i\left\langle T(x_i), z_i\right\rangle\bigg| = \bigg\|\sum_{i=1}^n a_i \delta_{x_i, z_i} \bigg\|_{\mathcal{F}(\cX,\cX)},
    \end{aligned}
\end{equation*}
so that $|(\Lambda T)(\xi)| \leq \mathcal{L}(T) \|\xi\|_{\mathcal{F}(\cX,\cX)}$ for all $\xi\in M$ (note that $\|\,\cdot\,\|_{\mathcal{F}(\cX,\cX)}$ is induced by the norm on $\Lip_0(\cX, \cX)^*$). Because of this property and the extension theorem for linear, continuous operators on subspaces of Banach spaces, the operators $\Lambda T$ are uniquely determined by \eqref{eq:equation_in_proof} as elements in $\mathcal{F}(\cX,\cX)^*$ (by density). Moreover, we have $\|\Lambda\| \leq 1$. Next, for all $\phi \in \mathcal{F}(\cX,\cX)^*$, we define the map
$$
\Gamma \colon \mathcal{F}(\cX,\cX)^*\to \Lip_0(\cX,\cX); \quad \langle\Gamma \phi (x), z \rangle := \phi(\delta_{x,z}) \quad \text{for all } x, z \in \cX.
$$ 
Notice that the above condition uniquely defines the function $\Gamma \phi: \cX\to \cX$ for all $\phi \in \mathcal{F}(\cX,\cX)^*$. We finally have that indeed, $\Gamma\Phi \in \Lip_0(\cX, \cX)$ since
\begin{equation*}
    \begin{aligned}
        |\langle (\Gamma \phi)(0), z\rangle| = |\phi (\delta_{0,z})| & \leq \|\phi\|_{\mathcal{F}(\cX,\cX)^*}\| \delta_{0,z}\|_{\mathcal{F}(\cX,\cX)} \\ & = \|\phi\|_{\mathcal{F}(\cX,\cX)^*}\sup_{T\in \mathcal{N}_0}\langle T(0),z\rangle = 0,
    \end{aligned}
\end{equation*}
for all $z \in \cX$ implying $\Gamma\Phi(0)=0$ and for $x$, $y\in\cX$,
\begin{equation*}
    \begin{aligned}
\|(\Gamma \phi)(x)-(\Gamma \phi)(y)\| & = \sup_{\|z\|=1} |\langle (\Gamma \phi)(x)-(\Gamma \phi) (y), z\rangle| = \sup_{\|z\|=1} |\phi(\delta_{x,z}-\delta_{y,z})|\\
& \leq \sup_{\|z\|=1}\|\phi\|_{\mathcal{F}(\cX,\cX)^*}\|\delta_{x,z}-\delta_{y,z}\|_{\mathcal{F}(\cX,\cX)} \\ & = \sup_{\|z\|=1}\|\phi\|_{\mathcal{F}(\cX,\cX)^*} \sup_{T\in \mathcal{N}_0}\langle T(x)-T(y), z \rangle \\ & \leq \|\phi\|_{\mathcal{F}(\cX,\cX)^*} \|x-y\|.
    \end{aligned}
\end{equation*}
Hence, $\Gamma \phi \in \Lip_0(\cX,\cX)$ and $\|\Gamma\|\leq 1$. Given $T\in \Lip_0(\cX,\cX)$, we have
\[\langle (\Gamma \Lambda T) (x), z\rangle =(\Lambda T)(\delta_{x,z})=\langle T(x), z \rangle \quad \text{for all } x, z \in \cX,\]
so that $\Gamma \Lambda T = T$ and $\Gamma\Lambda=\Id_{\Lip_0(\cX,\cX)}$. On the other hand, given $\phi \in \mathcal{F}(\cX,\cX)^*$, we have
\[\begin{aligned}
    (\Lambda \Gamma \phi) \left(\sum_{i=1}^na_i\delta_{x_i,z_i}\right) & = \sum_{i=1}^n a_i \langle \Gamma \phi (x_i), z_i\rangle = \sum_{i=1}^n a_i \phi(\delta_{x_i,z_i}) \\
    &= \phi \left(\sum_{i=1}^na_i\delta_{x_i,z_i}\right),
    \end{aligned}\]
for every $\sum_{i=1}^{n} a_i \delta_{x_i,z_i} \in M$. So that, by density of $M$, it must be $\Lambda \Gamma\phi=\phi$ and $\Lambda \Gamma = \Id_{\mathcal{F}(\cX,\cX)^*}$, which together with $\|\Lambda\| \leq 1$ and $\|\Gamma\| \leq 1$ concludes the proof.
\end{proof}

Combining \cref{thm:predual} with \cref{iota}, we get that 
$$
\left(\cX\times\mathcal{F}(\cX, \cX)\right)^*=\cX^*\times\Lip_0(\cX, \cX)=\cX\times\Lip_0(\cX, \cX)\cong \Lip(\cX, \cX),
$$
so $\cX\times\mathcal{F}(\cX, \cX)$ is indeed a predual space of $\Lip(\cX, \cX)$.
Next, we will show that, since $\cX$ is separable, $\mathcal{F}(\cX, \cX)$ is also separable. The case in which $\cX$ is finite-dimensional is well known (it follows for example from \cite[Theorem 3.14]{Weaver}). For completeness, we provide a proof for the infinite-dimensional case.
\begin{theorem}
    \label{thm:separability}
    The space $\mathcal{F}(\cX, \cX)$ is separable.
\end{theorem}
\begin{proof}
Let $\widetilde{\cX}$ be a dense countable set in $\cX$. Given an element $\delta_{x,z}$ with $x,z \in\cX$, for any $\bar {\varepsilon} > 0$ there are elements $\bar{x}, \bar{z} \in \widetilde{\cX}$ such that $\|\bar x - x\|\leq \bar \varepsilon$ and $\|\bar z - z\| \leq \bar \varepsilon$. Thus,
\begin{equation*}
\begin{aligned}
    \|\delta_{x,z}-\delta_{\bar x,\bar z}\|_{\mathcal{F}(\cX,\cX)} & = \sup_{T \in \N_0} |\langle T(x), z \rangle - \langle T(\bar x), \bar z \rangle| \\ 
    &\leq \sup_{T\in \N_0} |\langle T(x), z \rangle - \langle T( x), \bar z \rangle| + |\langle T(x), \bar z \rangle - \langle T(\bar x), \bar z \rangle| \\
    & \leq \sup_{T\in \N_0}  \| T(x) \| \|z-\bar z\| + \|\bar z\| \|T(x) - T(\bar x)\| \\
    & \leq \|x\| \|z-\bar z\| + (\| z\| + \|z-\bar z\|) \| x- \bar x \| \leq \bar \varepsilon \|x\| + \bar \varepsilon \|z\| + \bar{\varepsilon}^2.
    \end{aligned}
\end{equation*}
Let now $\xi \in \mathcal{F}(\cX, \cX)$. By definition of $\mathcal{F}(\cX, \cX)$, for every $\varepsilon>0 $, there exists $n \in \mathbb{N}$, points $x_i,z_i\in\cX$, and coefficients $a_i\in \R$, $i=1,\dots,n$, such that
$$
\left\|\sum_{i=1}^n a_i\delta_{x_i,z_i}-\xi\right\|_{\mathcal{F}(\cX,\cX)} \leq \frac\varepsilon 2,
$$
and, without loss of generality, we suppose that $a_i\neq 0$, $i=1,\dots,n$. Actually, we can find the same estimate restricting ourselves to $a_i\in \mathbb{Q}$, $i=1,\dots,n$. Setting $\varepsilon_i > 0$ such that $\varepsilon_i \|x_i\| + \varepsilon_i \|z_i\| + \varepsilon_i^2 \leq \frac{\varepsilon}{2n |a_i|}$, $i=1,\dots, n$, there exist $\bar x_i, \bar z_i \in \widetilde{\cX}$, $i=1,\dots,n$, such that 
$$
\|\delta_{x_i,z_i}-\delta_{\bar x_i,\bar z_i}\|_{\mathcal{F}(\cX,\cX)} \leq \varepsilon_i \|x_i\| + \varepsilon_i \|z_i\| + \varepsilon_i^2 \leq \frac{\varepsilon}{2n |a_i|}.
$$
Thus, \[\bigg\|\sum_{i=1}^n a_i\delta_{\bar x_i,\bar z_i}-\xi\bigg\|_{\mathcal{F}(\cX,\cX)}\leq \sum_{i=1}^n |a_i| \|\delta_{x_i,z_i}-\delta_{\bar x_i,\bar z_i}\|_{\mathcal{F}(\cX,\cX)} + \bigg\|\sum_{i=1}^n a_i\delta_{x_i,z_i}-\xi\bigg\|_{\mathcal{F}(\cX,\cX)}  \leq \varepsilon.\]
In this way, we have proven that the countable set $\text{span}_{\mathbb{Q}}\{\delta_{\bar x,\bar z} \colon \bar x \in \widetilde{X}, \bar z \in \widetilde{\cX}\}$ is dense in $\mathcal{F}(\cX, \cX)$.
\end{proof}
With the above results, we have shown that the space $\Lip(\cX, \cX)$ is indeed a dual space. Using this, we can consider on $\Lip(\cX, \cX)$ the associated weak$^*$ topology that we denote by $\tau_{\Lip}$. By \cref{iota}, this topology can be further identified with the product topology of $\tau_{X}$ and $\tau_{\Lip_0}$, where $\tau_{\cX}$ is the weak topology on $\cX$ and $\tau_{\Lip_0}$ is the weak$^*$ topology on $\Lip_0(\cX, \cX)$. In particular, $\Lip_0(\cX, \cX)$ is a weak$^*$-closed subspace of $\Lip(\cX, \cX)$. Further, since $\mathcal{F}(\cX, \cX)$ is separable, every bounded set of $(\Lip_0(\cX, \cX),\tau_{\Lip_0})$ is first countable. This implies that its topology can be completely characterized by its convergent sequences. It follows from the characterization stated in \cref{thm:predual} that, similarly to the scalar case \cite{Weaver}, $\tau_{\Lip_0}$ corresponds to the topology of pointwise weak convergence on bounded sets. In fact, the following result hods true.
\begin{lemma} A sequence $(T_k)_{k\in\mathbb{N}}$ in $\Lip_0(\cX, \cX)$ is weak$^*$ converging to $T\in\Lip_0(\cX, \cX)$ if and only if $(T_k)_{k\in\mathbb{N}}$ is bounded (i.e., there exists $C>0$ such that $\mathcal{L}(T_k)\leq C$ for every $k\in\mathbb{N}$) and $T_k(x)\rightharpoonup T(x)$ as $k\to\infty$ for every $x\in\cX$.
\end{lemma}
\begin{proof}
On the one hand, if $T_k\weakk T$ as $k\to\infty$, then it is bounded and for all $x$, $z\in \cX$ we have $\langle T_k(x), z\rangle = \langle\delta_{x, z}, T_k\rangle\to \langle \delta_{x, z}, T\rangle = \langle T(x), z\rangle$ as $k\to\infty$. On the other hand, $T_k(x) \rightharpoonup T(x)$ for each $x\in \cX$ implies that $\langle \xi', T_k\rangle \to \langle \xi', T\rangle$ for all $\xi'\in M$ as $k\to\infty$ (recall that $M$ was defined in \cref{eq:setM}). Finally, if $\xi\in\mathcal{F}(\cX, \cX)$, there is $\xi'\in M$ such that 
$$
(C + \|T\|_{\Lip})\|\xi-\xi'\|_{\mathcal{F}(\cX,\cX)}<\frac{\varepsilon}{2},
$$
and $k_0\in\mathbb{N}$ such that $|\langle \xi', T_k-T\rangle|<\varepsilon/2$ for all $k\geq k_0$. Then, 
$$
|\langle \xi, T_k-T\rangle|\leq |\langle \xi', T_k-T\rangle|+(C + \|T\|_{\Lip})\|\xi-\xi'\|_{\mathcal{F}(\cX,\cX)}<\varepsilon
$$ 
for all $k\geq k_0$, proving the result. 
\end{proof}
Notice that the convergence at each point is weak and not strong. Combining the latter Lemma with \cref{iota}, \cref{thm:predual}, and \cref{thm:separability}, we characterize the weak$^*$ topology in $\Lip(\cX, \cX)$ in the following corollary.
\begin{corollary}\label{topo}
Let $(T_k)_{k\in\mathbb{N}}$ be a sequence in $\Lip(\cX, \cX)$ and $T\in\Lip(\cX, \cX)$. Then, $T_k\weakk T$ in $\Lip(\cX, \cX)$ if and only if $(T_k)_{k\in\mathbb{N}}$ is bounded in $\Lip(\cX, \cX)$ and $T_k(x)\rightharpoonup T(x)$ as $k \to \infty$ for every $x\in\cX$.
\end{corollary}

\begin{proof}
By the identification $\Lip(\cX, \cX) \cong  \cX \times \Lip_0(\cX, \cX)$, $T_k\weakk T$ as $k \to \infty$ means $T_k(0) \rightharpoonup T(0)$ and $T_k - T_k(0) \weakk T - T(0)$ in $\Lip_0(\cX, \cX)$ as $k \to \infty$. So, $(T_k(0))_{k\in\mathbb{N}}$ and $(T_k - T_k(0))_{k\in\mathbb{N}}$ are bounded in $\cX \times \Lip_0(\cX, \cX) \cong \Lip(\cX, \cX)$ and $T_k(x) = T_k(x) - T_k(0) + T_k(0) \rightharpoonup T(x) - T(0) + T(0) = T(x)$ as $k \to \infty$ for each $x \in \cX$. Conversely, if $(T_k)_{k\in\mathbb{N}}$ is bounded in $\Lip(\cX, \cX)$ with $T_k(x) \rightharpoonup T(x)$ as $k \to \infty$ for all $x \in \cX$, then $T_k(0) \rightharpoonup T(0)$ as $k\to\infty$ and $(T_k - T_k(0))_{k\in\mathbb{N}}$ is bounded in $\Lip_0(\cX, \cX)$ with $T_k(x) - T_k(0) \rightharpoonup T(x) - T(0)$ for all $x \in \cX$. Hence, $T_k(x) - T_k(0) \weakk T(x) - T(0)$ as $k \to \infty$ for each $x \in \cX$, so $T_k \weakk T$ in $\Lip(\cX, \cX)$.
\end{proof}
As we will see in the forthcoming sections, the above results are key in our analysis, since they provide the correct tools for proving good structural results about the set of nonexpansive operators $\N$ defined in \cref{eq:nonexpspace}.

\subsection{\texorpdfstring{Properties of $\N$}{Properties of N}}
The next step is to study the topological properties of~$\mathcal{N}$, which will follow directly from the analysis above. To do so, we recall that nonexpansive operators are $1$-Lipschitz operators, i.e., a subset of $\Lip(\cX, \cX)$. As we did in \cref{iota}, it is easy to prove that
$$
\mathcal{N}\cong \cX\times\N_0.
$$
Note that $\N_0$, as a subset of $\Lip_0(\cX, \cX)$, is bounded. Therefore, its topology is the weak topology of pointwise convergence as \cref{topo} is valid for $\mathcal{N}$: if $(N_k)_{k\in\mathbb{N}}$ is a sequence in $\N$ and $N\in\N$, then $N_k\weakk N$ if and only if $(N_k)_{k\in\mathbb{N}}$ is bounded in $\Lip(\cX, \cX)$ and $N_k(x)\rightharpoonup N(x)$ as $k\to\infty$ for each $x\in\cX$. The main structural result about the class $\N$ reads as follows.
\begin{corollary}\label{nnzero}
The set $\N_0$ is non-empty, convex and weak$^*$-compact. Moreover, $\N$ is weak$^*$-closed and convex.
\end{corollary}
\begin{proof}
First, note that the identity mapping belongs to $\N_0$, so $\N_0$ is non-empty. Next, as $\mathcal{N}_0$ is the closed unit ball of $\Lip_0(\cX, \cX)$, it is weak$^*$-compact by the Banach--Alaoglu--Bourbaki theorem \cite[Theorem 1.4.18]{Cobzas2019}. Then, it is weak$^*$-closed and trivially convex (see e.g. \cite[Corollary 2.10.9]{megg}). Finally, since the topology $\tau_{\Lip}$ is the product topology of $\tau_{\cX}$ and $\tau_{\Lip_0}$, $\N\cong\cX\times\N_0$ is weak$^*$-closed and convex as the product of the $\tau_{\cX}$-closed convex set $\cX$ and the $\tau_{\Lip_0}$-closed convex set $\N_0$.
\end{proof}
We now have all the necessary tools to analyze the problem of learning firmly nonexpansive operators.

\section{Learning firmly nonexpansive operators}\label{sec3}
We turn now our attention to the main problem, which is the one of constructing, in a feasible way, nonexpansive operators approximating some given data. First, we state the minimization problem in the set $\N$ from a very general optimization point of view and prove existence of minimizers.
\subsection{A general problem}\label{problem} We are interested in finding solutions of the following problem 
\begin{equation}\label{eq:P}
\tag{P}
    \inf_{N\in \N}F(N).
\end{equation}
By mimicking classical results \cite{BCombettes}, we establish that existence is satisfied if we assume that $F:\N\to (-\infty,\infty]$ is proper, convex, weak$^*$ lower semi-continuous, i.e., 
\begin{equation}\label{eq:weaklsc}
    \text{if } N_k\weakk N \text{ as } k\to\infty, \text{ then } F(N)\leq\liminf_{k\to\infty}F(N_k),
\end{equation}
and coercive in the following sense: 
\begin{equation}\label{eq:coerciv}
\text{if } \|N_k(0)\|\to\infty \text{ as } k\to\infty, \text{ then } F(N_k)\to\infty \text{ as } k\to\infty.
\end{equation}
We state such result in the following.
\begin{theorem}
\label{thm:existence}
Let $F:\N\to(-\infty, +\infty]$ be a proper and convex function that satisfies assumptions \cref{eq:weaklsc} and \cref{eq:coerciv}. Then, there exists a minimizer of \cref{eq:P}.
\end{theorem}

\begin{proof} Let $(N_k)_{k\in\mathbb{N}}$ be a minimizing sequence for $F$ in $\N$; i.e., such that $F(N_k) \to \inf_{N\in \N} F < \infty$ as $k\to\infty$. This, by coercivity of $F$, implies that $(N_k)_{k\in \mathbb{N}}$ is bounded in $\N$. Indeed, as $\mathcal{L}(N_k-N_k(0))= \mathcal{L}(N_k)\leq 1$, $\|N_k\|_{\Lip}\to\infty$ as $k\to\infty$ implies that $\|N_{k}(0)\|\geq\|N_k\|_{\Lip}-1\to \infty$ as $k\to\infty$, and thus $F(N_{k})\to \infty$, a contradiction. Bounded sets in $\Lip(\cX, \cX)$ are weakly$^*$ relatively sequentially compact by the Banach--Alaoglu--Bourbaki Theorem and by \cref{thm:separability}. Combining this with the fact that, by \cref{nnzero}, $\N$ is a weak$^*$-closed set, we get that there exists $N^*\in \N$ and a subsequence $(N_{k_j})_{j\in \mathbb{N}}$ of $(N_k)_{k\in \mathbb{N}}$ such that $N_{k_j} \weakk N^*$. By lower semi-continuity of $F$ we have $\inf_{N\in \mathbb{N}} F(N)=\liminf_{j\to \infty} F(N_{k_j}) \geq F(N^*) $.
\end{proof}

\subsection{The statistical model}\label{sec:statmodel}
Let $(\Omega, \mathcal{A}, P)$ be a probability space. The statistical model that we want to study can be expressed in terms of a supervised learning problem \cite{cusma}: we consider a pair of $\cX$-valued random variables $(\bar{X}, \, \bar{Z})$ on $\Omega$ with joint Borel probability distribution $\mu'$ on $\cX\times\cX$, and we aim at finding a firmly nonexpansive operator $T^*:\cX\to\cX$ such that $T^*(\bar{X})$ is close to $\bar{Z}$ in the following sense: using a quadratic loss, $T^*$ will be the minimizer of the so-called \emph{expected risk},
\begin{equation}
T^*\in\argmin_{T\in\M} F'(T); \quad F'(T):=\int_{\cX\times\cX}\|T(\bar x)-\bar z\|^2 \, \mathrm{d}\mu'(\bar x, \bar z),
\end{equation}
within the space of firmly nonexpansive operators $\M$ (recall that $\M$ defined in \cref{eq:fnespace} coincides with the space of firmly nonexpansive operators by \cref{lem:fneequivmm}). As we have already mentioned in \cref{sec:charac}, this problem can be equivalently formulated in terms of nonexpansive operators by applying the formula $T=(1/2)(\Id+N)$. Indeed, we define $\bar{Y}:= 2\bar{Z} - \bar{X}$ and denote by $\mu$ the joint Borel probability distribution of the pair $(\bar{X}, \bar{Y})$, i.e., the push-forward of $\mu'$ via the following transformation:
$$
\varphi : \cX\times\cX \to \cX \times \cX;\quad (x,z) \mapsto (x,2z-x).
$$
Therefore, we aim at finding a nonexpansive operator $N^*$ that minimizes the \emph{expected risk}
\begin{equation}\label{experr}
\tag{CP}
N^*\in \argmin_{N\in\N}F(N); \quad F(N):=\int_{\cX\times\cX}\|N(\bar x)-\bar y\|^2 \, \mathrm{d}\mu(\bar x, \bar y).
\end{equation}
To prove that there exists a minimizer for \cref{experr}, we need to assume that
\begin{equation}\label{eq:as1}
\tag{A1}
\int_{\cX\times\cX}\left(\|\bar x\|^2+\|\bar y\|^2\right) \, \mathrm{d}\mu(\bar x, \bar y)< \infty,
\end{equation}
and we point out that, in general, minimizers of \eqref{experr} may not be unique. We obtain the desired existence result by combining \cref{eq:as1} with both \cref{thm:existence} and the following result.
\begin{proposition}\label{prop:experr}
The expected risk in \cref{experr} is proper and convex. Additionally, it is weak$^*$ lower semi-continuous and coercive in the sense of \cref{eq:weaklsc} and \cref{eq:coerciv}.
\end{proposition}
\begin{proof} 
First, let us prove that $F$ is proper. Taking $N=\Id$, we get
$$
F(\Id)=\int_{\cX\times\cX} \|\bar x-\bar y\|^2 \, \mathrm{d}\mu(\bar x, \bar y)\leq 2\int_{\cX\times\cX} \left(\|\bar x\|^2+\|\bar y\|^2\right) \, \mathrm{d}\mu(\bar x, \bar y),
$$
where the right-hand side is finite by \cref{eq:as1}. Next, we show that $F$ is convex. Let $N$, $S\in\N$. Take $\alpha\in(0,1)$ and observe, using convexity of the squared norm, that
\begin{align*}
    F(\alpha N+ (1-\alpha)S)&=\int_{\cX\times\cX} \|(\alpha N+(1-\alpha)S)(\bar{x})-\bar{y}\|^2 \, \mathrm{d}\mu(\bar x, \bar y)\\
    &=\int_{\cX\times\cX} \|\alpha(N(\bar{x})-\bar{y})+(1-\alpha)(S(\bar{x})-\bar{y})\|^2 \, \mathrm{d}\mu(\bar x, \bar y)\\
    &\leq \int_{\cX\times\cX} \left(\alpha\|N(\bar{x})-\bar{y}\|^2+(1-\alpha)\|S(\bar{x})-\bar{y}\|^2\right) \, \mathrm{d}\mu(\bar x, \bar y).
\end{align*}
By the linearity of the integral, we obtain the desired result. Let us prove now that $F$ is coercive in the sense of \cref{eq:coerciv}. Let $(N_k)_{k\in\mathbb{N}}$ be a sequence in $\N$ such that $\|N_k(0)\|\to\infty$ when $k\to\infty$. Observe that, for every $\bar x, \, \bar y\in\cX$,
$$
\|N_k(0)\|\leq \|N_k(0)-N_k(\bar x)\| + \|N_k(\bar x)-\bar y\|+\|\bar y\|\leq \|N_k(\bar x)-\bar y\| + \|\bar x\| + \|\bar y\|,
$$
where we used the fact that $\cL(N_k)\leq 1$. Taking expectations in both sides and using H\"older inequality, we have that
$$
\begin{aligned}
\|N_k(0)\|&\leq \int_{\cX\times\cX} \|N_k(\bar x)-\bar y\| \, \mathrm{d}\mu(\bar x, \bar y)+\int_{\cX\times\cX} \|\bar x\| \, \mathrm{d}\mu(\bar x, \bar y)+ \int_{\cX\times\cX} \|\bar y\| \, \mathrm{d}\mu(\bar x, \bar y)\\
&\leq \left(\int_{\cX\times\cX} \|N_k(\bar x)-\bar y\|^2 \, \mathrm{d}\mu(\bar x, \bar y)\right)^{1/2}+ \left(\int_{\cX\times\cX} \|\bar x\|^2 \, \mathrm{d}\mu(\bar x, \bar y)\right)^{1/2}\\
&\quad+ \left(\int_{\cX\times\cX} \|\bar y\|^2 \, \mathrm{d}\mu(\bar x, \bar y)\right)^{1/2}\\
&= F(N_k)^{1/2} + \left(\int_{\cX\times\cX} \|\bar x\|^2 \, \mathrm{d}\mu(\bar x, \bar y)\right)^{1/2}+ \left(\int_{\cX\times\cX} \|\bar y\|^2 \, \mathrm{d}\mu(\bar x, \bar y)\right)^{1/2}.
\end{aligned}
$$
By taking limits in both sides, we obtain the desired result. It is left to prove that $F$ is weak$^*$ lower semi-continuous in the sense of \cref{eq:weaklsc}. Let $(N_k)_{k\in\mathbb{N}}$ be a sequence in $\N$ and let $N\in \N$ such that $N_k\weakk N$ as $k\to\infty$. Observe that, by Fatou's lemma,
\begin{align*}
\liminf_{k\to\infty}F(N_k)&=\liminf_{k\to\infty}\int_{\cX\times\cX} \|N_k(\bar x)-\bar{y}\|^2 \, \mathrm{d}\mu(\bar x, \bar y)\\
&\geq \int_{\cX\times\cX} \liminf_{k\to\infty}\|N_k(\bar x)-\bar{y}\|^2 \, \mathrm{d}\mu(\bar x, \bar y)\\
    &\geq \int_{\cX\times\cX} \|N(\bar x)-\bar{y}\|^2 \, \mathrm{d}\mu(\bar x, \bar y)=F(N),
\end{align*}
where we have used the fact that the squared norm is lower semi-continuous with respect to the weak topology on $\cX$ and, by \cref{topo}, that $N_k(\bar x)-\bar{y}\rightharpoonup N(\bar x)-\bar{y}$ as $k\to\infty$ for every $\bar x$, $\bar{y}\in\cX$.
\end{proof}

In practice, minimizers of \cref{experr} cannot be computed since, in general, the probability distribution $\mu$ is unknown. Instead, we suppose to have access to finitely many identically distributed and independent copies of the pair $(\bar X, \bar Y)$ on $\Omega$, denoted by  $(\bar{X}_i,\bar{Y}_i)$, $i=1,\ldots, n$, and we aim at finding minimizers of the following \emph{empirical risk minimization} problem,
\begin{equation}\label{eq:EP}
\tag{EP}
N^*_n\in \argmin_{N\in\N}F_n(N), \quad F_n(N):=\frac{1}{n}\sum_{i=1}^n \|N(\bar{X}_i)-\bar{Y}_i\|^2
\end{equation}
within the space of nonexpansive operators. In order to show that $N_n^*$ exists, it suffices to observe that \cref{eq:EP} is a particular case of \cref{experr} in the following sense: if we define the random variable $\hat \mu_n:=(1/n)\sum_{i=1}^n\delta_{(\bX_i, \bY_i)}$, then, for every $\omega\in\Omega$, $\hat \mu_n^\omega=(1/n)\sum_{i=1}^n\delta_{(\bX_i(\omega), \bY_i(\omega))}$ defines an empirical distribution supported in the finitely many points $(\bX_i(\omega), \bY_i(\omega))$, $i=1,..., n$, and so, for every $\omega\in\Omega$,
$$
N_n^*(\omega)\in\argmin_{N\in\N} F_n(N)=\int_{\cX\times\cX} \|N(x)-y\|^2 \, \mathrm{d}\hat \mu_n(\omega)(\bx, \by)
$$
exists for every $n\in\mathbb{N}$. Notice that, for every $\omega\in\Omega$, the probability distribution $\hat{\mu}(\omega)$ is compactly supported, and hence satisfies \cref{eq:as1}. We recall again that, as it was the case for \eqref{experr}, minimizers of \eqref{eq:EP} may not be unique in general.

Now, we aim to show that \cref{eq:EP} is a good approximation of \cref{experr} for $n$ large enough. This is a standard question in supervised learning theory \cite{cusma}, and typically, it is tackled by showing that
$$
F(N^*_n)-F(N^*)\sim \mathcal{O}\left(\frac{1}{\sqrt{n}}\right),
$$
i.e., that the \emph{excess risk} goes to zero with rate $1/\sqrt{n}$. To prove this result, it is often assumed that the underlying space is compact \cite[Theorem C]{cusma}. In our case, the space of interest is $\N$, which is in general non-compact. Additionally, such an estimation would require measurability of the minimizers $N^*_n$ for any $n\in\mathbb{N}$, which is not obvious. Instead, we present in the following a more qualitative approach that avoids these considerations, and show that problem \cref{eq:EP}  $\Gamma$-converges to \cref{experr} when $n$ goes to infinity (see \cite{braides}). For this reason, we assume that $(\bar{X}_i, \bar{Y}_i)$, $i=1, 2,\ldots$, are countably many pairs of random variables on $\Omega$, all independent and distributed as $(\bar{X}, \bar{Y})$. In the following, we will denote by $N^*_n(\omega)$ and $F_n^\omega$ the corresponding versions of \cref{eq:EP} whenever $\omega\in \Omega$ is fixed; i.e., for every $\omega\in \Omega$ we denote
\begin{equation}\label{eq:Fn_sampled}
N^*_n(\omega)\in \argmin_{N\in\N}F_n^\omega(N), \quad F_n^\omega(N):=\frac{1}{n}\sum_{i=1}^n \|N(\bar{X}_i(\omega))-\bar{Y}_i(\omega)\|^2.
\end{equation}
We also recall the definition of the $p$-Wasserstein distance, $p\geq 1$.
\begin{definition}[Wasserstein distance]
    Let $\mathcal{H}$ be a separable Hilbert space, let $p\geq 1$, and let $\mathcal{P}_p(\mathcal{H})$ be the set of Borel probability measures with finite $p$-moments. Then, the \emph{$p$-Wasserstein distance between two measures $\nu_1,\nu_2\in\mathcal{P}_p(\mathcal{H})$} is defined by
    \begin{equation}\label{eq:ot}
    W_p(\nu_1,\nu_2) := \inf_{\pi \in \Gamma(\nu_1,\nu_2)} \left(\int_{\mathcal{H}\times \mathcal{H}} \|v-w\|^p \, \mathrm{d}\pi(v,w)\right)^{\frac{1}{p}},
    \end{equation}
    where $\Gamma(\nu_1,\nu_2)$ is the set of probability measures in $\mathcal{P}_p(\mathcal{H}\times \mathcal{H})$ with marginals $\nu_1$ and $\nu_2$. An optimal coupling $\pi$ attaining the above infimum always exists (see \cite[Theorem 4.1]{Villani2009}) and it is called an \emph{optimal $W_p$ coupling between $\nu_1$ and $\nu_2$.}
\end{definition}

For all $p\geq 1$ the function $W_p: \mathcal{P}_p(\mathcal{H})\times \mathcal{P}_p(\mathcal{H}) \to \mathbb{R}_+$ defines a distance and $(\mathcal{P}_p(\mathcal{H}), W_p)$ is called the $p$-Wasserstein space. In our setting, we have by assumption \cref{eq:as1} that $\mu\in\mathcal{P}_2(\cX\times\cX)$. Hence, we will from now on consider the $2$-Wasserstein distance $W_2$. In order to prove our result, we will use the following form of the so-called Glivenko--Cantelli theorem \cite[Theorem 3]{Varadarajan}.

\begin{theorem}[Glivenko--Cantelli]\label{thm:GliCa} There exists $\Omega_0\subset \Omega$ with $P(\Omega_0)=0$ such that, for all $\omega \in \Omega\setminus \Omega_0$, we have that $\hat \mu^\omega_n:=\frac{1}{n}\sum_{i=1}^n\delta_{(\bX_i(\omega), \bY_i(\omega))}$ converges in $W_2$ to $\mu$ as $n\to\infty$.
\end{theorem}
\begin{proof}
    Theorem 3 in \cite{Varadarajan} shows that there exists $\Omega_0'\subset \Omega$ with $P(\Omega_0')=0$ such that for all $\omega \in \Omega \setminus \Omega_0'$ we have that $\hat \mu^\omega_n$ converges narrowly to $\mu$ as $n\to \infty$. We recall that narrow convergence means convergence against continuous bounded functions. Considering now the measure $\bar \mu:=\|\cdot \|^2\mu$, by the same reasoning, there exists a set $\Omega_0''\subset \Omega$ with $P(\Omega_0'')=0$ such that the  corresponding empirical measure converges narrowly to $\bar \mu$ on $\Omega \setminus \Omega_0''$. Combining, we have that for all $\omega \in \Omega \setminus \Omega_0$ with $\Omega_0=\Omega_0'\cup \Omega_0''$, the sequence $(\hat \mu^\omega_n)_n$ converges narrowly to $\mu$ and its second moments converge to the second moment of $\mu$. By \cite[Theorem 6.9]{Villani2009}, we thus have $W_2(\hat \mu^\omega_n, \mu)\to 0$.
\end{proof}

In order to state our $\Gamma$-convergence result, we need to consider the following lemma, whose proof is an adaptation of \cite[Lemma 5]{Raginsky}. The lemma is further based on \cite[Proposition 1]{PolyanskiyWu}, and we need to adapt it to our specific (and strictly different) case. For this reason, we include the proof.

\begin{lemma}\label{lem:Gamma_help}
Let $\mathcal{X}$ be a separable Hilbert space and let $\nu_1,\nu_2\in\mathcal P_2(\cX\times \cX)$. Let $N$ be a nonexpansive operator and define, for all $x, y\in\cX$, the function $g(x,y):=\|N(x)-y\|^2$. Then, it holds true that
\[
\left|\int_{\cX\times \cX} g(v) \, \mathrm{d}(\nu_1-\nu_2)(v)\right|\leq 2\sqrt{3} \left( \int_{\cX\times \cX} \|v\|^2 \, \mathrm{d}(\nu_1+\nu_2)(v) + 2\|N(0)\|^2\right)^{\frac{1}{2}} W_2(\nu_1,\nu_2). 
\]
\end{lemma}
\begin{proof}
First, let us set $v_i=(x_i,y_i)\in \mathcal{H}:=\cX \times \cX$ for $i=1,2$. Using that, for all $a, b\in\mathbb{R}$, we have $\|a\|^{2}-\|b\|^{2}=\langle a+b,a-b\rangle$, and the fact that $N$ is nonexpansive, we have that
\begin{align*}
|g(v_1)-g(v_2)|
&=\big|\|N(x_1)-y_1\|^{2}-\|N(x_2)-y_2\|^{2}\big|\\
&=\big|\langle N(x_1)-y_1+N(x_2)-y_2,\,
            N(x_1)-N(x_2)-(y_1-y_2)\rangle\big|\\
&\le \big(\|N(x_1)-y_1\|+\|N(x_2)-y_2\|\big)\,
       \big(\|x_1-x_2\|+\|y_1-y_2\|\big)\\
&\le \big(\|x_1\|+\|y_1\|+\|x_2\|+\|y_2\|+2\|N(0)\|\big)\,
      \big(\|x_1-x_2\|+\|y_1-y_2\|\big).
\end{align*}
Recall that for any $a\in \mathbb{R}^d$ we have $\|a\|_1\leq \sqrt{d}\|a\|_2$. Using this fact both with $d = 6$ and with $d=2$, we obtain
\begin{equation}\label{eq:pointwise-explicit}
|g(v_1)-g(v_2)|
\;\le\;
2\sqrt{3}
\Big(\|v_1\|^{2}+\|v_2\|^{2}+2\|N(0)\|^{2}\Big)^{1/2} \|v_1-v_2\|.
\end{equation}
Let $\pi$ be an optimal $W_{2}$ coupling between $\nu_1$ and $\nu_2$. Integrating \eqref{eq:pointwise-explicit} against $\pi$ and using Cauchy-Schwarz, we obtain
\begin{align*}
\Big|\int_{\cX\times \cX} g(v)\,\mathrm{d}(\nu_1-\nu_2)(v)\Big|
&=\Big|\int_{\mathcal{H}\times \mathcal{H}}\big(g(v_1)-g(v_2)\big)\,\mathrm{d}\pi(v_1,v_2)\Big|\\
&\le 2\sqrt{3}\Big(\int_{\mathcal{H}\times \mathcal{H}}\|v_1\|^{2}+\|v_2\|^{2}+2\|N(0)\|^{2}\,\mathrm{d}\pi(v_1,v_2)\Bigg)^{1/2} W_{2}(\nu_1,\nu_2).
\end{align*}
Since $\pi$ has marginals $\nu_1$ and $\nu_2$,
\[
\int_{ \mathcal{H}\times\mathcal{H}}\|v_1\|^{2}+\|v_2\|^{2}\,\mathrm{d}\pi(v_1,v_2)
=\int_{\mathcal{H}}\!\|v\|^{2}\,\mathrm{d}\nu_1(v)+\int_\mathcal{H}\!\|v\|^{2}\,\mathrm{d}\nu_2(v),
\]
and the stated inequality follows.
\end{proof}
We are now ready to present the $\Gamma$-convergence result. Recall the definitions of $F_n^\omega$ and $F$ from \cref{eq:Fn_sampled} and \cref{eq:EP}, respectively.
\begin{theorem}\label{thm:gammaconv} With probability $1$, $F_n^{\omega}$ $\Gamma$-converges to $F$ as $n\to\infty$. More precisely, there exists a set $\Omega_0\subset \Omega$, with $P(\Omega_0)=0$ such that for all $\omega \in \Omega\setminus \Omega_0$ both of the following conditions, implying $\Gamma$-convergence, are satisfied:
\begin{enumerate}[label=(\roman*),leftmargin=*]
    \item \label{itm:gammaconv_i} for every $N\in\N$ and every sequence $(N_n)_{n\in\mathbb{N}}$ in $\N$ with $N_n\weakk N$ in $\Lip(\cX, \cX)$, we have that
    $$
    F(N)\leq \liminf_{n\to\infty} F_n^{\omega}(N_n),
    $$
    \item \label{itm:gammaconv_ii} for every $N\in\N$ it holds 
    \[\limsup_{n\to\infty} F_n^{\omega}(N)\leq F(N).\]
\end{enumerate}
\end{theorem}

\begin{proof} First, by \cref{thm:GliCa}, there exists a set $\Omega_0\subset\Omega$ with $P(\Omega_0)=0$ such that $W_2(\hat \mu^{\omega},\mu)\to 0$ as $n\to \infty$ for all $\omega \in \Omega \setminus \Omega_0$. Let $N$ be a nonexpansive operator. We start first by showing \cref{itm:gammaconv_ii}. To do so, we will show that the constant sequence $N_n:=N$, $n\in\mathbb{N}$, is a recovery sequence. Since $\hat \mu^\omega_n $ is converging to $\mu$ in $W_2$, it has uniformly bounded second moments (see \cite[Theorem 6.9]{Villani2009}), and from \cref{lem:Gamma_help} we have convergence against the function $(x,y)\mapsto \|N(x)-y\|^2$. In particular, we have
\[
\lim_{n\to\infty}F_n^\omega(N) =\lim_{n\to\infty} \int \|N(x)-y\|^2 \, \mathrm{d} \hat \mu_n^{\omega}(x,y) = \int \|N(x)-y\|^2 \, \mathrm{d} \mu(x,y)=F(N).
\]
We now turn to proving the ``$\liminf$'' inequality, i.e., \cref{itm:gammaconv_i}. Let $(N_n)_{n\in\mathbb{N}}$ be a sequence in $\N$ such that $N_n\weakk N$ for some $N\in \N$ and define, for all $n\in\mathbb{N}$ and all $x, y\in\cX$, the functions
\[
f_n(x,y):=\|N_n(x)-y\|^2,\qquad f(x,y):=\|N(x)-y\|^2.
\]
We want to prove that
\[
\liminf_{n\to\infty} F_n^\omega(N_n) =\liminf_{n\to\infty}\int f_n\,\mathrm{d}\hat \mu^\omega_n \;\ge\; \int f\,\mathrm{d}\mu = F(N).
\]
To do so, we notice that
\[
F_n^\omega(N_n)=\int f_n\,\mathrm{d}\hat \mu^\omega_n = \int f_n\,\mathrm{d}(\hat \mu^\omega_n-\mu)+\int f_n\,\mathrm{d}\mu.
\]
Regarding the first term in the right-hand side, by \cref{lem:Gamma_help} we have
\[
\left|\int f_n \, \mathrm{d}(\hat \mu^\omega_n-\mu)\right|\leq 2\sqrt{3} \left( \int_{\cX\times \cX} \|v\|^2 \, \mathrm{d}\hat \mu^\omega_n(v) + \int_{\cX\times \cX} \|v\|^2 \, \mathrm{d}\mu(v) + 2\|N_n(0)\|^2\right)^{1/2} W_2(\hat \mu^\omega_n,\mu). 
\]
The first factor is uniformly bounded because $\hat \mu^\omega_n\to\mu$ in $W_2$ implies uniformly bounded second moments (see again \cite[Theorem 6.9]{Villani2009}). Additionally, $N_n\weakk N$ implies that $N_n(0)$ is bounded. Combining this with $W_2(\hat \mu^\omega_n,\mu)\to 0$, we have
\[
\Big|\int f_n\,\mathrm{d}(\hat \mu^\omega_n-\mu)\Big|
\to 0, \qquad \text{as } \ n\to\infty.
\]
On the other hand, for fixed $y\in\cX$, the function $z\mapsto \|z-y\|^2$ is weakly lower semicontinuous, and so for all $x, y\in\cX$ we have that
\[
\liminf_{n\to\infty} f_n(x,y)
=\liminf_{n\to\infty}\|N_n(x)-y\|^2
\ge \|N(x)-y\|^2=f(x,y).
\]
Applying Fatou’s lemma gives
\[
\liminf_{n\to\infty}\int f_n\,\mathrm{d}\mu
\ge \int \liminf_{n} f_n\,\mathrm{d}\mu
\ge \int f\,\mathrm{d}\mu.
\]
Putting everything together,
\[
\liminf_{n\to\infty}\int f_n\,\mathrm{d}\hat \mu^\omega_n
=\liminf_{n\to\infty}\Big(\int f_n\,\mathrm{d}(\hat \mu^\omega_n-\mu)+\int f_n\,\mathrm{d}\mu\Big)
\ge 0 + \int f\,\mathrm{d}\mu
=F(N),
\]
as desired.
\end{proof}

A direct consequence of the fundamental theorem of $\Gamma-$convergence \cite[Theorem 2.1]{braides} is the following: if on a set of probability $1$ the sequence $(F_n^\omega)_{n\in\mathbb{N}}$ is an equicoercive sequence of functions that is $\Gamma-$converging to $F$, then, up to subsequences, a sequence of minimizers $(N_n^*)_{n\in\mathbb{N}}$ of \cref{eq:EP} converges to a minimizer of the continuous problem \cref{experr} on a set of probability $1$. To show this, it is only left to prove that the sequence $(F_n)_{n\in\mathbb{N}}$ is equicoercive up to a null set. Before doing so, we recall that a sequence $(F_n)_{n\in\mathbb{N}}$, $F_n\colon \N\to (-\infty, \infty]$, $n\in\mathbb{N}$, is equicoercive on $\N$ with respect to the weak$^*$ topology if for any constant $C>0$, there exists a weakly$^*$-compact set $K\subset \N$ such that, for any $n\in\mathbb{N}$, the sublevel set of $F_n$, $\{F_n\leq C\}$, is contained in $K$. We are now ready to present the desired result.

\begin{prop}\label{prop:equicoerciv} Let $\Omega_0$ be the set given by \cref{thm:GliCa}. Then, for every $\omega \in \Omega\setminus \Omega_0$, the sequence $(F_n^\omega)_{n\in\mathbb{N}}$ defined as in \cref{eq:Fn_sampled} is equicoercive on $\N$ with respect to the weak$^*$ topology.
\end{prop}
\begin{proof} First, observe that
$$
\begin{aligned}
\|N(\bX_i(\omega))-\bY_i(\omega)\|^2&=\|N(\bX_i(\omega))-N(0)-\bY_i(\omega)\|^2+\|N(0)\|^2\\
& \hspace{5cm} +2\langle N(\bX_i(\omega))-N(0)-\bY_i(\omega),N(0)\rangle\\
&\geq \|N(\bX_i(\omega))-N(0)-\bY_i(\omega)\|^2+\|N(0)\|^2\\
& \hspace{5cm} -2\|N(\bX_i(\omega))-N(0)-\bY_i(\omega)\|\|N(0)\|\\
&\geq \frac12\|N(0)\|^2-\|N(\bX_i(\omega))-N(0)-\bY_i(\omega)\|^2\\
&\geq\frac12\|N(0)\|^2-2\left(\|\bX_i(\omega)\|^2+\|\bY_i(\omega)\|^2\right)
\end{aligned}
$$
since, for every $a$, $b\in\mathbb{R}$, it holds that $ab\leq a^2/4+b^2$, which implies $-2ab\geq -a^2/2-2b^2$. Then,
$$
F_n^\omega(N)\geq\frac12\|N(0)\|^2-\frac2n\sum_{i=1}^n\left(\|\bX_i(\omega)\|^2+\|\bY_i(\omega)\|^2\right),
$$
where, by \cref{thm:GliCa} combined with \cite[Theorem 6.9]{Villani2009}, the term in the right-hand side satisfies
$$
\frac2n\sum_{i=1}^n(\|\bX_i(\omega)\|^2+\|\bY_i(\omega)\|^2)\rightarrow 2\int_{\cX\times\cX}\left(\|\bx\|^2+\|\by\|^2\right)\, \mathrm{d}\mu(\bar x, \bar y),\quad \text{as } n\to \infty. 
$$
Since it is a converging sequence, it is also bounded by some constant $C>0$, which leads to
$$
F_n^\omega(N)\geq c\|N(0)\|^2-C,
$$
for some $c>0$. Hence, for every $C'>0$ we have that
$$
\left\{F_n^{\omega}(N)\leq C'\right\}\subset\left\{N\in\N \ \colon \|N(0)\|\leq C''\right\}
$$
for some $C''>0$ (in fact, $C''=((C'+C)/c)^{1/2}$) and where the set on the right-hand side is weak$^*$-compact in $\Lip(\cX, \cX)$. Indeed, by convexity, weak$^*$ lower semi-continuity and \cite[Corollary 2.6.19]{megg}, we only need to show boundedness in $\Lip(\cX, \cX)$:
$$
\|N\|_{\Lip}=\|N(0)\|+\cL(N-N(0))\leq C''+1.
$$
This concludes the proof.
\end{proof}

As we previously mentioned, the following corollary is a direct consequence of both \cref{thm:gammaconv} and \cref{prop:equicoerciv}. Its proof can be found in \cite{braides} in a more general setting. 

\begin{corollary}\label{cor:convofminimizers} There exist a set $\Omega_0\subset \Omega$ with $P(\Omega_0)=0$ such that for all $\omega \in \Omega \setminus \Omega_0$ and all sequences of minimizers $(N^*_n(\omega))_{n\in\mathbb{N}}$ of \cref{eq:Fn_sampled} it holds true that
\begin{enumerate}[label=(\roman*),leftmargin=*]
    \item the sequence $(N^*_n(\omega))_{n\in\mathbb{N}}$ admits a weak$^*$-cluster point,
    \item evey weak$^*$ cluster point of $(N^*_n(\omega))_{n\in\mathbb{N}}$ is in $\argmin F$,
    \item we have $\inf_\N F_n^\omega \to \inf_\N F$.
\end{enumerate}
\end{corollary}

The above results provide a theoretical analysis that highlights the fact that \cref{eq:EP} is indeed approximating its continuous version \cref{experr} for large values of $n$. In the next section, we will introduce and analyze a practical approach to learn nonexpansive operators. To do so, we shift to a deterministic, finite dimensional setting with $\cX=\R^d$, and consider finitely many samples $(\bar x_i, \bar y_i)$, $i=1,\ldots,n$, where for each $i=1,\ldots, n$,  $\bx_i, \by_i\in\R^d$, and each pair $(\bx_i, \by_i)$ is a realization of the corresponding pair of random variables $(\bX_i, \bY_i)$ (i.e., there exists $\omega\in\Omega$ such that $\bX_i(\omega)=\bx_i$ and $\bY_i(\omega)=\by_i$). For later purposes, we require that the realizations $\bx_i$, $i=1,\dots, n$, are distinct points and do not lie on a $(d-1)-$dimensional hyperplane. This is true almost surely if we assume that $n > d$ and that the law of $\bar X$ gives zero mass to each $(d-1)-$dimensional hyperplane. With this, we can consider the following discrete problem:
\begin{equation}\label{eq:DP}
\tag{DP}
\widehat{N}\in \argmin_{N\in\N}\widehat{F}(N), \quad \widehat{F}(N):=\frac{1}{n}\sum_{i=1}^n \|N(\bar{x}_i)-\bar{y}_i\|^2.
\end{equation}
As a particular case of \cref{eq:EP} (actually, it is one realization of \cref{eq:EP}), we derive that $\widehat{N}$ exists. In order to solve \cref{eq:DP}, classical optimization algorithms can be considered, but the class of nonexpansive operators is infinite-dimensional, and so, problem \cref{eq:DP} remains computationally infeasible. Therefore, a further approximation needs to be considered. To do so, a first idea could be to consider problem \cref{eq:DP} restricted to the set $\mathcal{N}(\bar{D})$ of nonexpansive operators defined on the discrete set $\bar{D}:=\{\bar{x}_1,\ldots, \bx_n\}$; i.e., operators $N:\bar{D}\to \cX$ such that $\|N(\bar{x}_i)-N(\bar{x}_j)\|\leq \|\bar{x}_i-\bar{x}_j\|$ for $i,j=1,\dots,n$. The corresponding reformulation of this discrete version of \cref{eq:DP} would be
\begin{equation}\label{eq:sololip}
\begin{aligned}
    \widehat Y\in\argmin_{y_1,\ldots, \, y_n\in\cX} & \quad \frac{1}{n}\sum_{i=1}^n\|y_i-\bar{y}_i\|^2\\ 
    & \text{ s.t. } \|y_i-y_j\|\leq \|\bx_i-\bx_j\|, \text{ for every } i,j\in\{1,\dots,n\},
\end{aligned}
\end{equation}
with $\widehat Y= (\hat y_1, \ldots,\hat y_n)\in \mathcal{X}^n$. In this way, the desired operator $N$ in $\mathcal{N}(\bar D)$ is uniquely characterized by the solution $\widehat Y$ via the formula $N(\bar x_i):= \hat y_i$, for each $i=1, \ldots, n$. However, the above problem leaves us with the following challenge: given a solution of \cref{eq:sololip}, how can we extend it to the whole space? In fact, in practical scenarios, we need to know the value of such an operator at any point in $\cX$. For instance, we know that a nonexpansive extension to the whole space always exists \cite{kirz}, but it is difficult to construct it in practice. We point out the work \cite{Zaichyk2023}, in which authors solve \cref{eq:sololip} for a general Lipschitz constant $L>0$, and propose an alternative method to extend the solution to the whole space, via the Kirszbraun extension Theorem. Our proposed approach, however, will be based on the following considerations. In the finite-dimensional setting $\cX=\R^d$ we are able to construct operators $N$ that are $1$-Lipschitz and, in addition, piecewise affine on $d$-simplices (or just simplices). In this way, we will show that problem \cref{eq:DP} can be reduced to finding the value of $N$ in finitely many points. This idea will be explained in the following.

\begin{remark}
A classical way of approaching constrained problems like \cref{eq:DP} or \cref{eq:sololip} is to consider the so-called Frank--Wolfe algorithms or conditional gradients methods \cite{fwa,jaggi2013}. As noted in recent works \cite{breetal22,brelinear,crisiglewal}, a proper characterization of the set of extremal points of the constraint set ---or the unit ball associated with the regularizer, see \cite{duval}--- can improve the efficiency of such type of algorithms. Nevertheless, a meaningful characterization of the extremal points of $\N_0(\cX)$ has not been provided yet. A partial answer to this has been given in \cite{bredies2023extreme} in the case of finite metric spaces, and can be applied to find a solution to problem \cref{eq:sololip}. We believe that further analysis in this direction could be crucial to solve this class of problems.
\end{remark}

In the following, we develop the main tool that we consider in order to discretize the set $\N$ in the finite-dimensional setting: simplicial partitions.

\subsection{Simplicial partitions}
\label{triangs}
Let $D:=\{x_1,\ldots, x_m\}$, $x_i\in\R^d$, $d+1\leq m$, be a general finite set of distinct points that do not lie on a $(d-1)-$dimensional hyperplane and consider its convex envelope $\conv(D)$, which has non-empty interior. We want to utilize simplicial partitions of $\conv(D)$. For this purpose, let $\ell\in\mathbb{N}$ and
$$
\mathfrak{T}=\{S_1,\ldots, S_\ell\},
$$
where for every $t=1,\ldots, \, \ell$, $S_t=\conv\{x_{i_0},\dots,x_{i_d}\}$, for some subcollection of $d+1$ distinct elements of $D$. We assume $\mathfrak{T}$ to have the following properties:

\begin{enumerate}[label=(P\arabic*)]
    \item \label{item:P1} $\mathfrak{T}$ forms a partition of $\conv(D)$ up to Lebesgue null sets; i.e.,
    $$
    \sum_{t=1}^\ell \Chi_{S_t} =\Chi_{\conv(D)} \quad \text{a.e.;}
    $$
    \item \label{item:P2} the interior of every simplex $S_t$ is non-empty;
    \item \label{item:P3} the intersection of every two simplices has to be either empty or coincide with the convex envelope of its common vertices.
\end{enumerate}
A partition defined by these conditions is also known as \emph{face-to-face simplicial partition} for a polytope \cite{hannu2014}, where \textit{face} stands for the convex hull of any collection of $d$ vertices not lying on a $(d-2)$-dimensional hyperplane. For example, in dimension $d=2$, \textit{face} denotes the edge of a triangle, while in dimension $d=3$, the \textit{faces} of a tetrahedron are triangles. An example of a classical partition satisfying all of the above conditions is the so-called \emph{Delaunay triangulation} \cite{fortune}. Nevertheless, we are not interested in fixing a particular triangulation, but only one satisfying the above conditions.

\subsection{Piecewise affine nonexpansive operators}\label{sec:piecewise}
The goal of this section is to construct a finite-dimensional set of operators which will be characterized only by their value on each node $x_i$, $i=1,\dots, m$. We first fix the values $y_1,\ldots,y_m\in \R^d$. These will determine the value of the constructed operator at every point in $D$, i.e., the operator $N$ will be uniquely defined by $N(x_i):=y_i$, $i=1,\ldots, n$. In addition, we aim at constructing $N$ in such a way that it is also nonexpansive. To do so, we first consider an operator $N$ that is affine on every simplex $S_t$ and, second, find a suitable condition for the vertices so that the resulting operator is nonexpansive. Let us remark that, once the nonexpansive operator $N$, as an approximation to a solution of \cref{eq:DP}, is well defined, in order to then construct a corresponding firmly nonexpansive operator $T$, one just needs to employ the formula $T= (1/2)(\Id+N)$, since the samples $(\bx_1, \by_1),\ldots, (\bx_n,\by_n)$ were drawn from $(\bX, \bY)$, and $\bar Y=2\bar Z-\bar X$, see \cref{sec:statmodel}.

Let $\mathfrak{T}=\{S_1,...,S_\ell\}$, $\ell\in\mathbb{N}$, denote a simplicial partition of $\conv(D)$ with properties \ref{item:P1}, \ref{item:P2} and \ref{item:P3} given in the section above. Consider
$$
\lambda_1,\ldots, \lambda_m\colon\conv(D)\to[0, 1]
$$ 
the Lagrange elements of order $1$ associated with the simplicial partition $\mathfrak{T}$ (see \cite[Chapter 4]{Quarteroni2017} for more details); i.e., such that
\begin{enumerate}[label=(\roman*),leftmargin=*]
\item $\lambda_i(x_j)=\delta_{ij}$, for $i, j = 1,\ldots, m$ (here, $\delta_{ij}$ denotes the Kronecker delta);
\item $ \lambda_i\vert_{S_t}$ is a polynomial of degree at most $1$ for each  $i=1,\ldots, m$ and $t=1,\ldots, \ell$.
\end{enumerate}
Note that each $\lambda_i$ is continuous in $\conv(D)$. In addition, if $i_0,\ldots, i_d$ denote the indices of the vertices $x_{i_0}, \ldots, x_{i_d}$ of the simplex $S_t$, then $\lambda_{i_0}\vert_{S_t},\ldots, \lambda_{i_d}\vert_{S_t}$ correspond to the barycentric coordinates on $S_t$; i.e., the unique non-negative functions that satisfy
\begin{equation}\label{eq:barcoords}
\sum_{j=0}^d \lambda_{i_j}(x)=1, \quad \sum_{j=0}^d\lambda_{i_j}(x)x_{i_j}=x,
\end{equation}
for each $x\in S_t$. Further, as $\lambda_i\vert_{S_t}=0$ for each $i\in\{1,\ldots, m\}\setminus\{i_0,\ldots, i_d\}$, we have that $\sum_{i=1}^m\lambda_i(x)=1$ for any $x\in\conv(D)$, which immediately implies that 
$$
\sum_{i=1}^m\lambda_i=\Chi_{\conv(D)},\quad \text{and } \quad \sum_{i=1}^m\lambda_i x_i=\Id\vert_{\conv(D)}.
$$ 
Next, given $y_1,..., y_m\in\R^d$, we define the operator
\begin{equation}\label{eq:def_of_Ntilde}
\widetilde{N}\colon\conv(D)\to\R^d; \quad \widetilde{N}(x):=\sum_{i=1}^m\lambda_i(x)y_i.
\end{equation}
Finally, for every simplex $S_t$, $t=1,\ldots, \ell$, with vertices $i_0,\ldots, i_d$, we have that $\widetilde{N}\vert_{S_t}=\sum_{j=0}^d \lambda_{i_j}\vert_{S_t}y_{i_j}$, which, since each $\lambda_{i_j}\vert_{S_t}$ is affine linear, implies that $\widetilde{N}\vert_{S_t}$ is also affine linear. Clearly, $\widetilde{N}(x_i)=y_i$ for each $i=1,\ldots, m$.

In order to extend such an operator to the whole space, we consider
$$
N:=\tilde{N}\circ \pi_{\conv(D)}
$$
where $\pi_{\conv(D)}$ stands for the projection onto $\conv(D)$. Recall that $\pi_{\conv(D)}$ is a nonexpansive operator. Therefore, if $\tilde{N}$ is nonexpansive, then also $N$ shares this property. In particular, $N(x_i)=\tilde N (x_i)$ for $i=1,\dots, m$. Given this construction, we can define the set of piecewise affine operators on $\conv(D)$ associated with the simplicial partition $\mathfrak{T}$ that are extended to the whole space via projection:
$$
\mathrm{PA}(\mathfrak{T}):=\left\{N:\R^d\to\R^d  \, : \,  N:=\tilde{N}\circ \pi_{\conv(D)}, \ y_1, \ldots, y_m\in\R^d\right\},
$$
where $\tilde{N}$ is given by \cref{eq:def_of_Ntilde}. The finite-dimensional, non-empty, closed convex set of nonexpansive operators we consider now is $\mathrm{PA}(\mathfrak{T}) \cap \N$ (note that $\mathrm{PA}(\mathfrak{T})$ is a finite-dimensional subset of $\Lip(\R^d)$, hence $\mathrm{PA}(\mathfrak{T}) \cap \N$ is weak$^*$-closed). We focus on solving the deterministic problem \cref{eq:DP} in this class; i.e.,
\begin{equation}\label{eq:PAP}
\tag{PAP}
\widehat{L}\in\argmin_{N\in\mathrm{PA}(\mathfrak{T})\cap\N}\widehat{F}(N).
\end{equation}
The existence of solutions of $\widehat{F}$ follows by repeating the arguments that led to existence of solutions of \cref{eq:DP} for the non-empty, weak$^*$-closed set $\mathrm{PA}(\mathfrak{T}) \cap \N$, and noticing that the function $\widehat{F}+\iota_{\mathrm{PA}(\mathfrak{T})\cap\N}$ is coercive (see \cref{thm:existence} and \cref{prop:experr}). Our objective now is to derive an equivalent formulation of the above problem by imposing conditions on the finitely many points $(x_1, y_1), \ldots (x_m, y_m)$. To do so, since $\widehat{F}$ in \cref{eq:DP} only depends on the set $\bar{D}$, we assume for simplicity that $\bar D\subset D$. In particular, we consider $m \geq n$ and $x_i=\bar{x}_i$, for $i=1,\dots, n$ (i.e., the first $n$ points enumerate $\bar D$). A first attempt could be to impose the nonexpansivity condition on these points, i.e., $\|N(x_i)-N(x_j)\|\leq \|x_i-x_j\|$, for every $i, \, j=1,\ldots,m$. Nevertheless, simple computations show that the resulting extended operator, according to \cref{eq:def_of_Ntilde}, can fail to be nonexpansive. In order to find the right condition, we first observe that the operator $\tilde{N}$ can be rewritten in a more convenient way. For this purpose, let us introduce some useful tools. For every simplex $S_t\in \mathfrak{T}$, denote again by $i_0,\ldots, i_d$ the indices associated with the vertices $x_{i_0}, \ldots, x_{i_d}$ in $S_t$. Define the matrices
\begin{equation}\label{eq:matrices}
A_t:=[x_{i_1}-x_{i_0}\mid\dots \mid x_{i_d}-x_{i_0}], \quad B_t=[y_{i_1}-y_{i_0}\mid\ldots \mid y_{i_d}-y_{i_0}].
\end{equation}
Moreover, for every $x\in\conv(D)$, there exists $t\in\{1,\ldots, \ell\}$ such that $x\in S_t$. Then, since the barycentric coordinates $\lambda_{i_0},\ldots, \lambda_{i_d}$ sum up to $1$, we can write
$$
x=\sum_{j=0}^d\lambda_{i_j}(x)x_{i_j}= x_{i_0} + \sum_{j=1}^d\lambda_{i_j}(x)(x_{i_j}-x_{i_0}).
$$
Then,
$$
x-x_{i_0}=\sum_{j=1}^d\lambda_{i_j}(x)(x_{i_j}-x_{i_0})=A_t[\lambda_{i_1}(x),\dots,\lambda_{i_d}(x)]^T.
$$
By defining $\Lambda_t(x):=[\lambda_{i_1}(x),\dots,\lambda_{i_d}(x)]^T$, we derive that
$$
\Lambda_t(x)=A_t^{-1}(x-x_{i_0}).
$$
Hence, we obtain the following affine expression for the operator $\tilde{N}$ in every simplex:
\begin{equation}\label{eq:piecewise}
\tilde{N}\vert_{S_t} (x)=y_{i_0} + \sum_{j=1}^d\lambda_{i_j}(x) (y_{i_j}-y_{i_0})= y_{i_0} + B_t\Lambda_t(x) =  y_{i_0} + B_tA_t^{-1}(x-x_{i_0}).
\end{equation}
We are now ready to present a necessary and sufficient condition for $\tilde{N}$ (hence for $N$) to be nonexpansive by characterizing the elements in $\mathrm{PA}(\mathfrak{T}) \cap \N$. To do so, we first recall a preliminary result about the differentiability of Lipschitz functions, see \cite{Heinonen2005} for a proof.
\begin{prop}\label{prop:aediff}
Let $N$ be a nonexpansive operator. Then, it is differentiable almost everywhere and, for every differentiability point $x$, it holds that $\|\nabla N(x)\|\leq 1$.
\end{prop}
Recall that $\nabla N(x)$ denotes the Jacobian matrix of $N$ at the point $x$ and, from now on, $\|\cdot\|$ denotes the operator norm of the matrix. The nonexpansivity characterization reads as follows.
\begin{theorem}\label{thm:1lip}
The operator $N=\tilde{N}\circ\pi_{\mathrm{conv}(D)}$, with $\tilde N$ according to \cref{eq:def_of_Ntilde} is nonexpansive if and only if, for every $t=1,\ldots, \ell$, we have $\|B_tA_t^{-1}\|\leq 1$, with $A_t$, $B_t$ given by \cref{eq:matrices}.
\end{theorem}
\begin{proof}
Let $N=\tilde{N}\circ\pi_{\conv(D)}$ be nonexpansive. By \cref{prop:aediff}, $\|\nabla N(x)\|\leq 1$ for every $x\in\R^d$ where $N$ is differentiable. In particular, for every $t=1,\dots, \, \ell$ we can take a point $x$ in the interior of $S_t$ and observe that the derivative of $N$ at $x$ coincides with the derivative of $\tilde{N}$ at $x$, which is $\nabla\tilde{N}(x)=B_tA_t^{-1}$ by the characterization \cref{eq:piecewise}. This establishes $\|B_tA_t^{-1}\|\leq 1$ for each $t=1,\ldots, \ell$.

Conversely, assume that $\|B_tA_t^{-1}\|\leq 1$ for each $t=1,\ldots, \ell$ and recall that projections are nonexpansive operators. Therefore, if we prove that $\tilde{N}$ is nonexpansive, then the operator $N=\tilde{N}\circ\pi_{\conv(D)}$ will also be nonexpansive. Let $x, x'\in\conv(D)$ and suppose first that  $x, x'\in \mathrm{int}(S_t)$ for some $t=1,\ldots,\ell$. We consider the path $\gamma:[0,1]\to\R^d$; $\gamma(s)=(1-s)x+sx'$ and observe that the function $\tilde{N}\circ\gamma$ is differentiable since it is the composition of a segment and  $\tilde{N}$, which is an affine function in the image of $\gamma$. Therefore, by the fundamental theorem of calculus, we get that
\begin{align*}
\|\tilde{N}(x')-\tilde{N}(x)\|&=\|(\tilde{N}\circ\gamma)(1)-(\tilde{N}\circ\gamma)(0)\|=\left\|\int_0^1 (\tilde{N}\circ\gamma)'(s) \, \mathrm{d}s\right\|\\
&\leq \int_0^1\|(\tilde{N}\circ\gamma)'(s)\| \, \mathrm{d}s=\int_0^1\|\nabla\tilde{N}(\gamma(s))(x-x')\| \, \mathrm{d}s\\
&\leq \|B_tA_t^{-1}\|\|x-x'\|\leq \|x'-x\|,
\end{align*}
where we used that $\nabla \tilde{N}(\gamma(s))=B_tA_t^{-1}$ for every $s\in[0,1]$ since \cref{eq:piecewise} holds. Since the interior of $S_t$ is dense in $S_t$ and $\tilde{N}$ is continuous, one immediately obtains
$$
\|\tilde{N}(x')-\tilde{N}(x)\|\leq \|x'-x\|
$$ 
for $x,x'\in S_t$. Now, if $x\in S_t$, $x'\in S_{t'}$ for $t\neq t'$, we consider again the segment connecting both points $\gamma:[0,1]\to\R^d$ ; $\gamma(s)=(1-s)x+sx'$, and observe that $\gamma([0, 1])\subset\conv(D)$ since, by definition, $\mathfrak{T}$ partitions $\conv(D)$. Moreover, the segment intersects a finite number of distinct simplices. We can therefore find $k\in\mathbb{N}$ and $0=s_0<...<s_k=1$ such that for all $i=0,\ldots,k-1$, $\gamma(s_i)$ and $\gamma(s_{i+1})$ belong to the same simplex. By applying the result we have just obtained, we get that
\begin{align*}
\|\tilde{N}(x')-\tilde{N}(x)\|\leq \sum_{i=0}^{k-1}\|\tilde N(\gamma(s_{i+1}))-\tilde N(\gamma(s_i))\|
\leq \sum_{i=0}^{k-1} \|\gamma(s_{i+1})-\gamma(s_i)\|
=\|x'-x\|.
\end{align*}
Hence, $\tilde{N}$ is nonexpansive.
\end{proof}

Next, we aim at tackling \cref{eq:PAP} from a computational point of view. Since we have assumed that $x_i=\bar{x}_i$ for $i=1,\dots,n$, we arrive at the following problem: 
\begin{equation}\label{eq:FP}
\tag{FP}
\begin{aligned}
    \widehat{Y}\in\argmin_{y_1,\ldots, \, y_m\in\R^d} & \quad \frac{1}{n}\sum_{i=1}^n\|y_i-\bar{y}_i\|^2\\ 
    & \text{ s.t. } \|B_t(y_1,\ldots,y_m)A_t^{-1}\|\leq 1, \text{ for every } t\in\{1,\dots,\ell\},
\end{aligned}
\end{equation}
where  the matrices $A_t$, $B_t(y_1,\ldots,y_m)$, $t=1,\ldots, \, \ell$, are defined as in \cref{eq:matrices}. Note that here we introduce the notation $B_t(y_1,\ldots,y_m)$ to recall the fact that $B_t$ depends on the variables $y_i$, while $A_t$ is fixed. Observe that only the first $n$ vectors $y_i$ are taken into consideration in the minimization, while all the $m$ points are required to satisfy the constraint. In the experiments of \cref{exps} we will consider $m=n$, while in the next section, it will be crucial to let $m$ vary. A direct consequence of \cref{thm:1lip} is that problems \cref{eq:FP} and \cref{eq:PAP} are equivalent.

\subsection{A density result}\label{sec:density}
The last part of this section will consist in proving the convergence of \cref{eq:PAP} to \cref{eq:DP}. In short, we aim to show that, when increasing the number of points in the triangulation, piecewise affine solutions of the corresponding \cref{eq:PAP} approximate better solutions of \cref{eq:DP}. To do so, we adapt the definition given in \cite[Definition 2.2]{hannu2014} and recall some preliminary results. 
\begin{definition}\label{def:vanishing_family} Let $\Omega \subset \mathbb{R}^d$ be a bounded polytope. We say that $\mathfrak{F}=(\mathfrak{T}_k)_{k\in\mathbb{N}}$ is a vanishing family of simplicial partitions for $\Omega$ if, for every $k\in\mathbb{N}$, the longest edge of $\mathfrak{T}_k$ is of length at most $1/k$. In addition, we say that $\mathfrak{F}$ is strongly regular if there exists a constant $c>0$ such that for all partitions $\mathfrak{T}_k$ and all simplices $S\in \mathfrak{T}_k$ we have
$$
\mathrm{meas}_d\, S \geq c k^{-d}.
$$
\end{definition}

In the sequel, the bounded polytope $\Omega$ that we will consider is $\mathrm{conv}(D)$. We are ready to present the first preliminary lemma.
\begin{lemma}\label{lem:bound_invA} Let $\mathfrak{F}=(\mathfrak{T}_k)_{k\in\mathbb{N}}$ be a strongly regular vanishing family of simplicial partitions for $\mathrm{conv}(D)$. Then, there exists a constant $c'>0$ such that, for every partition $\mathfrak{T}_k$ in $\mathfrak{F}$ and all simplices $S_t\in \mathfrak{T}_k$, $t=1,\ldots, \, \ell(k)$, $\ell(k)\in\mathbb{N}$, if we denote by $A_t$ the matrix  relative to the simplex $S_t$ defined in \cref{eq:matrices}, we have
$$
\|A_t^{-1}\|\leq \frac{k}{c'}.
$$
\end{lemma}
\begin{proof}
Fix $t$ and consider the singular values of the matrix $A_t$, $\sigma(A_t)=(\sigma_i)_{i=1}^d$, $\sigma_1\geq\ldots\geq\sigma_d>0$, and observe that for every $i=1,\ldots,d$, 
$$
\sigma_i\leq\sigma_{1}:=\|A_t\|\leq\|A_t\|_{\text{F}},
$$ 
where, $\|\cdot\|_{\mathrm{F}}$ denotes the Frobenius norm, i.e.,
$$
\|A_t\|^2_{\text{F}}=\sum_{j,h=1}^d a_{jh}^2=\sum_{j=1}^d\|a_j\|^2\leq \frac{d}{k^2},
$$
with $a_{j}=(x_{i_j}-x_{i_0})$, $a_{jh}=(x_{i_j}-x_{i_0})_h$, $j,h=1,\ldots,d$, where $(x_{i_j})_{j=0}^d$ denote the vertices of $S_t$. Then, for every $i=1,\ldots,d$,
\begin{equation}
    \sigma_i\leq\frac{\sqrt{d}}{k}.
\end{equation}
Now, let us recall that $(1/d!)|\det A_t|=\mathrm{meas}_d\, S_t$. Thus, by strong regularity,
$$
\prod_{i=1}^d\sigma_i=|\det A_t|\geq d! (\mathrm{meas}_d\, S_t)\geq ck^{-d},
$$ 
for some $c>0$. Consequently, for every $j=1,\dots, d$, 
$$
\sigma_j=\frac{\prod_{i=1}^d\sigma_i}{\prod_{i\neq j}^d\sigma_i}\geq \frac{c k^{-d}}{(\sqrt{d}/k)^{d-1}}\geq \frac{c'}{k},
$$
for some $c'>0$. Since $\|A_t^{-1}\|=\sigma_{d}^{-1}$, we conclude.
\end{proof}

From now on, we assume that, for every $k\in\mathbb{N}$, the points $(x_i^k)_{i=1}^{n_k}$, $n_k\in\mathbb{N}$, forming the simplicial partition $\mathfrak{T}_k$ of $\conv(D)$, contain $\bar D$; i.e., without loss of generality, $x_i^k=\bx_i$ for every $i=1,\ldots, n$.
 
 As it will be seen, we prove that minimizers of $\widehat{F}$ on the set $\mathrm{PA}(\mathfrak{T}_k)\cap\mathcal{N}$ converge, in some sense, to a minimizer of \cref{eq:DP} as $k$ goes to infinity. We recall that the function $\widehat{F}$ remains the same for every $k\in\mathbb{N}$. In other words, although we increase the number of points at every iteration, penalization is done only for the initial training points. Therefore, once a training set is fixed, we are able to prove that our constructed operators converge to a minimizer of \cref{eq:DP}. The following lemma is key in this result.
\begin{lemma}\label{lem:311}
Let $\mathfrak{F}=(\mathfrak{T}_k)_{k\in\mathbb{N}}$ be a vanishing family of simplicial partitions for $\conv(D)$ and define, for every $k\in\mathbb{N}$,
\begin{equation}\label{eq:Lk*bounded}
\widehat{L}_k\in\argmin_{N\in \mathrm{PA}(\mathfrak{T}_k)\cap\mathcal{N}}\widehat{F}(N),
\end{equation}
where $\widehat{F}$ is defined as in \cref{eq:DP}. Then, the sequence $(\widehat{L}_k)_{k\in\mathbb{N}}$ is bounded in $\Lip(\R^d, \R^d)$ and there exists a subsequence $(\widehat{L}_{k_j})_{j\in\mathbb{N}}$ and an operator $\widehat{L}\in \mathcal{N}$ such that $\widehat{L}_{k_j}\weakk \widehat{L}$ as $j \to \infty$.
\end{lemma}
\begin{proof}
For every $k\in\mathbb{N}$, let $\widehat{L}_k$ be defined as in \cref{eq:Lk*bounded}; i.e., a minimizer of $\widehat{F}$ in $\mathrm{PA}(\mathfrak{T}_k)\cap\mathcal{N}$. Clearly, $\widehat{F}(\widehat{L}_{k})\leq \widehat{F}(\pi_{\conv (D)})$ for every $k\in\mathbb{N}$, where indeed, $\pi_{\conv(D)}\in \mathrm{PA}(\mathfrak{T}_k)\cap\N$ for every $k\in\mathbb{N}$. Hence, $\widehat{F}$ is bounded on $(L_k)_{k\in\mathbb{N}}$ and, by coercivity of $\widehat{F}$ in the sense of \cref{eq:coerciv} and \cref{prop:experr} (applied to the empirical measure $\hat{\mu}:=(1/n)\sum_{i=1}^n\delta_{(\bx_i, \by_i)}$), we can conclude, as in the proof of \cref{thm:existence}, that the sequence $(L_k)_{k\in\mathbb{N}}$ is bounded in $\Lip(\R^d, \R^d)$. By weak$^*$ precompactness, there exists a subsequence $(L_{k_j})_{j\in\mathbb{N}}$ and $\widehat{L}\in\Lip(\R^d,\R^d)$ with $\widehat{L}_{k_j}\weakk \widehat{L}$ as $j\to\infty$. \cref{nnzero} finally gives $\widehat{L}\in\N$ as claimed.
\end{proof}

We are now ready to present the main result of the section.
\begin{theorem}\label{thm:density} Let $\mathfrak{F}=(\mathfrak{T}_k)_{k\in \mathbb{N}}$ be a strongly regular vanishing family of simplicial partitions for $\mathrm{conv}(D)$. Then, for every sequence of operators defined as in \cref{eq:Lk*bounded}, there exists a subsequence $(\widehat{L}_{k_j})_{j\in\mathbb{N}}$ such that
\[ \widehat{L}_{k_j}\weakk \widehat{L} \ \text{as} \ j \to \infty \ \ \text{ with } \ \ \widehat{L}\in\argmin_{N\in\mathcal{N}}\widehat{F}(N).\]
\end{theorem}
\begin{proof}
For every $k\in\mathbb{N}$, let $(x_i^k)_{i=1}^{n_k}$ be the vertices of all the simplices of the partition $\mathfrak{T}_k$ and recall that $\bar x_i=x_i^k$ for each $i=1,\ldots, n$ and every $k\in\mathbb{N}$. Let $\widehat{N}$ be a minimizer of \cref{eq:DP}, $\delta > 0$ and define the operator $N^\delta:=\widehat{N}\ast G_\delta$, where $G_\delta:\R^d\to\R$ is a mollifier; i.e., defined for every $x\in\R^d$ as
$$
G_\delta(x):=\frac{1}{\delta^d}G\left(\frac{x}{\delta}\right), \quad G\in \mathcal{C}^\infty(\R^d),\quad \mathrm{supp} \, G\subset B_1(0),\quad G\geq 0, \quad \int_{\R^d} G(x) \, \mathrm{d}x=1.
$$
By \cref{prop:aediff}, $\widehat{N}$ is almost everywhere differentiable and thus, $\|\nabla N^\delta\|_\infty\leq 1$. Indeed, by $G_\delta \geq 0$ and $\| G_\delta \|_1 = 1$,
$$
\begin{aligned}
\|\nabla N^\delta\|_\infty &=\sup_{x\in\R^d} \left\|\int_{\R^d}\nabla \widehat{N}(x-y) G_\delta(y) \, \mathrm{d}y \right\| \\
&\leq \sup_{x\in\R^d}\int_{\R^d} \left\| \nabla \widehat{N}(x-y) \right\| G_\delta(y) \, \mathrm{d}y\\
&\leq \|\nabla \widehat{N}\|_\infty\|G_\delta\|_1\leq 1.
\end{aligned}
$$
Similarly, we have $\|N^\delta - \widehat{N}\|_\infty \leq \delta$ since for each $x \in \mathbb{R}^d$, we can estimate 
\[
\|N^\delta(x) - \widehat{N}(x)\| \leq \int_{\mathbb{R}^d} \|\widehat{N}(x - y) - \widehat{N}(x)\| G_\delta(y) \, \mathrm{d}y
\leq \delta \|G_\delta\|_1 \leq \delta
\]
as $\mathcal{L}(\widehat{N}) \leq 1$ and $G_\delta$ is supported on $B_\delta(0)$.

Additionally, since $\nabla N^\delta$ is uniformly continuous on $\conv (D)$, there exists $\varepsilon>0$ such that
\begin{equation}\label{eq:312unifcont}
\|N^\delta(x)-N^\delta(x')-\nabla N^\delta(x')(x-x')\|\leq \delta\|x-x'\|
\end{equation}
for every $x, x'\in\conv(D)$ with $\|x-x'\|\leq \varepsilon$. Next, for each $k\in\mathbb{N}$, let  $L_k^\delta$ be the unique mapping in $\mathrm{PA}(\mathfrak{T}_k)$ such that $L_k^\delta(x_i^k)=N^\delta(x^k_i)$ for $i=1,\ldots, n_k$. For each $S_t\in\mathfrak{T}_k$ we denote by $(x^k_{i(t,j)})_{j=0}^d$ its vertices, and recall that the matrices $A_t$ and $B_t$ are given by
$$
A_t=[x^k_{i(t,1)}-x^k_{i(t,0)}\mid\dots \mid x^k_{i(t,d)}-x^k_{i(t,0)}],
$$
and 
$$
B_t=[N^\delta(x^k_{i(t,1)})-N^\delta(x^k_{i(t,0)})\mid\ldots \mid N^\delta(x^k_{i(t,d)})-N^\delta(x^k_{i(t,0)})].
$$
By assumption, $\|x^k_{i(t,j)}-x^k_{i(t,0)}\|\leq 1/k$ for each $S_t\in\mathfrak{T}_k$ and $j=1,\dots d$. Therefore, by choosing $k\in\mathbb{N}$ large enough so that $(1/k)\leq\varepsilon$, \cref{eq:312unifcont} gives
\begin{equation*}
\|N^\delta(x^k_{i(t,j)})-N^\delta(x^k_{i(t,0)})-\nabla N^\delta(x^k_{i(t,0)})(x^k_{i(t,j)}-x^k_{i(t,0)})\|\leq \delta\|x^k_{i(t,j)}-x^k_{i(t,0)}\| \leq \frac{\delta}{k},
\end{equation*}
for $j=1,\ldots, d$, and so
\begin{multline*}
\|B_t-\nabla N^\delta(x^k_{i(t,0)})A_t\| \leq 
\|B_t-\nabla N^\delta(x^k_{i(t,0)})A_t\|_{\mathrm{F}} \\
= \Biggl( \sum_{j=1}^d \| N^\delta(x^k_{i(t,j)})-N^\delta(x^k_{i(t,0)})-\nabla N^\delta(x^k_{i(t,0)})(x^k_{i(t,j)}-x^k_{i(t,0)})  \|^2 \Biggr)^{1/2} \leq
\frac{\sqrt{d}\delta}{k}.    
\end{multline*}
Consequently, by \cref{lem:bound_invA}, we obtain
$$
\begin{aligned}
\|B_tA_t^{-1}\|&\leq \|\nabla N^\delta(x^k_{i(t,0)})\|+\|B_tA_t^{-1}- \nabla N^\delta(x^k_{i(t,0)})\|\\
&\leq 1+ \|B_t - \nabla N^\delta(x_{i(t,0)}) A_t \| \| A_t^{-1} \| \leq 1+c\delta,
\end{aligned}
$$
for a suitable constant $c>0$. Since $\|B_tA_t^{-1}\|$ is
the norm of the gradient of $L_k^\delta$ in each simplex $S_t$, setting
$$
\widetilde{L}^\delta_k=\frac{1}{1+c\delta}L^\delta_k,
$$
yields $\widetilde{L}^\delta_k\in\mathrm{PA}(\mathfrak{T}_k)\cap \N$; i.e., $\widetilde{L}^\delta_k$ is feasible for \cref{eq:Lk*bounded}. By recalling that, for every $i=1,\ldots, n$, it holds $L^\delta_k(\bar x_i)=N^\delta(\bar x_i)$ and $\|N^\delta-\widehat{N}\|_\infty\leq \delta$, we have
$$
\begin{aligned}
\|\widetilde{L}_k^\delta(\bar x_i)-\widehat{N}(\bar x_i)\|&\leq \|\widetilde{L}_k^\delta(\bar x_i)- L_k^\delta(\bar x_i)\|+\|N^\delta(\bar x_i)-\widehat{N}(\bar x_i)\|\\
&\leq c\delta \bigl( \| \widehat{N}(\bar x_i) \| + \|N^\delta(\bar x_i) - \widehat{N}(\bar x_i) \| \bigr) + \delta\\
&\leq (1+c M+c\delta)\delta, 
\end{aligned}
$$
for $i=1,\ldots, n$, where $M:= \sup_{x \in \conv(D)} \| \widehat{N}(x) \| < \infty$. If $\widehat{L}_k$ minimizes \cref{eq:Lk*bounded} for each $k\in\mathbb{N}$, it follows that
\begin{align*}
\widehat{F}(\widehat{L}_k)&\leq \widehat{F}(\widetilde{L}_k^\delta)\leq \frac1n\sum_{i=1}^n(1+\delta)\|\widehat{N}(\bar x_i)-\bar y_i\|^2+\left(1+\frac1\delta\right)\|\widetilde{L}_k^\delta(\bar x_i)-\widehat{N}(\bar x_i)\|^2\\
&\leq (1+\delta)\widehat{F}(\widehat{N})+\left(\delta^2+\delta\right)(1+c M+c\delta)^2,
\end{align*}
where we have used Young's inequality and the definition of $\widehat{F}$. As $\delta>0$ was arbitrary, we can conclude that
$$
\limsup_{k\to\infty} \widehat{F}(\widehat{L}_k)\leq \widehat{F}(\widehat{N}).
$$
By \cref{lem:311}, there exists a subsequence  $(\widehat{L}_{k_j})_{j\in\mathbb{N}}$ of $(\widehat{L}_k)_{k\in\mathbb{N}}$ such that $\widehat{L}_{k_j}\weakk \widehat{L}$ as $j\to\infty$ for some $\widehat{L} \in \N$. The weak$^*$ lower semi-continuity of $\widehat{F}$ then gives
$$
\widehat{F}(\widehat{L})\leq \liminf_{j\to\infty}\widehat{F}(\widehat{L}_{k_j})\leq \limsup_{k\to\infty}\widehat{F}(\widehat{L}_k)\leq \widehat{F}(\widehat{N}).
$$
As we have shown that $\widehat{L}$ is a minimizer of \cref{eq:DP}, we conclude.
\end{proof}

\begin{remark}\label{rem:simpart}
Notice that the above result holds for any dimension $d\geq 1$, and for any family of simplicial partitions $(\mathfrak{T}_k)_{k\in \mathbb{N}}$ that is strongly regular and vanishing. However, constructing a strongly regular vanishing family of simplicial partitions is not a trivial task. In the case of $d=2$, there are examples of such families based on refinements. In particular, one can start with a simplicial partition and use the face-to-face longest-edge bisection algorithm \cite{krizek}, which was proven to provide a strongly regular family of nested simplicial partitions for which the length of longest edge converges to zero. An appropriately chosen subfamily is then vanishing in the sense of \cref{def:vanishing_family}. Whether this technique gives strongly regular families in higher dimensions is still an open problem. However, it seems to perform well in practice and numerical experiments have been already provided for such issue, see for example \cite{hannu2014}. 
\end{remark}

\section{Experiments}\label{exps}
We conclude this work with some experiments. First, by making use of the so-called Moreau's identity, we introduce a novel, interpretable, Plug-and-Play Chambolle--Pock primal-dual iteration, which will be further used in the subsequent numerics. Next, we provide a detailed explanation on how to construct piecewise affine firmly nonexpansive operators in practice, and its further application in the context of PnP methods. Finally, we demonstrate its effectiveness in an image denoising task with numerical simulations.

\subsection{Convergent PnP Chambolle--Pock primal-dual iteration}\label{sec:pnpcp}
Let us consider the problem $\min_x f(x)+R(Lx)$, where $L:\cX\to\cX'$ is a linear operator mapping to a real Hilbert space $\cX'$, with adjoint operator denoted as $L^*$, $f\in \Gamma_0(\cX)$, and $R\in \Gamma_0(\cX')$. Under mild additional assumptions \cite[Theorem 16.47]{BCombettes}, the corresponding monotone inclusion problem becomes finding $x\in\cX$ such that $0\in \partial f (x) + L^* \partial R(Lx)$, and in general, finding $x\in\cX$ such that 
\begin{equation}\label{eq:structured_problem}
    0\in A_1x + L^* A_2 Lx,
\end{equation}
with $A_1, A_2$ maximal monotone operators defined on $\cX$ and $\cX'$ respectively. In order to solve such a problem, one can employ the so-called Chambolle--Pock primal-dual iteration \cite{Chambolle2011}, whose iterations in the context of maximal monotone inclusions read
\begin{equation*}
    \begin{aligned}
    x^{k+1} & = J_{\tau A_1}( x^k-\tau L^* y^{k}),\\
    y^{k+1} & = J_{\sigma A_2^{-1}}\left(y^k + \sigma L(2x^{k+1}-x^k)\right),
    \end{aligned}
\end{equation*}
for $\tau, \sigma >0$ primal and dual step-sizes, respectively. Without assuming any further properties on the operators $A_1$ or $A_2$, the convergence of the method is guaranteed when $\tau \sigma \|L\|^2\leq 1$ (see \cite{Chambolle2011}, or \cite{bredies2021degenerate} for the edge case). In particular, the sequence $(x^k)_{k\in \mathbb{N}}$ generated by the algorithm weakly converges to a solution of \cref{eq:structured_problem}. 

Now, there are three main options for finding the corresponding PnP version of the above algorithm. The first option, investigated in \cite{Santos2024} for a specific maximal monotone operator $A_2$ instead of $\partial R$, is to replace the resolvent of the maximal monotone operator $u-\hat{u}_\delta$ (which in this case is the identity) with a firmly nonexpansive operator (see e.g. \cite{Suzuki2024}). The second option is to substitute $J_{\sigma (\partial R)^{-1}}$ instead, as done for instance in \cite{Santos2024}. In this case, however, the role of the parameter $\sigma$ remains unclear, and so, there is a lack of interpretability for the final solution $u$ of the resulting algorithm. A natural way to circumvent this issue is to make use of the so-called Moreau's identity, which allows us to rewrite the resolvent of $(\partial R)^{-1}$ in terms of the resolvent of $\partial R$, and gives, for every $x\in\cX$,
$$
J_{\sigma A_2^{-1}}(x)=\sigma (\mathrm{Id}-J_{\sigma^{-1}A_2})( \sigma^{-1}x).
$$
With this, it is possible to write the iterations of the algorithm as
\begin{equation}\label{eq:Moreau_CP}
    \begin{aligned}
    x^{k+1} & = J_{\tau A_1}( x^k-\tau L^* y^{k}),\\
    y^{k+1} & = \sigma (\mathrm{Id}- J_{\sigma^{-1} A_2})\left(\sigma^{-1}y^k + L(2x^{k+1}-x^k)\right).
    \end{aligned}
\end{equation}
We are now ready to present our proposed PnP Chambolle--Pock primal-dual iteration, which is obtainted by substituting the resolvent operator $J_{\sigma^{-1} A_2}$ with a firmly nonexpansive operator $T$:
\begin{equation}\label{eq:pnpcp}
\tag{PnP-CP}
    \begin{aligned}
    x^{k+1} & = J_{\tau A_1}( x^k-\tau L^* y^{k}),\\
    y^{k+1} & = \sigma (\mathrm{Id}- T)\left(\sigma^{-1}y^k + L(2x^{k+1}-x^k)\right).
    \end{aligned}
\end{equation}
In this case, if $\tau \sigma \|L\|^2\leq 1$, the sequence $(x_k)_{k\in\mathbb{N}}$ weakly converges to a solution of the monotone inclusion problem $0\in A_1 x + L^* A_T L x$, where $A_T$ is defined via the identity $T=(\mathrm{Id}+\sigma^{-1}A_T)^{-1}$. With this, we obtain the following underlying problem
\begin{equation}\label{eq:underlying_problem_CP} 0\in A_1 x +  L^* A_T L x= A_1 x + \sigma L^* (T^{-1}-\mathrm{Id}) L x.\end{equation}
In applications, the operator $T$ is learned from data, which means that the operator $(T^{-1}-\mathrm{Id})$ can be regarded as fixed. Hence, changing $\sigma$ does change the underlying problem and, thanks to the explicit expression in \cref{eq:underlying_problem_CP}, we can think of $\sigma$ as a regularization parameter (see also \cref{sec:image_denoising}). This was only possible by making use of Moreau's identity, which provides us with more interpretability of the solution than what we could have achieved by only substituting $J_{\sigma A_2^{-1}}$ directly. Moreover, constructing $T$ requires a training set tailored to the specific problem and the specific PnP algorithm of choice. A general strategy is to examine and use the fixed point conditions of the algorithm, or equivalently, the optimality conditions of the underlying problem. In the next section, we detail this process for a particular choice of the operator $A_1$.

In this section, we have presented our novel Plug-and-Play Chambolle--Pock primal-dual iteration. Following the approach proposed in \cref{sec:piecewise},  the learned operator $T$ to be further plugged into the method will be ensured to be firmly nonexpansive, thereby ensuring convergence of the algorithm. Notice, however, that our proposed methodology can be applied to any algorithm involving the resolvent of a maximal monotone operator, such as the Douglas--Rashford iteration (or, equivalently, the alternating directions method of multipliers (ADMM)), the Forward-Backward splitting method, or other primal-dual approaches, as described in the introduction. Other PnP adaptations of classical iterative methods can be found in \cite{pescterr2021,ryupnp2019,Santos2024,Sun2021,Suzuki2024,venkat2013}, among others. 

\subsection{Learning piecewise affine firmly nonexpansive operators}\label{sec:learningPA} 
In applications, given a set of data points $\{(\bar x_i,\bar z_i)\}_{i=1}^n$, we are interested in finding the best firmly nonexpansive operator that approximates these points in terms of the least squares distance. To do so, as explained in \cref{sec:charac} we can focus on learning a nonexpansive operator. We first define the points $\bar y_i=2\bar z_i-\bar x_i$. Then fix a triangulation $\mathfrak{T}$ for $D=\bar D =\{\bar x_i\}_{i=1}^n$ (as explained in \cref{triangs}), and solve the following problem
\begin{equation}\label{eq:learning_task}
\begin{aligned}
    \min_{y_1,\dots y_n\in\mathbb{R}^d} & \frac{1}{n}\sum_{i=1}^n\|y_i-\bar{y}_i\|^2\\ 
    & \text{ s.t. } \|B_t(y_1,\dots,y_n)A_t^{-1}\|\leq 1, \text{ for every } S_t\in \mathfrak{T},
\end{aligned}
\end{equation}
where $A_t$ and $B_t$ are defined as in \cref{sec:piecewise}, with $m=n$. This problem can be written in a more synthetic form as follows
\begin{equation}\label{eq:111}
\begin{aligned}
    \min_{Y\in \mathbb{R}^{d\times n}} \frac{1}{n}\|Y-\bar{Y}\|_{\mathrm{F}}^2 + \sum_{t=1}^{\ell}\iota_{\{\|B_t(\cdot)A_t^{-1}\|\leq 1\}}(Y),
\end{aligned}
\end{equation}
where $\bar Y = [\bar y_1 \mid \dots \mid \bar y_n] \in \mathbb{R}^{d \times n}$ and $\|\cdot\|_{\mathrm{F}}$ denotes again the Frobenius norm. We introduce the functions $f:\mathbb{R}^{d\times n}\to  \mathbb{R}\cup \{+\infty\}$ and $g:\mathbb{R}^{d\times d}\to  \mathbb{R}\cup \{+\infty\}$, defined by
$$
f(Y)=\frac{1}{n}\|Y-\bar Y\|_{\mathrm{F}}^2, \quad g(M)=\iota_{\{\|\cdot\|\leq 1\}}(M).
$$
Moreover, denote by $L_t:\mathbb{R}^{d\times n}\to\mathbb{R}^{d\times d}$ the linear operators that map $Y$ to $B_t(Y)A_t^{-1}$ for each $t=1,\dots, \ell$. Combining all of the above, we rewrite \cref{eq:111}
as follows:
\begin{equation}\label{Problem_for_CP}
\begin{aligned}
    \min_{\substack{Y\in \mathbb{R}^{d\times n}\\ U_1,\ldots, U_\ell\in \mathbb{R}^{d\times d} }} f(Y) + \sum_{t=1}^{\ell}g(U_t) \quad \text{ s.t. } U_t=L_t Y, \quad \text{for } t=1,\dots, \ell.
\end{aligned}
\end{equation}
In order to solve it, we make use of the ADMM algorithm \cite{admm} indicated in the following. Note, however, that the algorithm we use to solve problem \cref{Problem_for_CP} is not relevant, and one can make use of other suited iterative methods.

\paragraph{ADMM algorithm used for learning}
Introducing the collection $U = (U_1,\dots,U_\ell)$ and the linear operator 
$\mathbf{L}: \mathbb{R}^{d\times n} \to \mathbb{R}^{d\times d\times \ell}$, 
$\mathbf{L}(Y) = (L_1 Y,\dots,L_\ell Y)$, 
problem \eqref{Problem_for_CP} can be written compactly as
\[
    \min_{Y,U} \; f(Y) + g(U) 
    \quad \text{s.t. } \; U = \mathbf{L}(Y),
\]
where, with abuse of notation we have denoted $g(U) = \sum_{t=1}^\ell g(U_t)$.

Given a sequence of parameters $\rho_k>0$ and denoting the dual variable by $\Lambda\in\mathbb{R}^{d\times d\times \ell}$, the ADMM iterations read
\begin{equation}\label{eq:ADMM-scaled}
    \begin{aligned}
    Y^{k+1}
        &= \argmin_{Y}\;
            f(Y) + \frac{\rho_k}{2}\,\big\|\mathbf{L}(Y)-U^{k}+\Lambda_k/\rho_k\big\|_{\mathrm F}^{2}, \\[0.5em]
    U^{k+1}
        &= \argmin_{U}\;
            g(U) + \frac{\rho_k}{2}\,\big\|\mathbf{L}(Y^{k+1})-U+\Lambda_k/\rho_k\big\|_{\mathrm F}^{2}, \\[0.5em]
    \Lambda^{k+1}
        &= \Lambda^{k} \;+\; \rho_k\big(\mathbf{L}(Y^{k+1})-U^{k+1}\big),
    \end{aligned}
\end{equation}
where we adopt the following choice for the parameters $\rho_k$. We start with $\rho_0=0.01$ and compute
\begin{equation}\label{eq:adaptive-rho}
    \rho_{k+1} = \min\left\{\left(1+\frac{1}{k^{1.1}}\right)\rho_k, \ 100\right\}
\end{equation}
In this way, we assure that the sequence of penalty parameters $\rho_k>0$ is such that $\rho_k \to \rho > 0 $ and $\sum_k |\rho_{k+1}-\rho_k|< +\infty$, and since the ADMM method is equivalent to the Douglas-Rachford method applied to the dual problem (see for example \cite{EcksteinBertsekas_DR}) we still have convergence of the algorithm (see \cite{Lorenz2024, Lorenz2019}).

Since the $Y$ update does not admit closed-form solutions, it is solved approximately, while the $U$ update is exact:

\begin{itemize}
    \item \textbf{$Y$-update:} we start from $Y^{0}_{\text{temp}}=Y^k$ and apply gradient descent for $i=0,1,\dots$, with step size 
    $\sigma = \tfrac{1}{1+\rho\|\mathbf{L}\|^2}$, 
    \[
        Y_{\text{temp}}^{i+1} \;=\; Y_{\text{temp}}^{i} - \sigma \Big( \nabla f(Y_{\text{temp}}^{i}) 
            + \mathbf{L}^*\!\big( \rho_k(\mathbf{L}(Y_{\text{temp}}^{i}) - U^k) + \Lambda^k \big) \Big).
    \]
    Since the subproblem is strongly convex, the convergence is quite fast and we adopt as stopping criterion $i = 10$ or the relative residual $\|Y_{\text{temp}}^{i+1}-Y_{\text{temp}}^i\|/\|Y_{\text{temp}}^i\|$ to be smaller than $10^{-2}$. After stopping, we then set $Y^{k+1} = Y_{\text{temp}}^{i+1}$.
    
    \item \textbf{$U$-update:} we can directly project and compute
    \[
        U^{k+1} = \mathrm{proj}_{\mathcal{C}}( \mathbf{L}(Y^{k+1}) + \Lambda^k/\rho_k ),
    \]
    where $\mathcal{C}=\{W = (W_1,\dots,W_\ell) \in \R^{d\times d \times \ell} \mid \|W_t\|\leq 1, \ t = 1,\dots,\ell\}$. We compute the projection using singular-value decomposition and clipping; although this is computationally heavy in high dimensions, it is efficient for $d=2$, which is the setting we reduce to in our experiments.
\end{itemize}

\begin{remark}
Since the learned operator is obtained through the iterative optimization procedure explained above, the Lipschitz constraint is not necessarily satisfied at every intermediate step. To ensure that the final iterate is $1$-Lipschitz, we slightly modify the procedure in the experiments: we replace the Lipschitz constant by the value $0.99$. With this adjustment, the operator produced at the last iterate when we stop the algorithm is actually $1$-Lipschitz.
\end{remark}

\subsection{Image denoising}\label{sec:image_denoising} In the following, we focus on the image denoising problem, which consists of having access to an image that has been corrupted by some noise (which we suppose to be additive and Gaussian) and recovering a denoised version of it.
Specifically, given a clean image $\hat u\in \mathbb{R}^{p\times q}$, we assume to have a corrupted image $\hat u_\delta=\hat u + \delta$, where $\delta \sim N(0,\eta^2 \Id)$, $\eta>0$. In order to reconstruct the clean image (or a good approximation of it) we first consider the minimization problem
\begin{equation}\label{eq:image_denoise}
    \min_{u\in\R^{p\times q}} \frac{1}{2}\|u-\hat u_\delta\|_{\mathrm{F}}^2+ R(Du),
\end{equation}
where $D:\mathbb{R}^{p\times q}\to \mathbb{R}^{p\times q\times 2}$ is a discretized version of the gradient. The function $R:\mathbb{R}^{p\times q\times 2} \to \mathbb{R}$, supposed to be convex, determines the penalization imposed on the (discrete) gradient of the image. In practice, it is difficult to choose this function since it is problem-dependent: it should entail some prior knowledge of the image we want to reconstruct. Some classical examples are
\begin{equation}\label{eq:norm11norm21}
\begin{aligned}
R(v) & =\frac{\alpha}{2}\|v\|_{\mathrm{F}}^2,\\
R(v) & =\alpha\|v\|_{1,1}=\sum_{i=1}^{p}\sum_{j=1}^{q}\alpha\|v_{ij}\|_1,\quad \\
R(v) & =\alpha\|v\|_{2,1}=\sum_{i=1}^{p}\sum_{j=1}^{q}\alpha\|v_{ij}\|_2,
\end{aligned}
    \end{equation}
where $\alpha>0$ denotes the regularization parameter. Using the first choice in \cref{eq:norm11norm21} we end up with an $H^1$ penalty, while the second choice is the so-called anisotropic total variation (anisotropic TV), and the third one corresponds to isotropic TV.

Observe that all of the choices for the function $R$ above presented are separable with respect to the pixels and, moreover, the function is the same in every pixel (in the examples, respectively, $\frac{\alpha}{2}\|\cdot\|_2^2, \ \alpha\|\cdot\|_1$ and $\alpha\|\cdot\|_2$). This motivates us to assume that $R$ is of the form 
\[
R(v)=\sum_{i=1}^{p}\sum_{j=1}^{q}r(v_{ij}), \quad \text{with $r:\mathbb{R}^2\to \mathbb{R}$.}
\] 
Such observation is key in our study, since it implies that the proximal operator of $R$ is a diagonal operator of the form
\begin{equation}\label{eq:separable_prox}
\text{prox}_R = \text{diag}(\text{prox}_r, \dots, \text{prox}_r) : \, \mathbb{R}^{p\times q \times 2} \to \mathbb{R}^{p\times q \times 2},
\end{equation}
where $\text{prox}_r:\R^2 \to \R^2$ and so, in order to learn $\mathrm{prox}_R$, we just need to learn an operator $\mathrm{prox}_r$ mapping from $\R^2$ to $\R^2$. The reason for us to consider this particular setting is two-fold: first, it allows us to remain within the context of \cref{rem:simpart}; i.e, we are sure that the density result follows, as we have already pointed out that in dimension $2$ there are simplicial partitions that are strongly regular and vanishing (see \cref{rem:simpart}); and second, we avoid having to deal with the increasing computational power that is required to deal with simplicial partitions in high dimensions \cite{edelsbrunner1987}.

We recall that solving problem \cref{eq:image_denoise} is equivalent to solving 
\begin{equation}\label{eq:monotincldeno}
0\in u-\hat u_\delta + D^*\partial R (D u),
\end{equation}
and we decide to tackle this problem using the Chambolle--Pock primal-dual algorithm \cite{Chambolle2011}. Considering the maximal monotone operators defined by $A_1(u) = u-\hat u_\delta$ and $A_2(u) = \partial R (u)$, the iteration introduced in \cref{eq:Moreau_CP} reads as
\begin{equation}\label{eq:CPexpers}
    \begin{aligned}
    u^{k+1} & = (u^k-\tau D^*v^k+\tau\hat u_\delta)/(1+\tau),\\
    v^{k+1} & = \sigma (I-\text{prox}_{\sigma^{-1} R})\left(v^k/\sigma + D(2u^{k+1}-u^k)\right),
    \end{aligned}
\end{equation}
which, using \cref{eq:separable_prox}, can be written as
\begin{equation}\label{eq:Separable_CP}
    \begin{aligned}
    u^{k+1} & = (u^k-\tau D^*v^k+\tau\hat u_\delta)/(1+\tau),\\
    z^{k+1} & =  v^k+\sigma D(2u^{k+1}-u^k),\\
    \text{for } i & = 1,\dots,p, j=1,\dots,q, \text{ do }\\ 
    & v_{ij}^{k+1}=\sigma(I-\text{prox}_{\sigma^{-1}r})(z_{ij}^{k+1}/\sigma).
    \end{aligned}
\end{equation}
In this algorithm, the variable $u$ lives in $\mathbb{R}^{p\times q}$, and it is the one that converges to the denoised image. Additionally, the variables $v$ and $z$ both live in $\mathbb{R}^{p\times q\times 2}$, with components $v_{ij},z_{ij}\in\mathbb{R}^2$. We now want to make use of the PnP-CP method introduced in \cref{eq:pnpcp}. In order to do so, we assume that the firmly nonexpansive operator $T$ that has to substitute $\text{prox}_{\sigma^{-1} R}$ in \cref{eq:CPexpers} retains the same properties. In particular, we suppose that $T$ is a diagonal operator of the form
\begin{equation}\label{eq:separable_T}
T = \text{diag}(T_r, \dots, T_r) : \, \mathbb{R}^{p\times q \times 2} \to \mathbb{R}^{p\times q \times 2}.
\end{equation}
This means that we just need to learn an operator $T_r: \R^2\to \R^2$ that substitutes $\text{prox}_{\sigma^{-1} r}$ in \cref{eq:Separable_CP}, a task that is more suited for our method. The resulting PnP algorithm reads as follows
\begin{equation}\label{eq:PnP_CP}
    \begin{aligned}
    u^{k+1} & = (u^k-\tau D^*v^k+\tau\hat u_\delta)/(1+\tau),\\
    z^{k+1} & =  v^k+\sigma D(2u^{k+1}-u^k),\\
    \text{for } i & = 1,\dots,p, j=1,\dots,q, \text{ do }\\ 
    & v_{ij}^{k+1}=\sigma(\mathrm{Id}-T_r)(z_{ij}/\sigma),
    \end{aligned}
\end{equation}
where the step-sizes $\tau, \sigma>0$ are such that $\tau \sigma\leq \|D\|^2$. Notice that we used Moreau's identity to use directly $T_r$, as we explained in \cref{sec:pnpcp}.

Next, we describe how to construct the training set. We first observe that, given a noisy image $\hat u_\delta$, the algorithm \cref{eq:PnP_CP} should produce a solution that approximates the corresponding clean image $\hat u$. This motivates us to consider clean images as fixed points of the iteration \cref{eq:PnP_CP}. Hence, we assume the pair $(\hat u, v)$ to satisfy
\[
\begin{aligned}
(1+\tau)\hat u &=\hat u-\tau D^* v + \tau \hat u_\delta,\\
v & = \sigma (\mathrm{Id}-T)(v/\sigma + D \hat u),
\end{aligned}
\]
for some $v\in\R^{p\times q \times 2}$, and thus
\begin{equation}\label{eq:Du_T}
D \hat u = T(D \hat u + v/\sigma) \quad \text{ for some $v$ such that } D^* v  = \delta,
\end{equation}
where we used that $\hat u_\delta=\hat u + \delta$, where $\delta \sim N(0,\eta \mathrm{Id})$, $\eta>0$.
Thus, using \cref{eq:separable_T}, we can derive a pixelwise expression for the first equality in \cref{eq:Du_T}:
\[\begin{aligned}
(D \hat u)_{ij} = T_r((D \hat u)_{ij} + v_{ij}/\sigma) \quad i=1,\dots,p \, , \, j= 1,\dots , q.
\end{aligned}\]
This shows that the operator $T_r$ we want to learn should act as a denoiser on the discrete gradients of the original image $\hat u$. In fact, since $\delta$ is Gaussian and $D$ is a linear operator, $v/\sigma$ could also be chosen Gaussian. For simplicity, we suppose that $v/\sigma$ is distributed as $N(0,\tilde \eta ^2 \mathrm{Id})$ for some $\tilde \eta>0$. This motivates us to consider pixelwise noisy gradients of images as inputs (called $\bar x_i \in \R^2$ in \cref{sec3}) and their corresponding clean pixelwise gradients as outputs (called $\bar z_i\in \R^2$ at the beginning of \cref{sec:learningPA}). For constructing the training set, we use a noise level $\tilde \eta$ independent of $\eta$, since we want to be able to denoise images with various levels of noise. In fact, once the operator $T_r$ is learned and fixed, we use the step-size $\sigma$ as regularization parameter, as already discussed.

The training process of our proposed denoiser was done over samples from an image of a butterfly and images from the MNIST dataset, shown in \cref{figure:dataset}. This choice was made to learn an operator that allows the reconstruction of both edges and smoother parts. Notice that it is possible to choose any image sample. To achieve better performance, a good choice for training images could be using images similar to the one that has to be reconstructed, if available. Or, if the interest is more focused on the reconstruction of edges, a possibility is to choose (manually or automatically) pixels near edges or images with many edges. We show this in our experiments. All the images we consider in our experiments have pixel values between $0$ and $255$. After some analysis, we noticed that we arrived at good results also with a low number of data points (which made the learning process easy and relatively fast). We choose $1000$ pairs of data points $(\bar x_i,\bar z_i)$ and perform the learning process as described in the first part of \cref{exps}. To select these data points, we first consider the noisy gradients in every pixel from the data images \cref{figure:dataset} (the images we consider depend on the experiment we perform, see below). From these points, we then select, through some clustering process, $250$ points and then symmetrize them with respect to both axes to obtain $1000$ data points. The noisy data of gradients were generated using Gaussian distributions with fixed standard deviation $\tilde \eta =10$. As already discussed, we use the dual step-size as a regularization parameter. Tuning such parameter allows us to deal with different noise levels (in the experiments we set the standard deviation of the Gaussian noise of the test images to $10$, $20$, and $30$).

\paragraph{Code statement} All the simulations have been implemented in Matlab on a laptop with Intel Core i7 1165G7 CPU @ 2.80GHz and 8 Gb of RAM. The code is available at
\url{https://github.com/TraDE-OPT/Learning-firmly-nonexpansive-operators}.

\begin{figure}
    \centering
\begin{minipage}{0.29\textwidth}
    
\includegraphics[width=\linewidth]{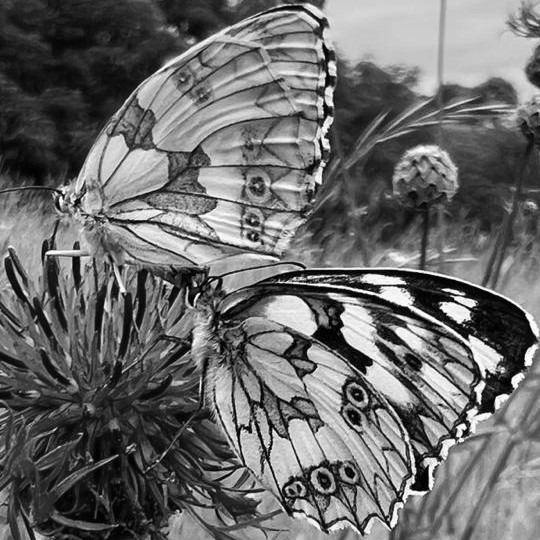}
        \end{minipage}
\begin{minipage}{0.7\textwidth}
\includegraphics[width=0.19\linewidth]{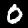}
        \includegraphics[width=0.19\linewidth]{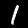}
        \includegraphics[width=0.19\linewidth]{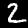}
        \includegraphics[width=0.19\linewidth]{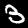}
        \includegraphics[width=0.19\linewidth]{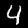}
        \includegraphics[width=0.19\linewidth]{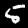}
        \includegraphics[width=0.19\linewidth]{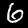}
        \includegraphics[width=0.19\linewidth]{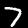}
\includegraphics[width=0.19\linewidth]{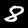}
\includegraphics[width=0.19\linewidth]{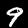}
            
\end{minipage}
        
    \caption{Clean data images. In order to construct the data set, we make use of Gaussian noise with noise level $\eta = 10$.}
    \label{figure:dataset}
\end{figure}

\begin{figure}
    \centering
    \includegraphics[width=0.3\linewidth]{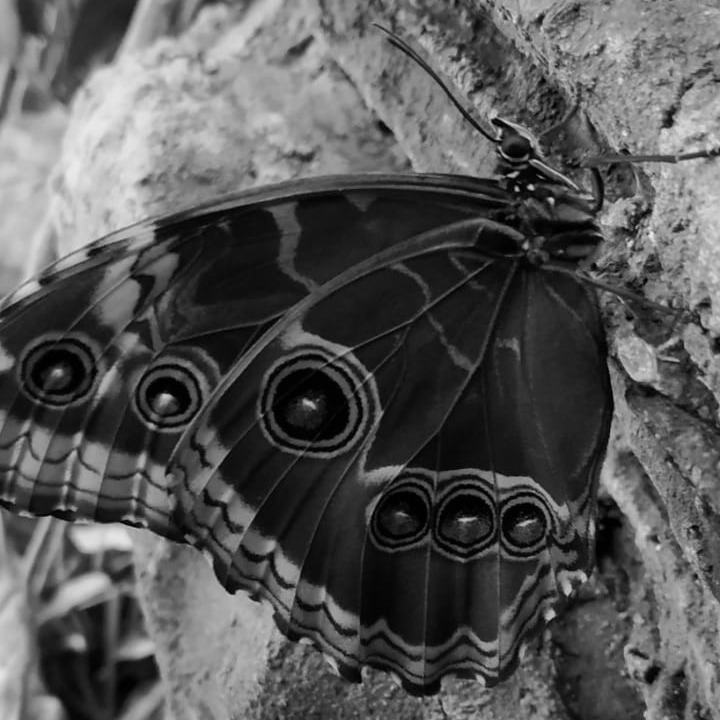}
    \includegraphics[width=0.3\linewidth]{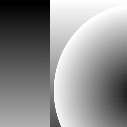}
    \caption{Clean version of the test images for experiments 1 and 2, respectively.}
    \label{figure:testset}
\end{figure}

\paragraph{First experiments: butterfly} 
In the first experiment, we test the learned operator using images similar to the ones utilized during the learning process. We use as training set clean and noisy data of gradients from the image on the left in \cref{figure:dataset} and as test image we consider a noisy version of the image on the left in \cref{figure:testset}. We perform experiments using three different levels of noise applied to the test image, considering Gaussian noise with standard deviation $\eta=10, \ 20$, and $30$, respectively. In \cref{figure:butterfly}, we report some results given by solving \cref{eq:image_denoise} with $R=\frac{\alpha}{2}\|\cdot\|_{\mathrm{F}}^2$ (indicated as H$1$ in the experiments), isotropic TV (TV), and the learned denoiser (Learned). Parameter and step-size choices were done manually, looking for best performance. In \cref{table:1}, we report the results in terms of PSNR and SSIM for reconstructions given by solving \cref{eq:image_denoise}. As it can be seen in \cref{table:1}, our method can perform well in practice and it is competitive with classical gradient-based regularization functions. We do not claim to have a method that improves state-of-the-art approaches. As we have already mentioned, the main contribution of our work is to propose a data-driven approach in the context of learning firmly nonexpansive resolvents, while still providing good theoretical guarantees. Further extending and exploiting this framework to more comprehensive settings is a topic of future work. 

\begin{figure}[ht!]
  \centering
  \includegraphics[width=\linewidth]{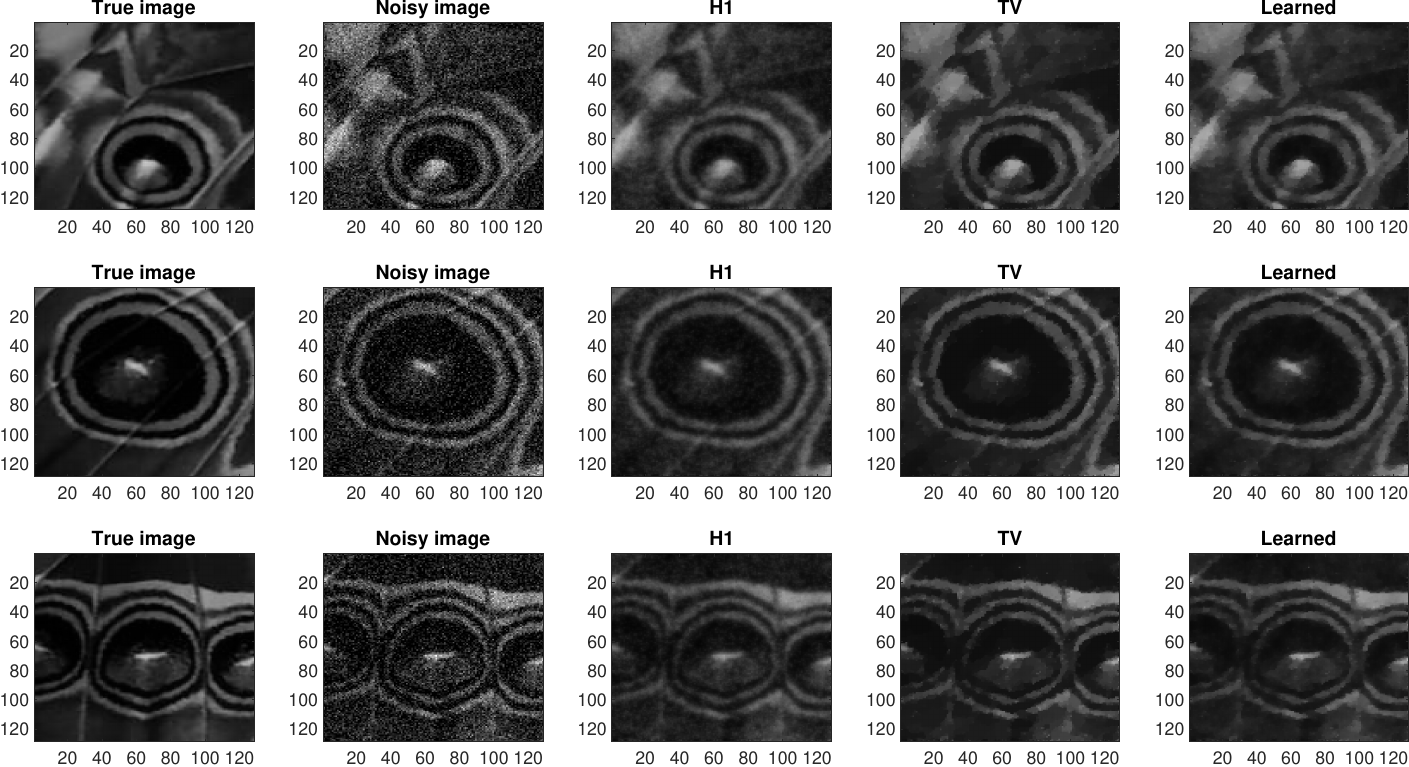}
  \caption{Butterfly images. Results of the experiment performed using a noisy image with Gaussian noise with noise level $\eta=30$.}
  \label{figure:butterfly}
\end{figure}

\begin{table}[ht!]\label{table:1}
\centering
\begin{tabular}{ |p{2.5cm}||p{2cm}|p{2cm}|p{2cm}|p{2cm}|  }
 \hline
 \multicolumn{5}{|c|}{Butterfly images} \\
 \hline
  & Noisy & H1 & TV & Learned \\
 \hline
 \hline
$\eta=10$ &&&& \\
PSNR (dB) & $ 28.2806 $ & $ 32.7332 $ & $ 33.7548 $ & $ \mathbf{33.8520} $ \\ 
SSIM & $ 0.67053 $ & $ 0.85164 $ & $ 0.90794 $ & $ \mathbf{0.90884} $ \\ 
 \hline
 \hline
$\eta=20$ & & & & \\
PSNR (dB) & $ 22.5353 $ & $ 29.3839 $ & $ 29.8716 $ & $ \mathbf{29.9487} $ \\ 
SSIM & $ 0.41033 $ & $ 0.77386 $ & $ \mathbf{0.83393} $ & $ 0.83243 $ \\ 
 \hline
 \hline
$\eta=30$ &&&& \\
PSNR (dB) & $ 19.3036 $ & $ 27.2925 $ & $ 27.5068 $ & $ \mathbf{27.5386} $ \\ 
SSIM & $ 0.27491 $ & $ 0.75030 $ & $ 0.76957 $ & $ \mathbf{0.77803} $ \\ 
 \hline
\end{tabular}
\caption{Comparison of performance for denoising using various regularization methods: $H^1$ penalty (indicated by $\mathrm{H1}$), Isotropic Total Variation (indicated by $\mathrm{TV}$), and ours (Learned). The comparison is given in terms of Peak Signal-to-Noise ratio (PSNR), and Structural Similarity Index Measure (SSIM).}
\end{table}

\paragraph{Second experiment: circles and edges} In the second experiment we performed tests on images of circles and shadows, using two different data sets and thus, two different learned operators. We use training data derived from the images shown in \cref{figure:dataset}. This choice was made to understand how Plug-and-Play methods were able to reconstruct both edges and smooth parts when provided with operators learned using completely different datasets. In \cref{figure:circles}, we report some results given by solving \cref{eq:image_denoise} with H$1$, TV, the denoiser learned using clean and noisy gradients from the image of butterflies on the left in \cref{figure:dataset} (Learned 1) and the one learned using clean and noisy gradients from the MNIST images on the right in \cref{figure:dataset} (Learned 2). Parameter and step-size choices were again done manually, looking for best performance. We show the values in terms of PSNR and SSIM for three different noise levels in \cref{table:2}. It is possible to see that (for both datasets) our data-driven method reconstructs well edges while not introducing too many artifacts in the image. It is interesting to notice that, while the operator learned using more natural images (the images of butterflies) reconstructs smoother solutions, the one learned using the MNIST dataset seems to reconstruct better edges while introducing some artifacts (similarly to what happens for TV regularization).

Since we have the full expression of the learned operators, we can analyze them further. In \cref{figure:convergence} (on the left), we plot the Lipschitz constants for the learned operator in every element $S_t$ of the triangulation; i.e., the norm of $B_tA_t^{-1}$. Recall that the Lipschitz constant for the learned operator is the maximum of the Lipschitz constant in every triangle. We can notice that the Lipschitz constant is indeed less than $1$, even if we used an approximative algorithm to find a solution of \cref{Problem_for_CP} that only satisfies the Lipschitz constraints in the limit. This is due to the fact that during training we replaced the constraints $\|B_t(\cdot)A_t^{-1}\|\leq 1$ in \cref{eq:111} with $\|B_t(\cdot)A_t^{-1}\|\leq 1-\varepsilon$, $\varepsilon \in (0,1)$, with $\epsilon = 0.01$, searching in this way $(1-\varepsilon)$-Lipschitz operators. We learned in this way actual nonexpansive operators for which we have strong convergence guarantees. In \cref{figure:convergence}, on the right, it is possible to see the convergence behavior of the PnP Chambolle-Pock method described in \cref{eq:PnP_CP}. In particular, we show the values of the ``primal residual'' $\|(u^{k}-u^{k+1})/\tau-D^*(v^k-v^{k+1})\|$ and of the ``dual residual'' $\|(v^{k}-v^{k+1})/\sigma-D(u^k-u^{k+1})\|$ as a function of iterations (see for example (12) in \cite{Goldstein2015} to see the relevance of these quantities). \cref{figure:action} shows how the action of the learned operator from $\R^2$ to $\R^2$ looks like. There, one can see that, while the operator learned on the butterfly image (Learned 1) looks similar to the proximal map of the $2$-norm, we can also notice the property of having larger displacements further away from zero, which is more typical of the proximity operator of the $2$-norm squared. The operator learned on the MNIST dataset (Learned 1) looks similar to the proximity operator of the $1$-norm, while presenting some differences.

\begin{figure}[ht!]
  \centering
\includegraphics[width=\linewidth]{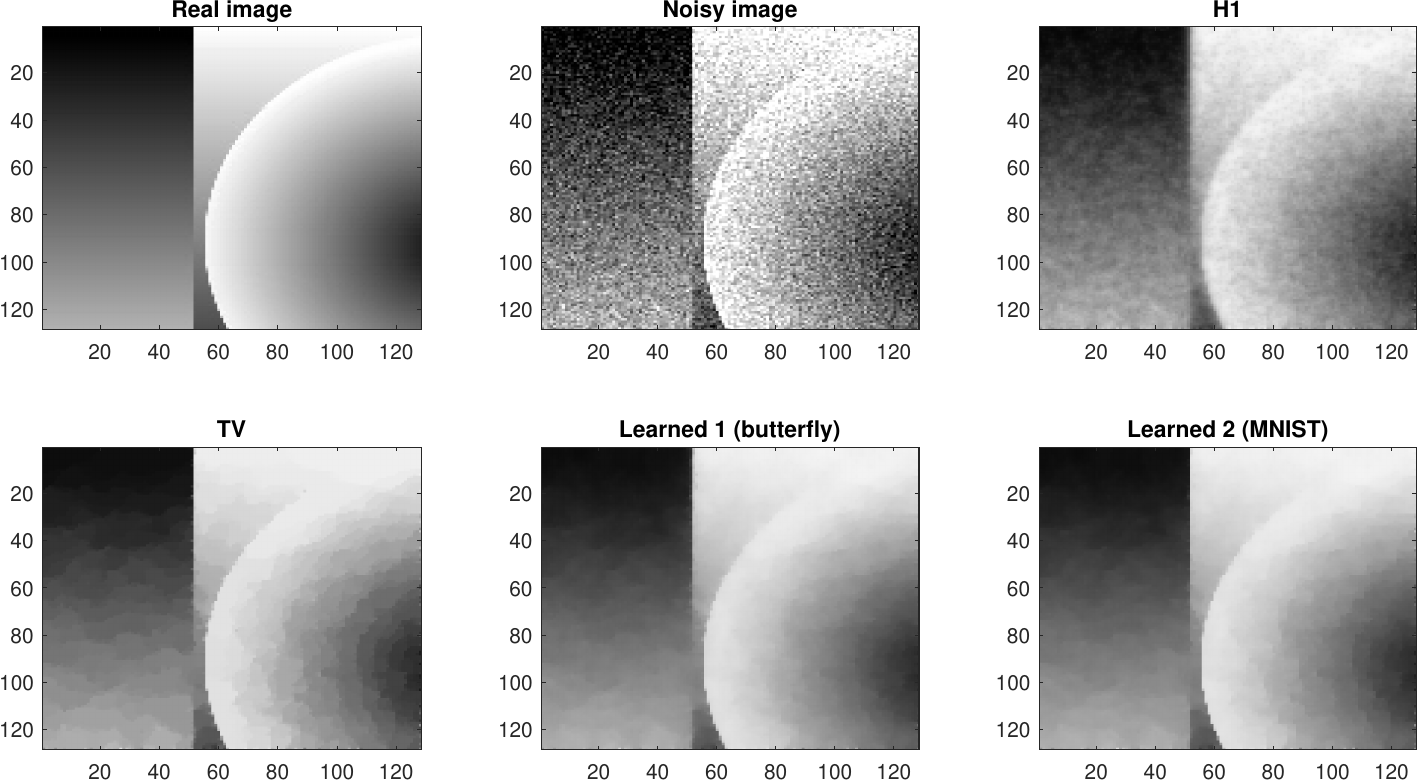}
  \caption{Circles. Obtained results when using a noisy image with Gaussian noise, with noise level $\eta=30$. We compare the performance of our method with classical regularization methods, and also when considering different training sets.}
  \label{figure:circles}
\end{figure}

\begin{table}[ht!]\label{table:2}
\centering
\begin{tabular}{ |p{2.2cm}||p{1.6cm}|p{1.6cm}|p{1.6cm}|p{1.65cm}|p{1.65cm}|  }
 \hline
 \multicolumn{6}{|c|}{Circle images} \\
 \hline
  & Noisy & H1 & TV & Learned 1 & Learned 2 \\
 \hline
 \hline
$\eta=10$ &&&&& \\
PSNR (dB) & $ 28.3301 $ & $ 32.3894 $ & $ 38.2521 $ & $ 37.8624$ & $ \mathbf{38.7341} $ \\ 
SSIM & $ 0.54501 $ & $ 0.83355 $ & $ 0.94874 $ & $ 0.95977 $ & $ \mathbf{0.96274} $ \\
 \hline
 \hline
$\eta=20$ & & & && \\
PSNR (dB) & $ 22.5004 $ & $ 29.1904 $ & $ 33.9466 $ & $ 33.4002$ & $ \mathbf{34.4220} $ \\ 
SSIM & $ 0.28188 $ & $ 0.76196 $ & $ 0.91487 $ & $ \mathbf{0.94063} $ & $ 0.93997 $ \\
 \hline
 \hline
$\eta=30$ &&&&& \\
PSNR (dB) & $ 19.1533 $ & $ 27.4474 $ & $ 31.0878 $ & $ 30.4950$ & $ \mathbf{31.4184} $ \\ 
SSIM & $ 0.17782 $ & $ 0.76036 $ & $ 0.89535 $ & $ \mathbf{0.92679} $ & $ 0.92502 $ \\
 \hline
\end{tabular}
\caption{Comparison of performance for denoising using various regularization methods: $H^1$ penalty (indicated by $\mathrm{H1}$), Isotropic Total Variation (indicated by $\mathrm{TV}$), the learned operator using the butterfly dataset (Learned 1), and the learned operator using the MNIST dataset (Learned 2). The comparison is given in terms of Peak Signal-to-Noise ratio (PSNR), and Structural Similarity Index Measure (SSIM).}
\end{table}

\begin{figure}
    \centering
    \includegraphics[width=0.5\linewidth]{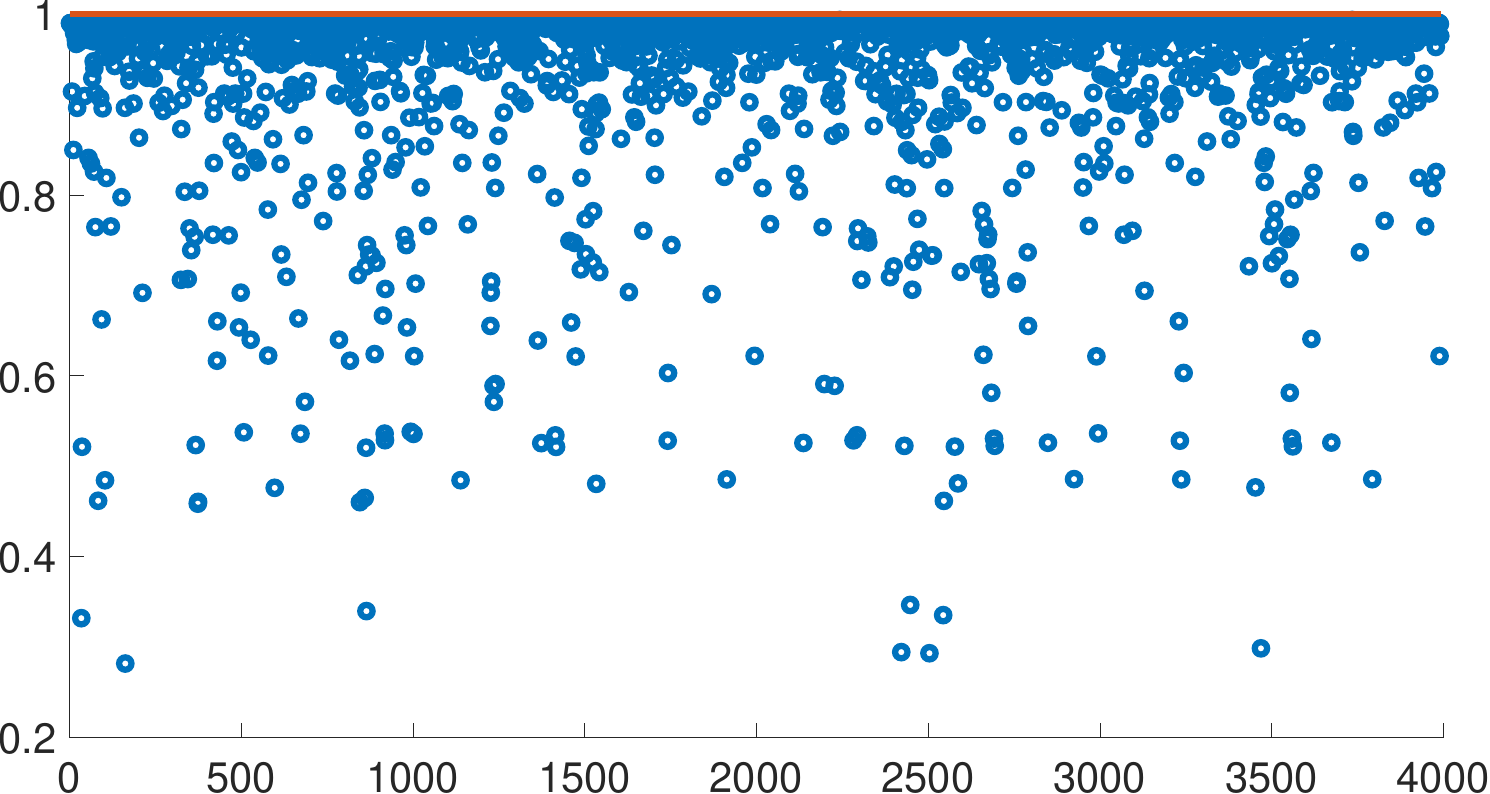}
    \includegraphics[width=0.49\linewidth]{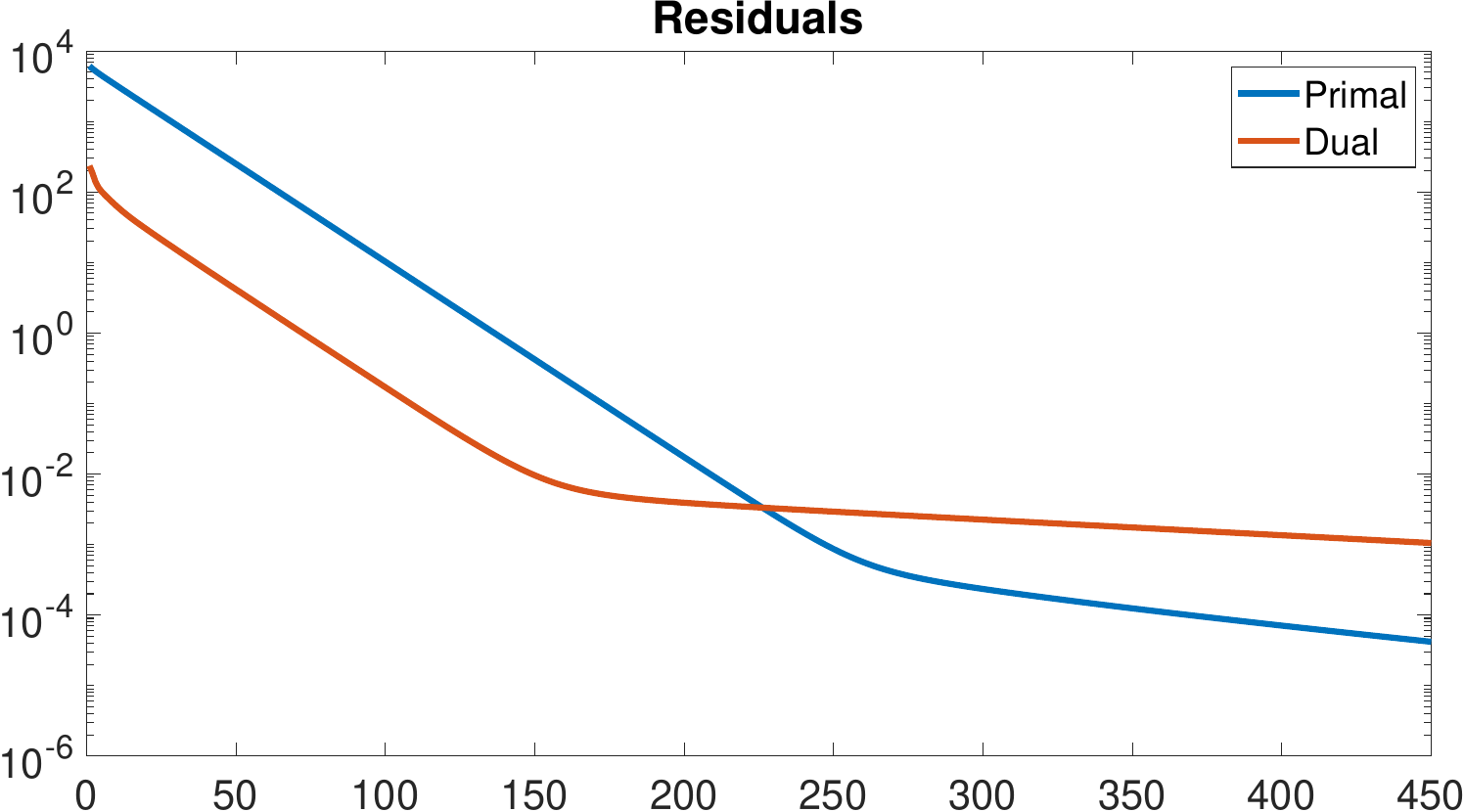}
    \caption{On the left: Lipschitz constants for the learned operator (Learned 1) in every element of the triangulation. On the right: convergence behavior of the primal and dual residuals for the PnP Chambolle-Pock method described in \cref{eq:PnP_CP}, employing the operator learned from butterfly images (Learned 1).}
    \label{figure:convergence}
\end{figure}

\begin{figure}
    \centering
    \includegraphics[width=0.49\linewidth]{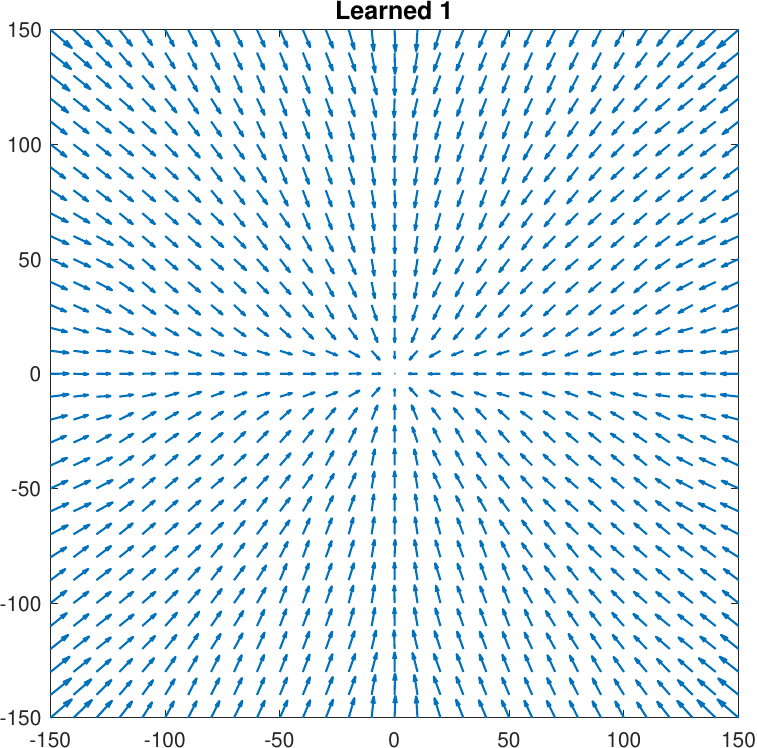}
    \includegraphics[width=0.49\linewidth]{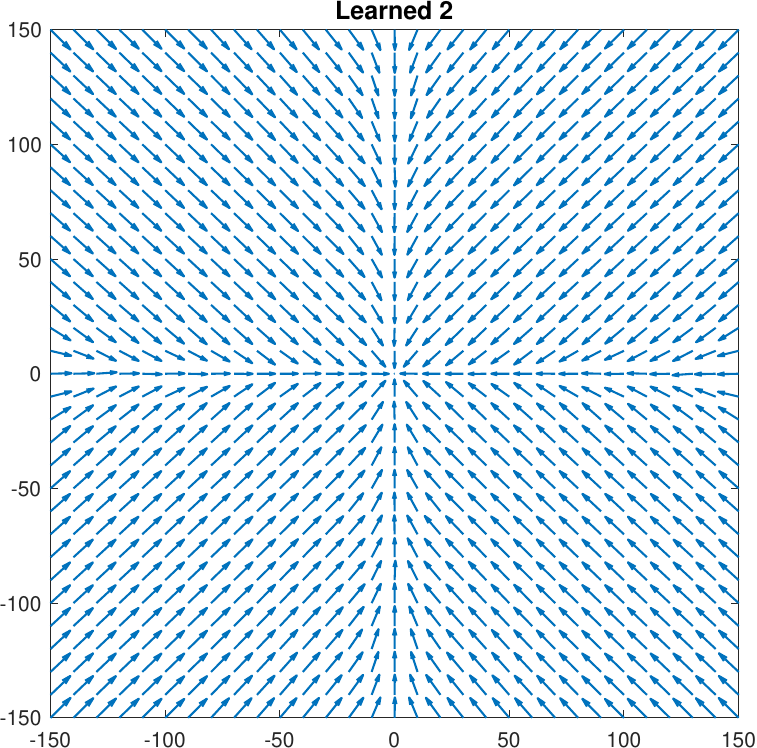}
    \includegraphics[width=0.32\linewidth]{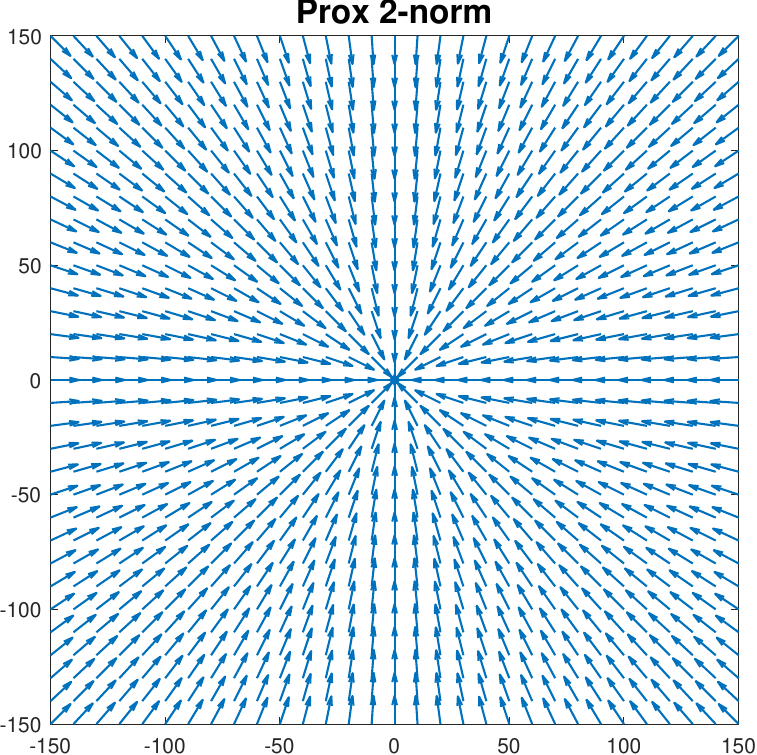}
    \includegraphics[width=0.32\linewidth]{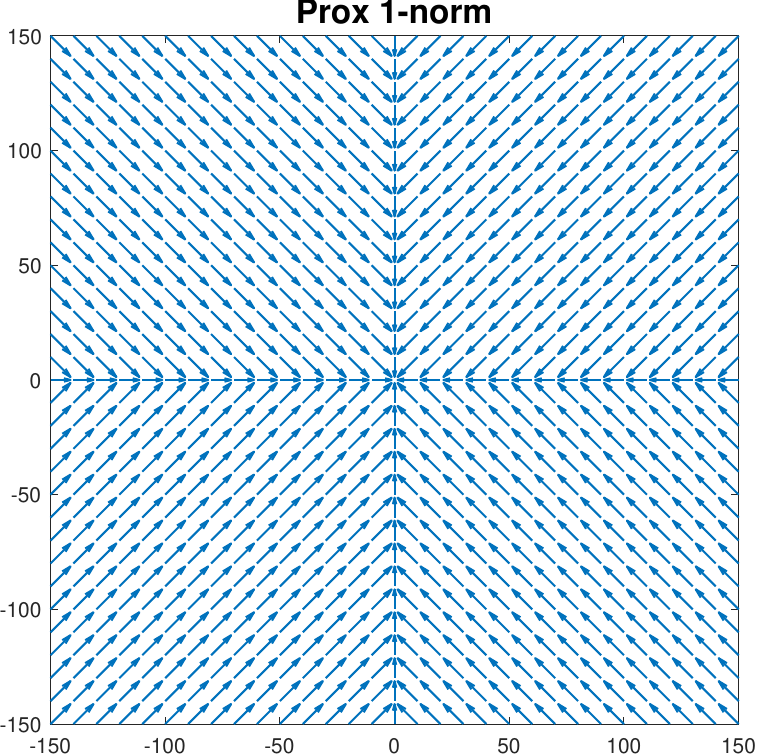}
    \includegraphics[width=0.32\linewidth]{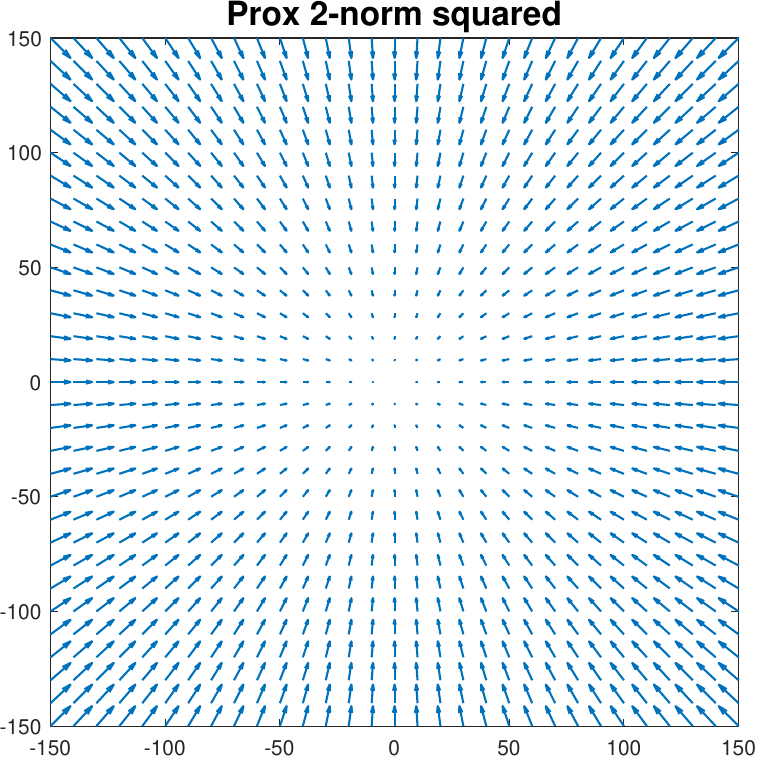}
    \caption{Action of proximal operators on a section of $\R^2$. Top-left: our operator $T_r$ learned on the butterfly image (Learned 1), top-right: our operator $T_r$ learned on the MNIST dataset (Learned 2), bottom-left: $\mathrm{prox}$ operator of the $2$-norm (used in the isotropic TV), bottom-center: $\mathrm{prox}$ map of the $1$-norm (used in the anisotropic TV), bottom-right: $\mathrm{prox}$ map of the $2$-norm squared.}
    \label{figure:action}
\end{figure}

\paragraph{Third experiment: proximity operators} A natural question that arises looking at these pictures is whether the learned operator is indeed a proximity operator of some underlying function or if it is just a firmly nonexpansive operator. Every proximal mapping of a proper, lower semicontinuous convex function $f$, is firmly nonexpansive and, moreover, is the gradient of a convex $C^{1}$ potential (the Moreau envelope of the conjugate). Hence, proximity operators are conservative vector fields (path integrals around closed loops are zero). Conversely, if a firmly nonexpansive mapping is conservative (i.e., admits a convex scalar potential), then it is the proximity operator of a convex function. Therefore, to test if a learned firmly nonexpansive operator is a proximity operator, we test if it is a gradient. A practical and informative necessary check is to test conservativity via closed–loop line integrals: small integrals indicate the field is close to being a proximity operator, while large integrals rule this out. To probe whether the learned operator is (approximately) a gradient, we performed a circulation test on many closed curves. For a vector field \(T:\mathbb{R}^2\to\mathbb{R}^2\), we evaluate the line integral \(\oint_\gamma T\cdot d\mathbf r\) over circles
\[
\gamma=\bigl\{\,c + R(\cos\theta,\sin\theta)\ :\ \theta\in[0,2\pi)\,\bigr\}
\]
 contained in the domain \([{-}150,150]^2\). Since a gradient field is conservative, \(\oint_\gamma \nabla\phi\cdot d\mathbf r=0\) for every closed curve; departures from zero reveal non-conservative content. To make results comparable across radii and locations, we report the \emph{normalized circulation}
\[
\mathrm{NormCirc}(T;\gamma)
=
\frac{\bigl|\oint_\gamma T\cdot d\mathbf r\bigr|}
     {\overline{\|T\|}_\gamma\,(2\pi R)},
\qquad
\overline{\|T\|}_\gamma:=\frac{1}{|\gamma|}\int_\gamma \|T\|\,ds,
\]
so that gradients yield \(\mathrm{NormCirc}\approx 0\), while purely rotational fields give constant values \(\mathrm{NormCirc}\approx 1\).
In our experiments, we sampled \(N=2048\) points per circle, tested radii \(R\in\{10,20,40,60,80,100\}\), and \(25\) random centers per radius (all circles fully inside the domain).

We compared with the family
\[
T_\alpha(x)=\alpha\,R_{90}x+(1-\alpha)\,x,
\qquad
R_{90}=\begin{bmatrix}0&-1\\[2pt]1&0\end{bmatrix},
\]
which linearly blends a pure rotation (non-conservative) with the identity (the proximity operator of the zero function, and clearly a gradient). 

We compute the mean of \(\mathrm{NormCirc}\) over all tested circles for both the learned operator and \(T_\alpha\), and define $\alpha^*$ to be the tested value of $\alpha$ such that $\mathrm{NormCirc}$ defined above for $T_{\alpha}$ and the learned operator $T$ are the most similar. In this way, we can think of \(\alpha^*\) as quantifying in some sense the amount of rotational content of the learned operator; smaller \(\alpha^*\) means “more gradient-like". The results in \cref{figure:isgradient_BFLY} and \cref{figure:isgradient_MNIST} show that the learned operator has a very small mean normalized circulation. We obtain approximately \(3.8\times 10^{-3}\) for the operator trained with the butterfly image (\cref{figure:isgradient_BFLY}) and \(4.7\times 10^{-4}\) for the operator trained with the MNIST dataset (\cref{figure:isgradient_MNIST}). This corresponds to a small \(\alpha^*\) on the calibration curve (around \(4.6 \times 10^{-3}\) and \(4.2 \times 10^{-4}\), respectively). The histogram of deviations around the mean further indicates that the circulation remains small across radii and centers, not just on average. Altogether, this provides strong evidence that the learned (firmly nonexpansive) operator is close to conservative, i.e., close to a gradient mapping (and thus a proximity operator) on the tested domain. This finding is noteworthy: despite not imposing circulation constraints or similar during training, the best data-fitting, firmly nonexpansive operators we obtained appear to be nearly gradient maps. This raises an interesting question for future work: do data-driven, firmly nonexpansive operators preferentially converge to proximity operators, or, even more profoundly, are proximity operators the only objects in the argmin of \cref{experr} for ``real" data distributions $\mu$?

\begin{figure}
    \centering
    \includegraphics[width=0.5\linewidth]{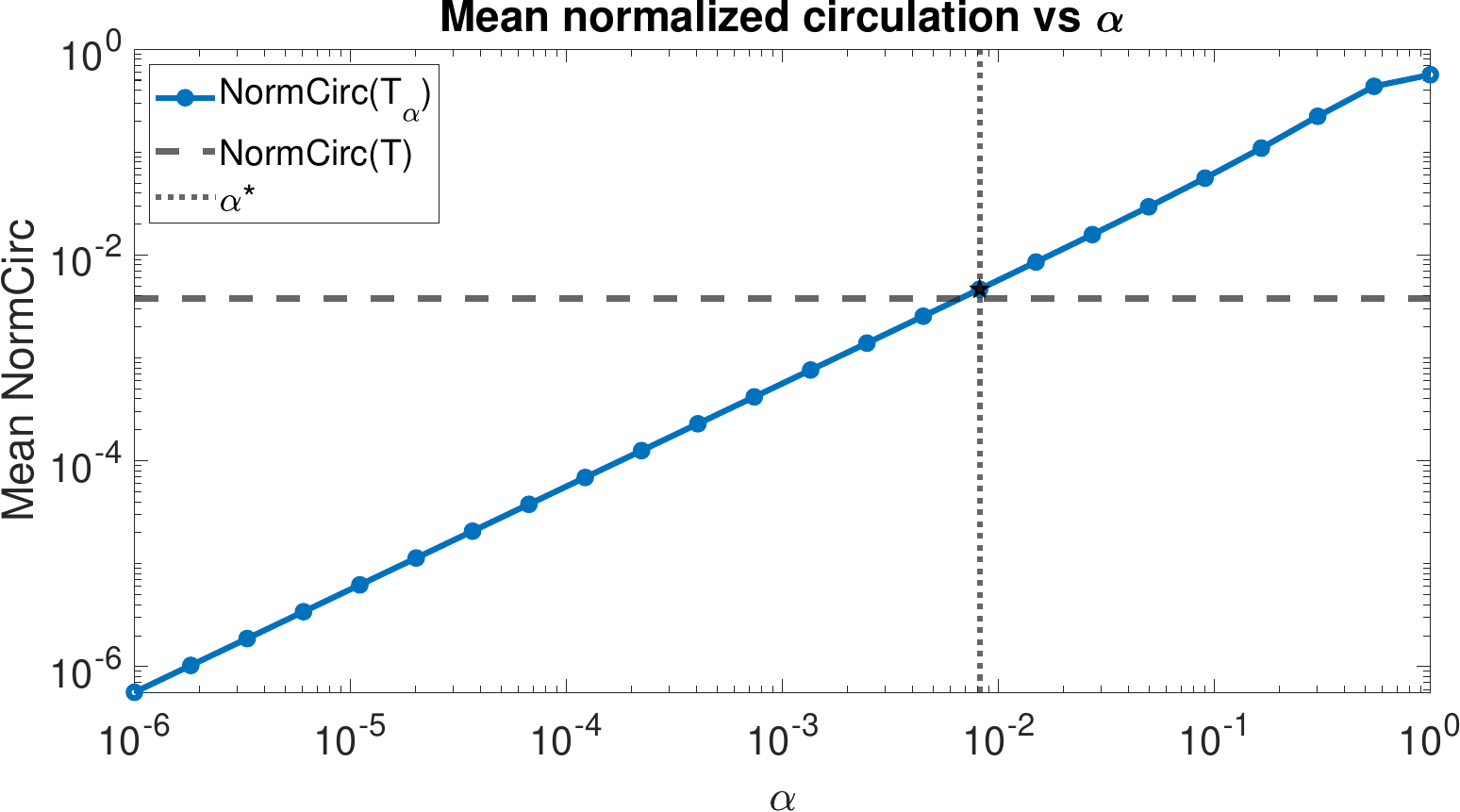}
    \includegraphics[width=0.49\linewidth]{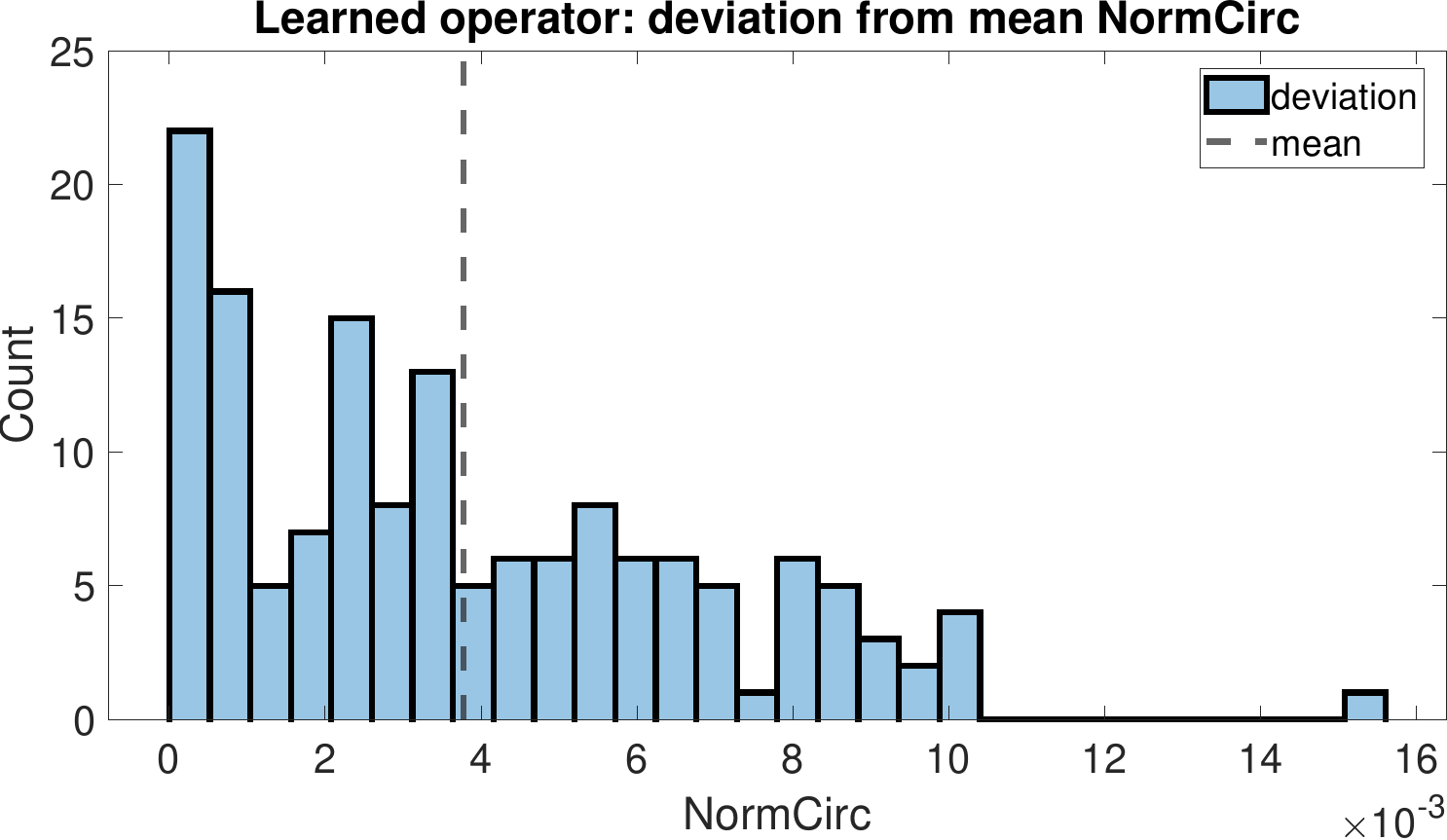}
    \caption{Results of the gradient test for the operator learned on the butterfly image (Learned 1). Left: average $\mathrm{NormCirc}(T_\alpha;\gamma)$ for the operator $T_\alpha(x)=\alpha\,R_{90}x+(1-\alpha)\,x$ and average $\mathrm{NormCirc}(T;\gamma)$ for the learned operator, over all the the $150$ tested paths $\gamma$. Right: histogram of values $\mathrm{NormCirc}(T;\gamma)$ for the $150$ tested paths.}
    \label{figure:isgradient_BFLY}
\end{figure}

\begin{figure}
    \centering
    \includegraphics[width=0.5\linewidth]{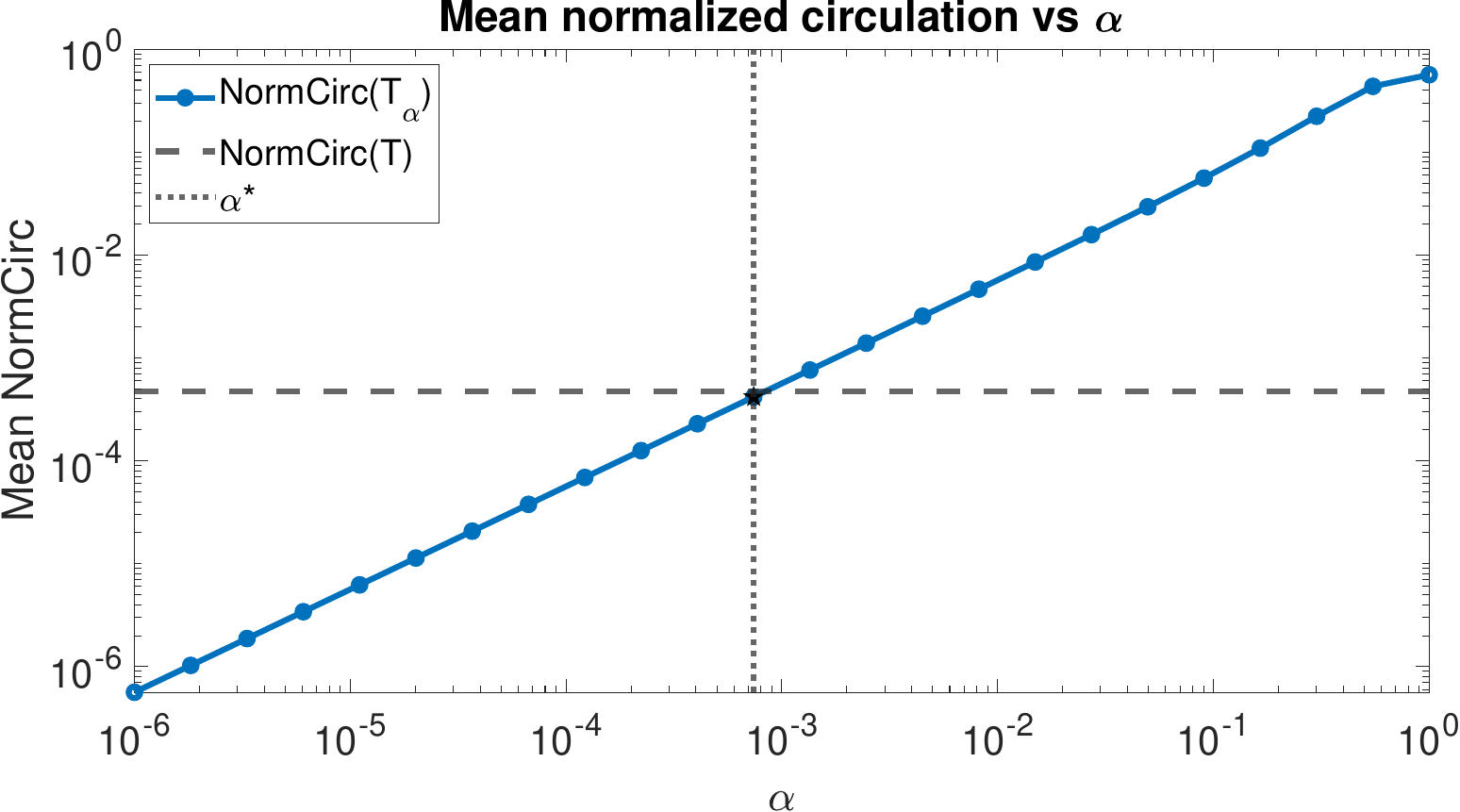}
    \includegraphics[width=0.49\linewidth]{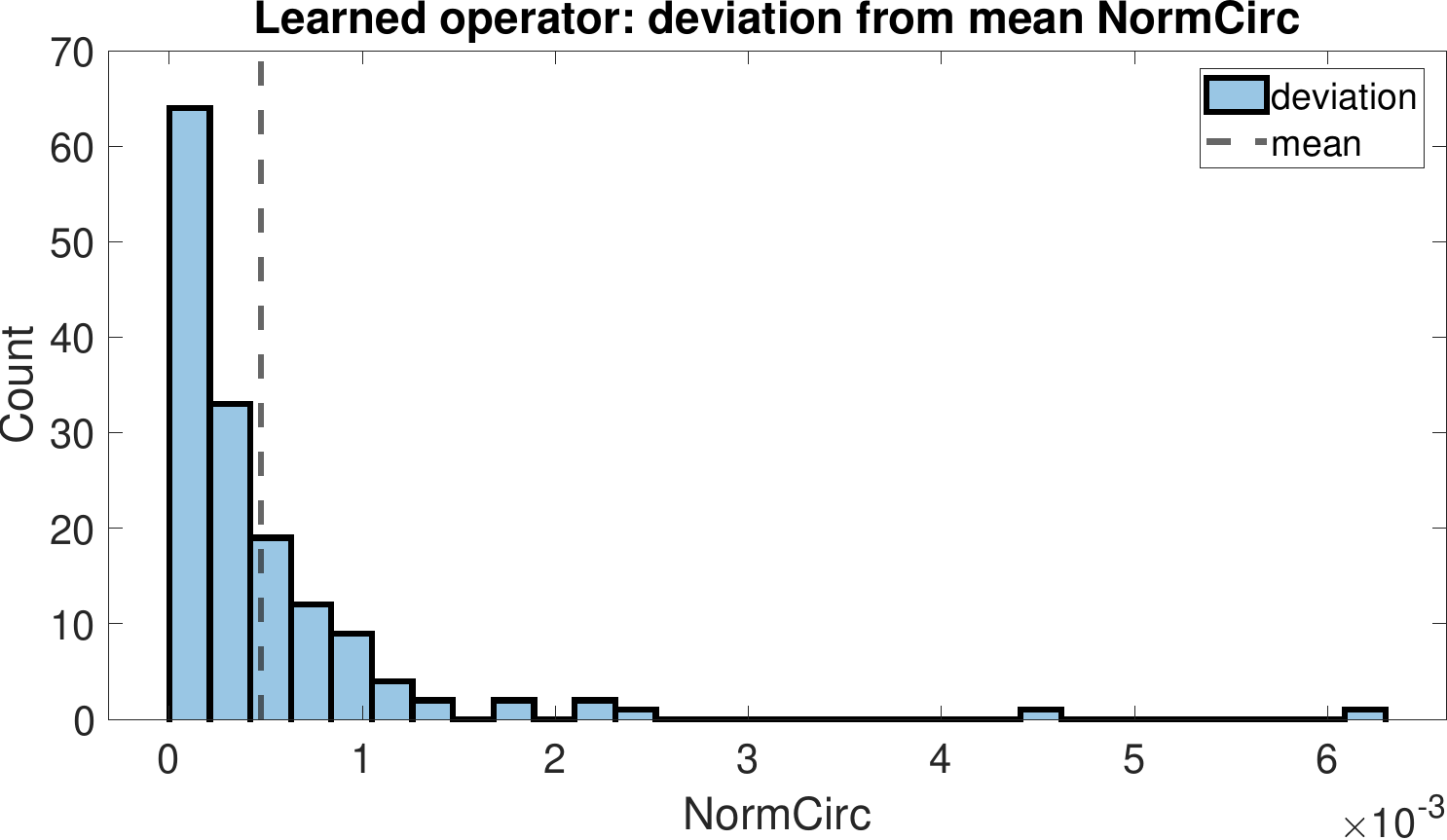}
    \caption{Results of the gradient test for the operator learned on the MNIST dataset (Learned 2).  Left: average $\mathrm{NormCirc}(T_\alpha;\gamma)$ for the operator $T_\alpha(x)=\alpha\,R_{90}x+(1-\alpha)\,x$ and average $\mathrm{NormCirc}(T;\gamma)$ for the learned operator, over all the the $150$ tested paths $\gamma$. Right: histogram of values $\mathrm{NormCirc}(T;\gamma)$ for the $150$ tested paths.}
    \label{figure:isgradient_MNIST}
\end{figure}

\section{Conclusions}
Plug-and-Play (PnP) models \cite{venkat2013} have turned to be very successful in a wide variety of applications such as image recovery problems. In this work, considering standard statistical learning theory \cite{cusma}, we present a complete theoretical framework for studying the problem of learning firmly nonexpansive operators. In addition, we propose a constructive method to produce firmly nonexpansive operators that adapts well to the characteristics of a given training set of noisy measurements/solutions. Our proposed approach, based on simplicial partitions and their refinements, gives theoretical guarantees for the convergence of PnP algorithms. In practice, constructing simplicial partitions in high dimensions is a challenging task due to their poor scalability properties. For this reason, we limit our numerical experiments to a low-dimensional setting. Nevertheless, we show that our technique can still be applied in complex scenarios, obtaining promising results. Furthermore, we do not claim that we can beat state of the art approaches involving, for instance, artificial neural networks \cite{Devalla2018,zhangIEEE,zhang}, since our problem depends on a low number of parameters, or higher order TV-based approaches such as TGV \cite{tgv}, since we consider only first order information in our experiments. Instead, we detail and illustrate our presented methodology, demonstrate its applicability in image denoising problems and, thanks to the low-dimensional setting, provide a clear geometrical interpretation of the action of the learned operators, providing also numerical insights on how the learned operators are close to being gradients and thus proximity operators of some underlying function. In future projects, we plan to investigate other instances where the method can be used with good practical performance. Moreover, we think that the learning process could be improved numerically in the future, finding other ways to solve \cref{eq:learning_task}. We also aim to go beyond firmly nonexpansiveness and learn proximal operators of nonconvex regularizers. Recent works have pointed out that going beyond convexity can be beneficial, achieving better reconstructions \cite{hurault22, ryupnp2019}. In all these cases, however, it is still required to learn an operator with a fixed Lipschitz constant. Since the theory developed in the present work easily adapts to the case of learning $L$-Lipschitz operators for any $L>0$, we believe that our contribution could be an important complement to existing analyses. 

\section{Acknowledgements}
This project has been supported by the TraDE-OPT project, which received funding from the European Union’s Horizon 2020 research and innovation program under the Marie Skłodowska-Curie grant agreement No 861137. The Department of Mathematics and Scientific Computing at the University of Graz, with which K.B. is affiliated, is a member of NAWI Graz (\url{https://nawigraz.at/en}). J.C.R. acknowledges the support of the Italian Ministry of Education, University and Research (grant ML4IP R205T7J2KP). The research by E.N. has been supported by the MUR Excellence Department Project awarded to Dipartimento di Matematica, Università di Genova, CUP D33C23001110001, and the US Air Force Office of Scientific Research (FA8655-22-1-7034). E.N. is a member of the Gruppo Nazionale per l’Analisi Matematica, la Probabilità e le loro Applicazioni (GNAMPA) of the Istituto Nazionale di Alta Matematica (INdAM). This work represents only the view of the authors. The European Commission and other organizations are not responsible for any use that may be made of the information it contains. We want to thank the anonymous reviewers for their valuable comments, which we believe made it possible to significantly improve the quality of the paper.

\bibliographystyle{siamplain}
\bibliography{references}

\end{document}

%% file: shared.tex
\newcommand{\Id}{\operatorname{Id}}

\newcommand{\Lip}{\operatorname{Lip}}
\newcommand{\gra}{\operatorname{gra}}
\newcommand{\N}{\mathcal{N}}
\newcommand{\M}{\mathcal{M}}
\newcommand{\weakk}{\stackrel{\ast}{\rightharpoonup}}
\newcommand{\R}{\mathbb{R}}

\newcommand{\conv}{\operatorname{conv}}

\newcommand{\bx}{\bar{x}}
\newcommand{\by}{\bar{y}}
\newcommand{\bX}{\bar{X}}
\newcommand{\bY}{\bar{Y}}
\newcommand{\cL}{\mathcal{L}}
\newcommand{\cX}{\mathcal{X}}
\DeclareMathOperator*{\argmin}{arg\,min}
\DeclareRobustCommand{\Chi}{{\mathpalette\irchi\relax}}
\newcommand{\irchi}[2]{\raisebox{\depth}{$#1\chi$}} 

\usepackage{lipsum}
\usepackage{graphicx}
\usepackage{stackengine}
\usepackage{amsfonts}
\usepackage{amssymb}
\usepackage{graphicx}
\usepackage{epstopdf}
\usepackage{algorithmic}
\usepackage{comment}
\usepackage{dsfont}
\usepackage{hyperref}
\usepackage{enumitem}

\ifpdf
  \DeclareGraphicsExtensions{.eps,.pdf,.png,.jpg}
\else
  \DeclareGraphicsExtensions{.eps}
\fi

\newsiamremark{remark}{Remark}
\newsiamremark{example}{Example}
\newsiamremark{hypothesis}{Hypothesis}
\crefname{hypothesis}{Hypothesis}{Hypotheses}
\newsiamthm{claim}{Claim}
\newsiamthm{prop}{Proposition}
\newsiamthm{defi}{Definition}

\headers{Learning firmly nonexpansive operators}{K. Bredies, J. Chirinos-Rodriguez, E. Naldi}

\title{Learning firmly nonexpansive operators{\thanks{July 18, 2024}}}

\author{Kristian Bredies\textsuperscript{1} \and Jonathan Chirinos-Rodríguez\textsuperscript{2} \and Emanuele Naldi\textsuperscript{3}}

\usepackage{amsopn}

\graphicspath{{./Images}}


%% file: article.bbl
\begin{thebibliography}{10}

\bibitem{Anil2019}
{\sc C.~Anil, J.~Lucas, and R.~Grosse}, {\em Sorting out {L}ipschitz function
  approximation}, in Proceedings of the 36th International Conference on
  Machine Learning, 2019, pp.~291--301,
  \url{https://proceedings.mlr.press/v97/anil19a.html}.

\bibitem{Arens56}
{\sc R.~F. Arens and J.~Eells, Jr.}, {\em On embedding uniform and topological
  spaces}, Pacific J. Math., 6 (1956), pp.~397--403,
  \url{https://doi.org/10.2140/pjm.1956.6.397}.

\bibitem{Arridge2019}
{\sc S.~Arridge, P.~Maass, O.~Öktem, and C.-B. Schönlieb}, {\em Solving
  inverse problems using data-driven models}, Acta Numerica, 28 (2019),
  pp.~1--174, \url{https://doi.org/10.1017/S0962492919000059}.

\bibitem{ABS11}
{\sc H.~Attouch, J.~Bolte, and B.~F. Svaiter}, {\em {Convergence of descent
  methods for semi-algebraic and tame problems: proximal algorithms,
  forward-backward splitting, and regularized Gauss--Seidel methods}},
  {Mathematical Programming, Series A}, 137 (2011), pp.~91--124,
  \url{https://doi.org/10.1007/s10107-011-0484-9}.

\bibitem{BCombettes}
{\sc H.~H. Bauschke and P.~L. Combettes}, {\em Convex analysis and monotone
  operator theory in {H}ilbert spaces}, CMS Books in Mathematics/Ouvrages de
  Math\'{e}matiques de la SMC, Springer, Cham, second~ed., 2017,
  \url{https://doi.org/10.1007/978-3-319-48311-5}.

\bibitem{burgerbenning}
{\sc M.~Benning and M.~Burger}, {\em Modern regularization methods for inverse
  problems}, Acta Numerica, 27 (2018), pp.~1--111,
  \url{https://doi.org/10.1017/S0962492918000016}.

\bibitem{Bot_ADMM}
{\sc R.~I. Bo\c{t} and E.~R. Csetnek}, {\em A{DMM} for monotone operators:
  convergence analysis and rates}, Adv. Comput. Math., 45 (2019), pp.~327--359,
  \url{https://doi.org/10.1007/s10444-018-9619-3}.

\bibitem{admm}
{\sc S.~Boyd, N.~Parikh, E.~Chu, B.~Peleato, and J.~Eckstein}, {\em Distributed
  optimization and statistical learning via the alternating direction method of
  multipliers}, Foundations and Trends in Machine Learning, 3 (2011),
  pp.~1--122, \url{https://doi.org/10.1561/2200000016}.

\bibitem{braides}
{\sc A.~Braides}, {\em {Gamma-Convergence for Beginners}}, Oxford University
  Press, 2002, \url{https://doi.org/10.1093/acprof:oso/9780198507840.001.0001}.

\bibitem{Bredies_FB}
{\sc K.~Bredies}, {\em A forward–backward splitting algorithm for the
  minimization of non-smooth convex functionals in {B}anach space}, Inverse
  Problems, 25 (2008), p.~015005,
  \url{https://doi.org/10.1088/0266-5611/25/1/015005}.

\bibitem{breetal22}
{\sc K.~Bredies, M.~Carioni, S.~Fanzon, and F.~Romero}, {\em A generalized
  conditional gradient method for dynamic inverse problems with optimal
  transport regularization}, Foundations of Computational Mathematics,  (2022),
  \url{https://doi.org/10.1007/s10208-022-09561-z}.

\bibitem{brelinear}
{\sc K.~Bredies, M.~Carioni, S.~Fanzon, and D.~Walter}, {\em Asymptotic linear
  convergence of fully-corrective generalized conditional gradient methods},
  Math. Program., 205 (2024), pp.~135--202,
  \url{https://doi.org/10.1007/s10107-023-01975-z}.

\bibitem{bredies2021degenerate}
{\sc K.~Bredies, E.~Chenchene, D.~A. Lorenz, and E.~Naldi}, {\em Degenerate
  preconditioned proximal point algorithms}, SIAM Journal on Optimization, 32
  (2022), pp.~2376--2401, \url{https://doi.org/10.1137/21M1448112}.

\bibitem{Bredies_2020}
{\sc K.~Bredies and M.~Holler}, {\em Higher-order total variation approaches
  and generalisations}, Inverse Problems, 36 (2020), p.~123001,
  \url{https://doi.org/10.1088/1361-6420/ab8f80}.

\bibitem{tgv}
{\sc K.~Bredies, K.~Kunisch, and T.~Pock}, {\em Total generalized variation},
  SIAM J. Imaging Sci., 3 (2010), pp.~492--526,
  \url{https://doi.org/10.1137/090769521}.

\bibitem{bredies2023extreme}
{\sc K.~Bredies, J.~C. Rodriguez, and E.~Naldi}, {\em On extreme points and
  representer theorems for the {Lipschitz} unit ball on finite metric spaces},
  Arch. Math., 122 (2024), pp.~651--658,
  \url{https://doi.org/10.1007/s00013-024-01978-y}.

\bibitem{BrediesSun}
{\sc K.~Bredies and H.~Sun}, {\em Preconditioned {Douglas--Rachford} splitting
  methods for convex-concave saddle-point problems}, SIAM Journal on Numerical
  Analysis, 53 (2015), pp.~421--444, \url{https://doi.org/10.1137/140965028}.

\bibitem{Bredies_2017}
{\sc K.~Bredies and H.~Sun}, {\em A proximal point analysis of the
  preconditioned alternating direction method of multipliers}, Journal of
  Optimization Theory and Applications, 173 (2017), pp.~878–--907,
  \url{https://doi.org/10.1007/s10957-017-1112-5}.

\bibitem{buzz}
{\sc G.~T. Buzzard, S.~H. Chan, S.~Sreehari, and C.~A. Bouman}, {\em
  {Plug-and-Play} unplugged: Optimization-free reconstruction using consensus
  equilibrium}, SIAM Journal on Imaging Sciences, 11 (2018), pp.~2001--2020,
  \url{https://doi.org/10.1137/17M1122451}.

\bibitem{chambtv}
{\sc A.~Chambolle and P.-L. Lions}, {\em Image recovery via total variation
  minimization and related problems}, Numerische Mathematik, 76 (1997),
  pp.~167--188, \url{https://doi.org/10.1007/s002110050258}.

\bibitem{Chambolle2011}
{\sc A.~Chambolle and T.~Pock}, {\em A first-order primal-dual algorithm for
  convex problems with applications to imaging}, Journal of Mathematical
  Imaging and Vision, 40 (2011), pp.~120--145,
  \url{https://doi.org/10.1007/s10851-010-0251-1}.

\bibitem{ChambollePock2016}
{\sc A.~Chambolle and T.~Pock}, {\em An introduction to continuous optimization
  for imaging}, Acta Numerica, 25 (2016), pp.~161–--319,
  \url{https://doi.org/10.1017/S096249291600009X}.

\bibitem{chan}
{\sc S.~H. Chan, X.~Wang, and O.~A. Elgendy}, {\em {Plug-and-Play} {ADMM} for
  image restoration: Fixed-point convergence and applications}, IEEE
  Transactions on Computational Imaging, 3 (2017), pp.~84--98,
  \url{https://doi.org/10.1109/TCI.2016.2629286}.

\bibitem{chanmarqmu}
{\sc T.~Chan, A.~Marquina, and P.~Mulet}, {\em High-order total variation-based
  image restoration}, SIAM Journal on Scientific Computing, 22 (2000),
  pp.~503--516, \url{https://doi.org/10.1137/S1064827598344169}.

\bibitem{Rockafellar_FB}
{\sc G.~H.-G. Chen and R.~T. Rockafellar}, {\em Convergence rates in
  forward--backward splitting}, SIAM Journal on Optimization, 7 (1997),
  pp.~421--444, \url{https://doi.org/10.1137/S1052623495290179}.

\bibitem{Cobzas2019}
{\sc S.~Cobzaş, R.~Miculescu, and A.~Nicolae}, {\em Lipschitz Functions},
  vol.~2241 of Lecture Notes in Mathematics, Springer, 2019,
  \url{https://doi.org/10.1007/978-3-030-16489-8}.

\bibitem{combpesq}
{\sc P.~L. Combettes and J.-C. Pesquet}, {\em Proximal splitting methods in
  signal processing}, in Fixed-point algorithms for inverse problems in science
  and engineering, vol.~49 of Springer Optim. Appl., Springer, New York, 2011,
  pp.~185--212, \url{https://doi.org/10.1007/978-1-4419-9569-8_10}.

\bibitem{CombettesPesquet2021}
{\sc P.~L. Combettes and J.-C. Pesquet}, {\em Fixed point strategies in data
  science}, IEEE Transactions on Signal Processing, 69 (2021), pp.~3878--3905,
  \url{https://doi.org/10.1109/TSP.2021.3069677}.

\bibitem{CW05}
{\sc P.~L. Combettes and V.~R. Wajs}, {\em Signal recovery by proximal
  forward-backward splitting}, Multiscale Modeling \& Simulation, 4 (2005),
  pp.~1168--1200, \url{https://doi.org/10.1137/050626090}.

\bibitem{condat}
{\sc L.~Condat}, {\em {A primal-dual splitting method for convex optimization
  involving Lipschitzian, proximable and linear composite terms}}, {Journal of
  Optimization Theory and Applications}, 158 (2013), pp.~460--479,
  \url{https://doi.org/10.1007/s10957-012-0245-9}.

\bibitem{crisiglewal}
{\sc G.~Cristinelli, J.~A. Iglesias, and D.~Walter}, {\em Conditional gradients
  for total variation regularization with pde constraints: a graph cuts
  approach}, Computational Optimization and Applications,  (2025),
  \url{https://doi.org/10.1007/s10589-025-00699-4}.

\bibitem{cusma}
{\sc F.~Cucker and S.~Smale}, {\em On the mathematical foundations of
  learning}, Bull. Amer. Math. Soc. (N.S.), 39 (2002), pp.~1--49,
  \url{https://doi.org/10.1090/S0273-0979-01-00923-5}.

\bibitem{dabov}
{\sc K.~Dabov, A.~Foi, V.~Katkovnik, and K.~Egiazarian}, {\em Image denoising
  by sparse 3-{D} transform-domain collaborative filtering}, IEEE Transactions
  on Image Processing, 16 (2007), pp.~2080--2095,
  \url{https://doi.org/10.1109/TIP.2007.901238}.

\bibitem{Devalla2018}
{\sc S.~K. Devalla, P.~K~R, B.~Sreedhar, G.~Subramanian, L.~Zhang, S.~Perera,
  J.~M. Mari, K.~Chin, T.~Tun, N.~Strouthidis, T.~Aung, A.~Thiéry, and
  M.~Girard}, {\em {DRUNET}: a dilated-residual {U}-net deep learning network
  to segment optic nerve head tissues in optical coherence tomography images},
  Biomedical Optics Express, 9 (2018), pp.~3244--3265,
  \url{https://doi.org/10.1364/BOE.9.003244}.

\bibitem{dong}
{\sc W.~Dong, P.~Wang, W.~Yin, G.~Shi, F.~Wu, and X.~Lu}, {\em Denoising prior
  driven deep neural network for image restoration}, IEEE Transactions on
  Pattern Analysis and Machine Intelligence, 41 (2019), pp.~2305--2318,
  \url{https://doi.org/10.1109/TPAMI.2018.2873610}.

\bibitem{DR_original}
{\sc J.~Douglas, Jr. and H.~H. Rachford, Jr.}, {\em On the numerical solution
  of heat conduction problems in two and three space variables}, Trans. Amer.
  Math. Soc., 82 (1956), pp.~421--439, \url{https://doi.org/10.2307/1993056}.

\bibitem{ducotterd2024}
{\sc S.~Ducotterd, A.~Goujon, P.~Bohra, D.~Perdios, S.~Neumayer, and M.~Unser},
  {\em Improving {L}ipschitz-constrained neural networks by learning activation
  functions}, Journal of Machine Learning Research, 25 (2024), pp.~1--30,
  \url{http://jmlr.org/papers/v25/22-1347.html}.

\bibitem{duval}
{\sc V.~Duval}, {\em {Faces and extreme points of convex sets for the
  resolution of inverse problems}}, habilitation {\`a} diriger des recherches,
  {Ecole doctorale SDOSE}, 2022.

\bibitem{EcksteinBertsekas_DR}
{\sc J.~Eckstein and D.~P. Bertsekas}, {\em On the {D}ouglas-{R}achford
  splitting method and the proximal point algorithm for maximal monotone
  operators}, Mathematical Programming, 55 (1992), pp.~293--318,
  \url{https://doi.org/10.1007/BF01581204}.

\bibitem{Eckstein_ADMM}
{\sc J.~Eckstein and W.~Yao}, {\em Understanding the convergence of the
  alternating direction method of multipliers: theoretical and computational
  perspectives}, Pac. J. Optim., 11 (2015), pp.~619--644,
  \url{http://manu71.magtech.com.cn/Jwk3_pjo/EN/Y2015/V11/I4/619}.

\bibitem{edelsbrunner1987}
{\sc H.~Edelsbrunner}, {\em Algorithms in combinatorial geometry}, vol.~10,
  Springer Science \& Business Media, 1987.

\bibitem{fortune}
{\sc S.~Fortune}, {\em Vorono\u{\i} diagrams and {D}elaunay triangulations}, in
  Computing in {E}uclidean geometry, vol.~1 of Lecture Notes Ser. Comput.,
  World Sci. Publ., River Edge, NJ, 1992, pp.~193--233,
  \url{https://doi.org/10.1142/9789814355858\_0006}.

\bibitem{fwa}
{\sc M.~Frank and P.~Wolfe}, {\em An algorithm for quadratic programming},
  Naval Research Logistics Quarterly, 3 (1956), pp.~95--110,
  \url{https://doi.org/10.1002/nav.3800030109}.

\bibitem{Gabay_FB}
{\sc D.~Gabay}, {\em Chapter {IX} applications of the method of multipliers to
  variational inequalities}, in Augmented Lagrangian Methods: Applications to
  the Numerical Solution of Boundary-Value Problems, vol.~15, Elsevier, 1983,
  pp.~299--331, \url{https://doi.org/10.1016/S0168-2024(08)70034-1}.

\bibitem{ggar}
{\sc G.~Garrigos, L.~Rosasco, and S.~Villa}, {\em Convergence of the
  forward-backward algorithm: beyond the worst-case with the help of geometry},
  Mathematical Programming, 198 (2023), pp.~937--996,
  \url{https://doi.org/10.1007/s10107-022-01809-4}.

\bibitem{Godefroy03}
{\sc G.~Godefroy and N.~J. Kalton}, {\em Lipschitz-free {B}anach spaces},
  Studia Math., 159 (2003), pp.~121--141, \url{http://eudml.org/doc/285365}.

\bibitem{Goldstein2015}
{\sc T.~Goldstein, M.~Li, and X.~Yuan}, {\em Adaptive primal-dual splitting
  methods for statistical learning and image processing}, in Advances in Neural
  Information Processing Systems, vol.~28, 2015,
  \url{https://proceedings.neurips.cc/paper_files/paper/2015/file/cd758e8f59dfdf06a852adad277986ca-Paper.pdf}.

\bibitem{Guerrero2018}
{\sc J.~B. Guerrero, G.~López-Pérez, and A.~Rueda~Zoca}, {\em {Octahedrality
  in Lipschitz-free Banach spaces}}, Proceedings of the Royal Society of
  Edinburgh: Section A Mathematics, 148 (2018), pp.~447--–460,
  \url{https://doi.org/10.1017/S0308210517000373}.

\bibitem{Hanin92}
{\sc L.~G. Hanin}, {\em {Kantorovich-Rubinstein} norm and its application in
  the theory of {L}ipschitz spaces}, Proc. Amer. Math. Soc., 115 (1992),
  pp.~345--352, \url{https://doi.org/10.2307/2159251}.

\bibitem{hannu2014}
{\sc A.~Hannukainen, S.~Korotov, and M.~Křížek}, {\em On numerical
  regularity of the face-to-face longest-edge bisection algorithm for
  tetrahedral partitions}, Science of Computer Programming, 90 (2014),
  pp.~34--41, \url{https://doi.org/10.1016/j.scico.2013.05.002}.

\bibitem{he}
{\sc J.~He, Y.~Yang, Y.~Wang, D.~Zeng, Z.~Bian, H.~Zhang, J.~Sun, Z.~Xu, and
  J.~Ma}, {\em Optimizing a parameterized {Plug-and-Play} {ADMM} for iterative
  low-dose {CT} reconstruction}, IEEE Transactions on Medical Imaging, 38
  (2019), pp.~371--382, \url{https://doi.org/10.1109/TMI.2018.2865202}.

\bibitem{heide}
{\sc F.~Heide, M.~Steinberger, Y.-T. Tsai, M.~Rouf, D.~Pajkak, D.~Reddy,
  O.~Gallo, J.~Liu, W.~Heidrich, K.~Egiazarian, J.~Kautz, and K.~Pulli}, {\em
  {FlexISP}: A flexible camera image processing framework}, ACM Trans. Graph.,
  33 (2014), pp.~231:1--231:13, \url{https://doi.org/10.1145/2661229.2661260}.

\bibitem{Heinonen2005}
{\sc J.~Heinonen}, {\em Lectures on {L}ipschitz analysis}, tech. report,
  University of Jyväskylä, Jyväskylä, 2005.

\bibitem{steidl2021}
{\sc J.~Hertrich, S.~Neumayer, and G.~Steidl}, {\em Convolutional proximal
  neural networks and {Plug-and-Play} algorithms}, Linear Algebra and its
  Applications, 631 (2021), pp.~203--234,
  \url{https://doi.org/10.1016/j.laa.2021.09.004}.

\bibitem{hurault22}
{\sc S.~Hurault, A.~Leclaire, and N.~Papadakis}, {\em Proximal denoiser for
  convergent {Plug-and-Play} optimization with nonconvex regularization}, in
  Proceedings of the 39th International Conference on Machine Learning,
  vol.~162, 2022, pp.~9483--9505,
  \url{https://proceedings.mlr.press/v162/hurault22a.html}.

\bibitem{jaggi2013}
{\sc M.~Jaggi}, {\em Revisiting {Frank--Wolfe}: Projection-free sparse convex
  optimization}, in Proceedings of the 30th International Conference on Machine
  Learning, vol.~28(1), 2013, pp.~427--435,
  \url{https://proceedings.mlr.press/v28/jaggi13.html}.

\bibitem{pnpreview}
{\sc U.~S. Kamilov, C.~A. Bouman, G.~T. Buzzard, and B.~Wohlberg}, {\em
  {Plug-and-Play} methods for integrating physical and learned models in
  computational imaging: Theory, algorithms, and applications}, IEEE Signal
  Processing Magazine, 40 (2023), pp.~85--97,
  \url{https://doi.org/10.1109/MSP.2022.3199595}.

\bibitem{Kantorovich82}
{\sc L.~V. Kantorovich and G.~P. Akilov}, {\em Functional analysis}, Pergamon
  Press, 2nd~ed., 1982.

\bibitem{kirz}
{\sc M.~Kirszbraun}, {\em Über die zusammenziehende und lipschitzsche
  {T}ransformationen}, Fundamenta Mathematicae, 22 (1934), pp.~77--108,
  \url{http://eudml.org/doc/212681}.

\bibitem{krizek}
{\sc M.~Křížek and T.~Strouboulis}, {\em How to generate local refinements
  of unstructured tetrahedral meshes satisfying a regularity ball condition},
  Numerical Methods for Partial Differential Equations, 13 (1997),
  pp.~201--214,
  \url{https://doi.org/10.1002/(SICI)1098-2426(199703)13:2<201::AID-NUM5>3.0.CO;2-T}.

\bibitem{Lanthaler2024}
{\sc S.~Lanthaler}, {\em Operator learning of {L}ipschitz operators: An
  information-theoretic perspective}, arXiv preprint arXiv:2406.18794,  (2024).

\bibitem{DR_monotone}
{\sc P.~Lions and B.~Mercier}, {\em Splitting algorithms for the sum of two
  nonlinear operators}, SIAM Journal on Numerical Analysis, 16 (1979),
  pp.~964--979, \url{https://doi.org/10.1137/0716071}.

\bibitem{Lorenz2024}
{\sc D.~A. Lorenz, J.~Marquardt, and E.~Naldi}, {\em The degenerate variable
  metric proximal point algorithm and adaptive stepsizes for primal-dual
  {D}ouglas–{R}achford}, Optimization, 74 (2024),
  \url{https://doi.org/10.1080/02331934.2024.2325552}.

\bibitem{Lorenz2019}
{\sc D.~A. Lorenz and Q.~Tran-Dinh}, {\em Non-stationary {D}ouglas–{R}achford
  and alternating direction method of multipliers: adaptive step-sizes and
  convergence}, Computational Optimization and Applications, 74 (2019),
  pp.~67--92, \url{https://doi.org/10.1007/s10589-019-00106-9}.

\bibitem{megg}
{\sc R.~E. Megginson}, {\em An introduction to {B}anach space theory}, vol.~183
  of Graduate Texts in Mathematics, Springer-Verlag, New York, 1998,
  \url{https://doi.org/10.1007/978-1-4612-0603-3}.

\bibitem{mein}
{\sc T.~Meinhardt, M.~Moeller, C.~Hazirbas, and D.~Cremers}, {\em Learning
  proximal operators: Using denoising networks for regularizing inverse imaging
  problems}, in Proceedings of the IEEE International Conference on Computer
  Vision, 2017, pp.~1799--1808, \url{https://doi.org/10.1109/ICCV.2017.198}.

\bibitem{Neumayer2023}
{\sc S.~Neumayer, A.~Goujon, P.~Bohra, and M.~Unser}, {\em Approximation of
  {L}ipschitz functions using deep spline neural networks}, SIAM Journal on
  Mathematics of Data Science, 5 (2023), pp.~306--322,
  \url{https://doi.org/10.1137/22M1504573}.

\bibitem{OconnorVandenberghe}
{\sc D.~O'Connor and L.~Vandenberghe}, {\em Primal-dual decomposition by
  operator splitting and applications to image deblurring}, SIAM Journal on
  Imaging Sciences, 7 (2014), pp.~1724--1754,
  \url{https://doi.org/10.1137/13094671X}.

\bibitem{pescterr2021}
{\sc J.-C. Pesquet, A.~Repetti, M.~Terris, and Y.~Wiaux}, {\em Learning
  maximally monotone operators for image recovery}, SIAM Journal on Imaging
  Sciences, 14 (2021), pp.~1206--1237,
  \url{https://doi.org/10.1137/20M1387961}.

\bibitem{PolyanskiyWu}
{\sc Y.~Polyanskiy and Y.~Wu}, {\em Wasserstein continuity of entropy and outer
  bounds for interference channels}, IEEE Transactions on Information Theory,
  62 (2016), pp.~3992--4002, \url{https://doi.org/10.1109/TIT.2016.2562630},
  \url{https://doi.org/10.1109/TIT.2016.2562630}.

\bibitem{Quarteroni2017}
{\sc A.~Quarteroni}, {\em Numerical models for differential problems}, vol.~16
  of MS\&A. Modeling, Simulation and Applications, Springer, Cham, third~ed.,
  2017, \url{https://doi.org/10.1007/978-3-319-49316-9}.

\bibitem{Raginsky}
{\sc M.~Raginsky, A.~Rakhlin, and M.~Telgarsky}, {\em Non-convex learning via
  stochastic gradient {L}angevin dynamics: a nonasymptotic analysis}, in
  Proceedings of the 2017 Conference on Learning Theory, vol.~65 of Proceedings
  of Machine Learning Research, PMLR, 2017, pp.~1674--1703,
  \url{https://proceedings.mlr.press/v65/raginsky17a.html}.

\bibitem{rof}
{\sc L.~I. Rudin, S.~Osher, and E.~Fatemi}, {\em Nonlinear total variation
  based noise removal algorithms}, Physica D. Nonlinear Phenomena, 60 (1992),
  pp.~259--268, \url{https://doi.org/10.1016/0167-2789(92)90242-F}.

\bibitem{ryupnp2019}
{\sc E.~Ryu, J.~Liu, S.~Wang, X.~Chen, Z.~Wang, and W.~Yin}, {\em
  {Plug-and-Play} methods provably converge with properly trained denoisers},
  in Proceedings of the 36th International Conference on Machine Learning,
  vol.~97 of Proceedings of Machine Learning Research, 2019, pp.~5546--5557,
  \url{https://proceedings.mlr.press/v97/ryu19a.html}.

\bibitem{Santos2024}
{\sc C.~Santos~Garcia, M.~Larchevêque, S.~O’Sullivan, M.~Van~Waerebeke,
  R.~R. Thomson, A.~Repetti, and J.-C. Pesquet}, {\em A primal–dual
  data-driven method for computational optical imaging with a photonic
  lantern}, PNAS Nexus, 3 (2024), p.~pgae164,
  \url{https://doi.org/10.1093/pnasnexus/pgae164}.

\bibitem{scher}
{\sc O.~Scherzer, M.~Grasmair, H.~Grossauer, M.~Haltmeier, and F.~Lenzen}, {\em
  Variational methods in imaging}, vol.~167 of Applied Mathematical Sciences,
  Springer, New York, 2009, \url{https://doi.org/10.1007/978-0-387-69277-7}.

\bibitem{celledoni23}
{\sc F.~Sherry, E.~Celledoni, M.~J. Ehrhardt, D.~Murari, B.~Owren, and C.-B.
  Sch{\"o}nlieb}, {\em Designing stable neural networks using convex analysis
  and {ODE}s}, Physica D: Nonlinear Phenomena, 463 (2024), p.~134159,
  \url{https://doi.org/10.1016/j.physd.2024.134159}.

\bibitem{sree}
{\sc S.~Sreehari, S.~V. Venkatakrishnan, B.~Wohlberg, G.~T. Buzzard, L.~F.
  Drummy, J.~P. Simmons, and C.~A. Bouman}, {\em {Plug-and-Play} priors for
  bright field electron tomography and sparse interpolation}, IEEE Transactions
  on Computational Imaging, 2 (2016), pp.~408--423,
  \url{https://doi.org/10.1109/tci.2016.2599778}.

\bibitem{Sun2021}
{\sc Y.~Sun, Z.~Wu, X.~Xu, B.~Wohlberg, and U.~S. Kamilov}, {\em Scalable
  plug-and-play admm with convergence guarantees}, IEEE Transactions on
  Computational Imaging, 7 (2021), pp.~849--863,
  \url{https://doi.org/10.1109/TCI.2021.3094062}.

\bibitem{Suzuki2024}
{\sc Y.~Suzuki, R.~Isono, and S.~Ono}, {\em A convergent primal-dual deep
  {P}lug-and-{P}lay algorithm for constrained image restoration}, in ICASSP
  2024 - 2024 IEEE International Conference on Acoustics, Speech and Signal
  Processing (ICASSP), 2024, pp.~9541--9545,
  \url{https://doi.org/10.1109/ICASSP48485.2024.10448023}.

\bibitem{SVAITER2011}
{\sc B.~F. Svaiter}, {\em On weak convergence of the {D}ouglas-{R}achford
  method}, SIAM J. Control Optim., 49 (2011), pp.~280--287,
  \url{https://doi.org/10.1137/100788100}.

\bibitem{Tanielian2021}
{\sc U.~Tanielian and G.~Biau}, {\em Approximating {L}ipschitz continuous
  functions with groupsort neural networks}, in International Conference on
  Artificial Intelligence and Statistics, 2021, pp.~442--450,
  \url{https://proceedings.mlr.press/v130/tanielian21a.html}.

\bibitem{teod}
{\sc A.~M. Teodoro, J.~M. Bioucas-Dias, and M.~A.~T. Figueiredo}, {\em Image
  restoration and reconstruction using variable splitting and class-adapted
  image priors}, in IEEE International Conference on Image Processing, 2016,
  pp.~3518--3522, \url{https://doi.org/10.1109/ICIP.2016.7533014}.

\bibitem{teod17}
{\sc A.~M. Teodoro, J.~M. Bioucas-Dias, and M.~A.~T. Figueiredo}, {\em
  Scene-adapted {Plug-and-Play} algorithm with convergence guarantees}, in IEEE
  International Workshop on Machine Learning for Signal Processing, 2017,
  pp.~1--6, \url{https://doi.org/10.1109/MLSP.2017.8168194}.

\bibitem{Terris2020}
{\sc M.~Terris, A.~Repetti, J.-C. Pesquet, and Y.~Wiaux}, {\em Building firmly
  nonexpansive convolutional neural networks}, in ICASSP 2020 - 2020 IEEE
  International Conference on Acoustics, Speech and Signal Processing (ICASSP),
  2020, pp.~8658--8662, \url{https://doi.org/10.1109/ICASSP40776.2020.9054731}.

\bibitem{Tseng_FB}
{\sc P.~Tseng}, {\em Applications of a splitting algorithm to decomposition in
  convex programming and variational inequalities}, SIAM Journal on Control and
  Optimization, 29 (1991), pp.~119--138, \url{https://doi.org/10.1137/0329006}.

\bibitem{Varadarajan}
{\sc V.~S. Varadarajan}, {\em On the convergence of sample probability
  distributions}, Sankhy{\=a}: The Indian Journal of Statistics (1933--1960),
  19 (1958), pp.~23--26, \url{https://www.jstor.org/stable/25048365}.

\bibitem{venkat2013}
{\sc S.~V. Venkatakrishnan, C.~A. Bouman, and B.~Wohlberg}, {\em
  {Plug-and-Play} priors for model based reconstruction}, in 2013 IEEE Global
  Conference on Signal and Information Processing, 2013, pp.~945--948,
  \url{https://doi.org/10.1109/GlobalSIP.2013.6737048}.

\bibitem{Villani2009}
{\sc C.~Villani}, {\em Optimal Transport: Old and New}, vol.~338 of Grundlehren
  der mathematischen Wissenschaften, Springer, Berlin, Heidelberg, 2009,
  \url{https://doi.org/10.1007/978-3-540-71050-9}.

\bibitem{Weaver}
{\sc N.~Weaver}, {\em Lipschitz algebras}, World Scientific, second~ed., 2018.

\bibitem{Zaichyk2023}
{\sc H.~Zaichyk, A.~Biess, A.~Kontorovich, and Y.~Makarychev}, {\em Efficient
  {K}irszbraun extension with applications to regression}, Mathematical
  Programming, 207 (2023), pp.~617--642,
  \url{https://doi.org/10.1007/s10107-023-02023-6}.

\bibitem{zhang2022}
{\sc B.~Zhang, D.~Jiang, D.~He, and L.~Wang}, {\em Rethinking {L}ipschitz
  neural networks and certified robustness: A {B}oolean function perspective},
  Advances in neural information processing systems, 35 (2022),
  pp.~19398--19413,
  \url{https://proceedings.neurips.cc/paper_files/paper/2022/file/7b04ec5f2b89d7f601382c422dfe07af-Paper-Conference.pdf}.

\bibitem{zhangIEEE}
{\sc K.~Zhang, W.~Zuo, Y.~Chen, D.~Meng, and L.~Zhang}, {\em Beyond a
  {G}aussian denoiser: Residual learning of deep {CNN} for image denoising},
  IEEE Transactions on Image Processing, 26 (2017), pp.~3142--3155,
  \url{https://doi.org/10.1109/TIP.2017.2662206}.

\bibitem{zhang}
{\sc K.~Zhang, W.~Zuo, S.~Gu, and L.~Zhang}, {\em Learning deep {CNN} denoiser
  prior for image restoration}, in 2017 IEEE Conference on Computer Vision and
  Pattern Recognition (CVPR), 2017, pp.~2808--2817,
  \url{https://doi.org/10.1109/CVPR.2017.300}.

\end{thebibliography}
